\documentclass[11pt]{amsart}

\usepackage{mathrsfs}
\usepackage{amsmath, amscd, amsthm,amssymb, amsfonts, verbatim,subfigure,slashed}
\usepackage[mathcal]{eucal}

\usepackage[all]{xy}

\usepackage{mathpazo}

\newcommand{\dirac}{\slashed{\partial}}
\newcommand{\zbar}{\br{z}}
\newcommand{\GL}{\op{GL}}

\newcommand{\dpa}[1]{\frac{\partial}{\partial #1}}

\newcommand{\Obs}{\op{Obs}}

\newcommand{\eps}{\varepsilon}
\newcommand{\g}{\mathfrak{g}}

\newcommand{\GF}{Q^{GF}}
\newcommand{\xto}{\xrightarrow}
\newcommand{\E}{\mscr{E}}

\newcommand{\A}{\mscr A}

\newcommand{\what}{\widehat}
\newcommand{\tr}{\triangle}

\newcommand{\til}{\widetilde}
\newcommand{\mscr}{\mathscr}
\renewcommand{\det}{\operatorname{det}}

\newcommand{\br}{\overline}

\newcommand{\iso}{\cong}
\newcommand{\C}{\mathbb C}

\newcommand{\Oo}{\mscr O}
\newcommand{\Z}{\mathbb Z}

\newcommand{\into}{\hookrightarrow}

\newcommand{\Gr}{\operatorname {Gr}}
\newcommand{\op}{\operatorname}
\newcommand{\mbf}{\mathbf}
\newcommand{\mbb}{\mathbb}
\newcommand{\mf}{\mathfrak}
\newcommand{\mc}{\mathcal}

\newcommand{\ip}[1]{\left\langle #1 \right\rangle}
\newcommand{\abs}[1]{\left| #1 \right|}

\newcommand{\R}{\mbb R}
\renewcommand{\d}{\mathrm{d}}
\newcommand{\liminv}{ \varprojlim }

\newcommand{\dbar}{\br{\partial}}

 \DeclareMathOperator{\End}{End}
 
\DeclareMathOperator{\Sym}{Sym} \DeclareMathOperator{\Hom}{Hom}
\DeclareMathOperator{\Spec}{Spec}

\newtheoremstyle{thm}
  {7pt}
  {7pt}
  {\itshape}
  {}
  {\bf}
  {.}
  {5pt}
  {\thmnumber{#2 }\thmname{#1}\thmnote{ (#3)}}

\newtheoremstyle{def}
  {7pt}
  {10pt}
  {\itshape}
  {}
  {\bf}
  {.}
  {5pt}
  {\thmnumber{#2} \thmname{#1}\thmnote{ (#3)}}

\newtheoremstyle{rem}
  {4pt}
  {7pt}
  {}
  {}
  {\itshape}
  {:}
  {3pt}
  {}

\newtheoremstyle{texttheorem}
  {8pt}
  {8pt}
  {\itshape}
  {}
  {\bf}
  {. \hspace{5pt}}
  {3pt}
  {}

\theoremstyle{thm}

\newtheorem*{theorem*}{Theorem}

\newtheorem*{corollary*}{Corollary}

\newtheorem{theorem}{Theorem}[subsection]
\newtheorem{thm-def}{Theorem/Definition}[theorem]
\newtheorem{proposition}[theorem]{Proposition}

\newtheorem{lemma}[theorem]{Lemma}

\newtheorem{corollary}[theorem]{Corollary}

\numberwithin{equation}{subsection}

\theoremstyle{def}
\newtheorem{definition}[theorem]{Definition}

\theoremstyle{rem}

\newtheorem*{remark}{Remark}

\theoremstyle{texttheorem}

\newcommand{\cinfty}{C^{\infty}}

\date{}

\renewcommand{\L}{\mscr{L}}
\newcommand{\Crit}{\op{Crit}}

\theoremstyle{thm}

\begin{document}
\title{Notes on supersymmetric and holomorphic field theories in dimensions $2$ and $4$}
\author{Kevin Costello}
\thanks{Partially supported by NSF grants DMS 0706945 and DMS 1007168, and by a Sloan fellowship.}
\address{Department of Mathematics, Northwestern University.}
\email{costello@math.northwestern.edu}
\begin{abstract}
These notes explore some aspects of formal derived geometry related to classical field theory.   One goal is to explain how many important classical field theories in physics -- such as supersymmetric gauge theories and supersymmetric $\sigma$-models -- can be described very cleanly using derived geometry.   In particular, I describe a mathematically natural construction of Kapustin-Witten's $\mbb{P}^1$ of twisted supersymmetric gauge theories.  
\end{abstract}
\dedicatory{\it To Dennis Sullivan on the occassion of his 70th birthday}
\maketitle
\subsection*{Elliptic moduli problems}
Moduli spaces of solutions to systems of elliptic equations (such as the Yang-Mills instanton equations, self-duality equations, holomorphic map equations, etc.) have played a central role in mathematics for many years.  The first aim of this paper is to develop a general homological language for discussing formal derived moduli problems of solutions to elliptic differential equations.    I call such an object an \emph{elliptic moduli problem}.

The equations of motion of a classical field theory are a system of elliptic differential equations, and so the formal moduli space of their solutions (infinitesimally near a given solution) is an elliptic moduli problem.    The fact that the moduli of solutions to the equations of motion of a classical field theory are  the critical points of an action functional mean that this elliptic moduli problem is equipped with an additional geometric structure: a symplectic form of cohomological degree $-1$.  Following a suggestion of Lurie, I will call a space with a degree $-1$ symplectic form \emph{$0$-symplectic}. \footnote{The reason for this terminology is that there is a close relationship between spaces with a symplectic form of cohomological degree $k$ and the $E_{k+1}$-operad.}

 This elliptic moduli problem, with its symplectic form, is a complete encoding of the classical field theory.      In this paper, we will define a classical field theory to be a $0$-symplectic elliptic moduli problem. 

I will only consider \emph{formal} derived spaces.  Even giving a good definition of symplectic form on a global derived stack is a highly non-trivial matter: the general theory of such objects is worked out in \cite{Vez11} and \cite{PanToeVaq11}. 

In ordinary geometry, the simplest construction of a symplectic manifold is as a cotangent bundle.  There is a similar construction in our context: given any elliptic moduli problem (corresponding to a system of elliptic differential equations) there is a corresponding classical field theory, which we call the \emph{cotangent field theory}.    Many field theories of interest in mathematics and physics arise as cotangent theories.

\subsection*{Quantization}
After setting up this language for discussing classical field theories,  I will briefly discuss what it means to quantize a classical field theory, following \cite{Cos11} and \cite{CosGwi11}.   I will show that a quantization of a cotangent theory to an elliptic moduli problem $\mc M$ on a compact manifold $X$ leads to a volume form\footnote{defined up to an overall constant} on the finite-dimensional formal derived space $\mc M(X)$ of global solutions.  This leads to a program for defining (and computing) the non-perturbative partition function for a cotangent theory: a quantization of the theory yields a volume form on the space of solutions, and the partition function is the volume.    

This program has been carried out successfully in \cite{Cos11a} for the cotangent theory associated to the moduli space of degree zero maps form an elliptic curve to a compact complex manifold $X$.  In this case, the partition function is the Witten genus of $X$.  

More details on this program, and on further examples, will appear in subsequent publications.  

\subsection*{Supersymmetry}
Much of the rest of the paper is devoted to studying examples of classical field theories using this language.  The hope is to convince mathematicians that the framework of derived geometry provides a very natural way to understand supersymmetric field theories (or at least, their holomorphic and topological twists).  

A more concrete goal is to give a mathematically natural construction of the $\mbb{P}^1$ of twisted $\mscr{N}=4$ supersymmetric gauge theories constructed by Kapustin and Witten \cite{KapWit06} in their study of the geometric Langlands program; and to explain how this $\mbb{P}^1$ of field theories is related to the $A$- and $B$-models with target the Hitchin system.

Geometrically, the $\mbb{P}^1$ of twisted $\mscr{N}=4$ gauge theories is given by a family of elliptic moduli problem with a symplectic form of degree $-1$.    At a special point (the $B$-model point) on $\mbb{P}^1$, this is the cotangent theory associated to the elliptic moduli problem of $G$-local systems on a complex surface $S$.  At generic points of $\mbb{P}^1$, the elliptic moduli problem can be interpreted as the de Rham stack of the moduli stack of $G$-local systems on $X$, equipped with a certain symplectic form.  At the $A$-model point, the elliptic moduli problem becomes the de Rham stack of the derived moduli space of Higgs bundles on $X$.

This $\mbb{P}^1$ of theories is expected to become, upon reduction to two dimensions, the $A$- and $B$-models of mirror symmetry with target moduli spaces of Higgs bundles and local systems. In order to understand this, I give a definition of the $A$- and $B$-models, as well as their half-twisted versions.  It turns out that the half-twisted $A$- and $B$-models are particularly easy to describe in the language used in this paper: they are both cotangent theories to natural elliptic moduli problems.  

Then, after briefly discussing the basics of supersymmetry in $4$-dimensions, I describe the self-dual limits of the $\mscr{N}=1,2$ and $4$ supersymmetric gauge theories on $\R^4$, using the twistor-space formalism developed in detail in \cite{BoeMasSki07} (building on earlier work of Witten \cite{Wit04}).    Again, these theories are cotangent theories to natural elliptic moduli problems.

Next, the concept of \emph{twisting} of a supersymmetric gauge theory is introduced.  We will see that the twisting procedure has many features that will be familiar to homological algebraists: for example, there's a spectral sequence starting from the observables of the physical (untwisted) theory and converging to the observables of the twisted theory.  

Once we have defined what it means to twist a supersymmetric field theory, we will analyze twists of the $\mscr{N}=1,2$ and $4$ supersymmetric gauge theories.  We will start by analyzing ``minimal'' twists, which are holomorphic theories in $4$ dimensions which can be further twisted to yield the more familiar topological twists. We will find, again, that the minimally twisted theories are cotangent theories to simple elliptic moduli problems.  The minimally twisted $\mscr{N}=1$ theory on a complex surface $X$ is the cotangent theory associated to the moduli problem of holomorphic $G$-bundles on $X$; the minimally twisted $\mscr{N}=2$ theory is the cotangent theory to the moduli space of such bundles equipped with a section of the adjoint bundle of Lie algebras; and  the twisted $\mscr{N}=4$ theory is the cotangent theory to the moduli of Higgs bundles on $X$.  

The $\mscr{N}=1$ theory can be twisted only once; but the $\mscr{N}=2$ and $4$ theories admit further twists.  In the final section of the paper we will show that the $\mscr{N}=4$ theory admits a further $\mbb{P}^1$ of twists (discussed above), which (at special points in $\mbb{P}^1$) dimensionally reduce to the $A$- and $B$-models with target the Hitchin system, as expected. 

I should say what I mean by dimensional reduction.  As we have seen, a classical field theory on $X$ is a sheaf $\mc M$ of derived spaces on $X$.  Dimensional reduction simply means pushing forward this sheaf along a map $X \to Y$.   The precise relationship between the twisted $\mscr{N}=4$ gauge theory and the $A$- and $B$-models is that if we take our gauge theory on a product $\Sigma_1 \times \Sigma_2$ of two Riemann surfaces, and dimensionally reduce in this sense along the map $\Sigma_1 \times \Sigma_2 \to \Sigma_1$, we find a field theory on $\Sigma_1$ which is equivalent to the $A$-model with target $T^\ast \op{Bun}_G(\Sigma_2)$ (if we use the $A$-twisted $\mscr{N}=4$ gauge theory), or the $B$-model with target $\op{Loc}_G(\Sigma_2)$ (if we use the $B$-twisted $\mscr{N}=4$ gauge theory).

One advantage of the point of view advocated here is that one does not need to know anything about supersymmetry to understand the $\mbb{P}^1$ of twisted $\mscr{N}=4$ theories I describe.   The only reason I discuss supersymmetry in this paper is to justify the assertion that the field theories I describe using derived geometry are the same as those discussed in the physics literature.   

\subsection*{Quantizing twisted supersymmetric gauge theories} So far, we have discussed supersymmetric gauge theories only at the classical level.  I will say almost nothing about quantization: except to prove that the minimally-twisted $\mscr{N}=1,2,4$ theories we consider admit a unique quantization on $\C^2$, compatible with all natural symmetries.   A twist of a deformation of the $\mscr{N}=1$ theory is analyzed in detail in \cite{Cos13}, where it is shown that the factorization algebra associated to this theory at the quantum level is closely related to the Yangian. 

\subsection*{Warning.} The main objects of study in this paper are certain formal derived stacks, equipped with extra geometrical structures (e.g.\ a symplectic form).  I will give all details for how to construct and work with these objects at the formal level.   Even though I don't supply all the details required for my constructions at the non-formal level, I will often (informally) talk about global derived stacks.  

\subsection*{Terminology.}  My use of the term ``twist'' (as in, twisted supersymmetric field theory) differs a little from the way some physicists use this term.  In this paper, given a supersymmetric field theory, the twisted theory is obtained by considering only quantities invariant under a particular supercharge.  (A supercharge is physics-speak for an odd element of the $\Z/2$ graded Lie algebra of supersymmetries, which acts on a supersymmetric field theory).  

For some authors, to twist a supersymmetric field theory is a two-step process.  Given a supersymmetric field theory on $\R^4$, one first changes the action of the Poincar\'e group on the theory by choosing a map from the Poincar\'e group into the $R$-symmetry group.  (The $R$-symmetry group is a Lie group acting on a supersymmetric field theory in a way lifting the trivial action on space-time).  Then, one finds a supercharge $Q$ invariant under this new action of the Poincar\'e group, and twists (in my sense) with respect to $Q$.

\subsection*{Acknowledgements}
First and foremost, I'd like to thank Dennis Sullivan for the many inspiring conversations we've had over the years.   When I met Dennis at a workshop in 2004, I was a first year postdoc and he was, of course, famous; nevertheless, we hit it off and spent much of the workshop talking.  His enthusiasm almost made me forget that we were in Trenton, New Jersey.  

Since then, I've been lucky to be able to spend a lot of time with Dennis, at his seminars in New York and at other conferences.  Our discussions about how one can use homological algebra to understand physics have had a profound influence on my thinking. 

Most mathematicians work in pretty much the same field throughout their career. Dennis is completely different: he will tackle any problem in any field that grips him, and work on it until he solves it. Though my background is in algebraic geometry, I've always wanted to understand quantum field theory and string theory.  Dennis's example encouraged me to think that it wouldn't be completely crazy to work on the foundations of QFT. 

I'm also very grateful to Owen Gwilliam for many helpful conversations about field theory in general, and cotangent field theories in particular; to Lionel Mason, for explaining to me the twistor-space approach to supersymmetry; to Anton Kapustin and Si Li, for illuminating conversations about  supersymmetry; to Jacob Lurie and Nick Rozenblyum, for patiently answering my questions about Grothendieck duality in derived algebraic geometry; to Davesh Maulik and Gabriele Vezzosi, for helpful conversations about derived symplectic geometry; and to Josh Shadlen, for helpful discussions concerning the BD operads and for his comments on the text. I'm also grateful to the anonymous referees for some very helpful suggestions.

\tableofcontents

\addcontentsline{toc}{part}{Elliptic moduli problems}
\section{Elliptic moduli problems}

Let us thus start by trying to give a general definition of an elliptic moduli space.   We will use a little loosely the basic ideas of derived geometry, as developed in \cite{Lur09b,Toe06}.  

\subsection{}
The following statement is at the heart of the philosophy of deformation theory: 
\begin{quote}
There is an equivalence of $(\infty,1)$ categories between the category of differential graded Lie algebras, and the category of formal pointed derived moduli problems. 
\end{quote}
In a different guise, this statement goes back to Quillen's \cite{Qui69} and Sullivan's \cite{Sul77} work on rational homotopy theory.   These ideas were developed extensively in the work of Kontsevich and Soibelman \cite{Kon97,KonSoi}.   More general theorems of this nature are considered in \cite{Lur10}, which is also an excellent survey of these ideas. 

The basic idea of this correspondence is as follows.  Formal moduli problems are defined using the functor of points.  Thus, a formal moduli problem is a functor $F$ which takes a nilpotent Artinian differential graded algebra $R$, and assigns to it the simplicial set $F(R)$ of $R$-points of the moduli problem. If $\g$ is a differential graded Lie algebra, then the formal moduli problem $F_{\g}$ associated to $\g$ is defined as follows.  If $R$ is a nilpotent Artinian dga, with maximal ideal $m \subset R$, then we set
$$
F_{\g}(R) = \op{MC}(\g \otimes m) 
$$
where $\op{MC}(\g \otimes m)$ is the simplicial set of Maurer-Cartan elements of the dg  Lie algebra $\g \otimes m$.  

We are interested in elliptic derived moduli problems: that is, derived moduli problems described by a system of elliptic partial differential equations on a manifold $M$.   As a first step towards a formal definition of an elliptic derived moduli problem, we will give a definition of formal pointed elliptic moduli problem.  Using the principle quoted above as a guide, we will define an formal pointed elliptic moduli problem on a manifold $M$ to be a sheaf of $L_\infty$ algebras on $M$ of a certain kind. 
\begin{definition}
Let $M$ be a manifold.  An elliptic $L_\infty$ algebra on $M$ consists of the following data.
\begin{enumerate}
\item
A graded vector bundle $L$ on $M$, whose space of sections will be denoted $\L$. 
\item A differential operator $\d : \L \to \L$, of cohomological degree $1$ and square $0$, which makes $\L$ into an elliptic complex.
\item  A collection of poly-differential operators
$$
l_n : \L^{\otimes n} \to \L
$$
which are alternating, of cohomological degree $2-n$, and which endow $\L$ with the structure of $L_\infty$ algebra.
\end{enumerate}
\end{definition}
Throughout this paper, formal pointed elliptic moduli problems will be described by elliptic $L_\infty$ algebras. 

If $\L$ is an elliptic $L_\infty$ algebra on a manifold $M$, then it yields a presheaf on $M$ of functors from dg Artin rings to simplicial sets.  If $(R,m)$ is a dg Artin ring with maximal ideal $R$, and if $U \subset M$ is an open subset, then we can consider the simplicial set
$$
\op{MC} ( \L(U) \otimes m)
$$
of Maurer-Cartan elements of the $L_\infty$ algebra $\L(U) \otimes m$ (where $\L(U)$ refers to the sections of $L$ on $U$). We will think of this as the $R$-points of the formal pointed moduli problem associated to $\L(U)$. 

\begin{remark}
When discussing global (i.e.\ non-formal) derived spaces, I will often be quite informal; a thorough treatment of such objects as the derived moduli stack of local systems on a manifold is out of reach of this paper.   I will, however, try to be more precise when talking about formal derived spaces, by giving an explicit description of the corresponding Lie or $L_\infty$ algebra. 
\end{remark}
\section{Examples of elliptic moduli problems}
\subsection{Flat bundles}
The most basic example of an elliptic moduli problem is that associated to flat bundles on a manifold $M$.  Let $G$ be a Lie group, and let $P \to M$ be a principal $G$-bundle equipped with a flat connection.  Let $\mf g_P$ be the adjoint bundle (associated to $P$ by the adjoint action of $G$ on its Lie algebra $\g$).  Thus, $\g_P$ is a bundle of Lie algebras on $M$, with a flat connection.

The elliptic $L_\infty$ algebra controlling deformations of the flat $G$-bundle $P$ is simply
$$
\L = \Omega^\ast(M,\g_P).
$$
The differential on $\L$ is the de Rham differential on $M$ coupled to the flat connection on $\g_P$. 

To see this, observe that any deformation of $P$ just as a $G$-bundle is trivial.  We can, however, deform the flat connection on $P$.  Let $(R,m)$ be a nilpotent Artin ring with maximal ideal $m$.  Then, a family of flat connections on $P$, parametrized by $\Spec R$, is the same as an element
$$
\alpha \in \Omega^1(M, \g_P) \otimes m
$$
satisfying the Maurer-Cartan equation
$$
\d \alpha + \tfrac{1}{2} [\alpha, \alpha ] = 0.
$$
Further, two such Maurer-Cartan elements give the same flat $G$-bundle if and only if they are gauge equivalent.  Gauge equivalences are represented by $1$-simplices in the Maurer-Cartan simplicial set.  Thus, we see that $\pi_0 \op{MC}\left(\Omega^1 ( M , \g_P ) \otimes m \right)$ is the set of isomorphism classes families of flat $G$-bundles over $\Spec R$, which restrict to the given $G$-bundle at the base-point of $\Spec R$. 

One can ask what role the forms $\Omega^i$ for $i > 2$ play in this story.  Of course, if we just probe our moduli problem with ordinary (not dg) Artin rings, we do not detect the higher forms.  However, if $R$ is a differential graded Artin ring, then a Maurer-Cartan element of $\Omega^\ast( M, \g_P) \otimes R$ may have components involving all of the $\Omega^i$.  

\subsection{Self-dual bundles}
Let $M$ be an oriented $4$-manifold.  Let $G$ be a Lie group, and let $P \to M$ be a principal $G$-bundle, and let $\g_P$ be the adjoint bundle of Lie algebras.  Suppose we have a connection $A$ on $P$ with self-dual curvature: 
$$F(A)_- = 0 \in \Omega^2_-(M, \g_P)$$
(here $\Omega^2_-(M)$ denotes the space of anti-self-dual two-forms).

Then, the elliptic Lie algebra controlling deformations of $(P, A)$ is described by the diagram
$$
\Omega^0(M, \g_P ) \xto{\d} \Omega^1(M, \g_P ) \xto{\d_-} \Omega^2_-(M, \g_ P ).
$$
Here $\d_-$ is the composition of the de Rham differential (coupled to the connection on $\g_P$) with the projection onto $\Omega^2_-(M, \g_P)$.

\subsection{Holomorphic bundles}
In a similar way, if $M$ is a complex manifold and if $P \to M$ is a holomorphic principal $G$-bundle, then the elliptic dg Lie algebra $\Omega^{0,\ast}(M, \g_P)$, with differential $\dbar$, describes the formal moduli space of holomorphic $G$-bundles on $M$.
 
\section{Symmetries of elliptic moduli problems}
Suppose that $R$ is a differential graded algebra. Let $R^\sharp$ refer to $R$ without the differential.  
\begin{definition}
An $R$-family of elliptic $L_\infty$ algebras on $X$ consists of graded bundle $L$ of locally-free $R^\sharp$-modules on $X$,  whose sheaf of sections will be denoted $\L$; together with an $R^\sharp$-linear differential operator
$$
\d : \L \to \L
$$
which makes $\L$ into a sheaf of dg $R$-modules.  We require that the complex $(\L, \d)$ is an elliptic complex of dg $R$-modules. Of course, this means that the symbol complex (which is a bundle of dg $R$-modules on $T^\ast X$) is exact away from the zero section. 

Further, $\L$ is equipped with a collection of $R$-linear polydifferential operators
$$
l_n : \L^{\otimes n} \to \L
$$
making $\L$ into a sheaf of $L_\infty$ algebras on $X$ over $R$.  
\end{definition}

\begin{remark}
Note that in this definition, $R$ can be a nuclear Fr\'echet dg algebra. In that case, the tensor products should be the completed projective tensor product. 
\end{remark}
Our main reason for introducing the concept of an $R$-family of elliptic $L_\infty$ algebras is to talk about symmetries.  Recall that in homotopy theory, to give an action of a group $G$ on an object is the same as to give a family of objects over the classifying space $BG$.  There is a similar picture in homotopical algebra: to given an action of an $L_\infty$ algebra $\g$ on some object is the same as to give a family of such objects over $C^\ast(\g)$.  We will take this as our definition of action of an $L_\infty$ algebra $\g$ on an $L_\infty$ space. 
\begin{definition}
If $\g$ is an $L_\infty$ algebra, and $\mscr{L}$ is an elliptic $L_\infty$ algebra on a space $X$, a $\g$-action on $\mscr{L}$ is a family of elliptic moduli problems $\mscr{L}^{\g}$ on $X$, over the base ring $C^\ast(\g)$, which specialize to $\mscr{L}$ modulo the maximal ideal $C^{>0}(\g)$ of $C^\ast(\g)$.  
\end{definition}
\begin{remark}
The Chevalley-Eilenberg cochain complex $C^\ast(\g)$ is the completed pro-nilpotent dg algebra, which is an inverse limit
$$
C^\ast(\g) = \liminv C^\ast(\g) / I^n
$$
where $I$ is the maximal ideal $C^{>0}(\g)$.  
\end{remark}

\subsection{}
There is one more generalization we would like to consider.  The symmetries we considered above always preserve the base point of a formal pointed elliptic moduli problem.  Indeed, we defined a symmetry as a family of formal pointed elliptic moduli problems. In order to consider symmetries which do not preserve the base point, we need to modify our definition so that our family is no longer equipped with a base point.

Let $R$ be a differential graded ring with a nilpotent differential ideal $I \subset R$.  Recall that a formal pointed derived space over $R$ can be described by an is an $L_\infty$ algebra $\g$, in the category of flat $R$-modules.    We can modify this definition to access formal derived spaces which are not equipped with a base point, as follows.

Let $R^\sharp$ denote the graded algebra $R$, with zero differential. 
\begin{definition}
A \emph{curved $L_\infty$ algebra over $R$} consists of a locally free finitely generated graded $R^\sharp$-modules $\g$,  together with a derivation
$$\d :\what{\Sym}^\ast( \g [1] ^\vee ) \to \what{\Sym}^\ast ( \g [1] ^\vee )$$
of cohomological degree $1$ and square zero. In this expression, all tensors and duals are over the graded algebra $R^\sharp$.

The derivation $\d$ must make the completed symmetric algebra $\what{\Sym}^\ast( \g [1] ^\vee )$ into a differential graded algebra over the differential graded algebra $R$.  

Further, when we reduce modulo the nilpotent ideal $I \subset R$, the derivation $\d$ must preserve the ideal in $\what{\Sym}^\ast ( \g[1]^\vee)$ generated by $\g$.
\end{definition}

If $\g$ is a curved $L_\infty$ algebra over $R$, then we let $C^\ast(\g)$ be the differential graded algebra $\what{\Sym}^\ast (\g[1]^\vee )$ over $R$.  Note that $C^\ast(\g)$ is a pro-nilpotent commutative dga over $R$, and can thus be thought of as a formal derived scheme over $R$.  However, the ideal $C^{>0}(\g)$ is not necessarily preserved by the differential, because of the presence of the curving.  This indicates that this formal derived scheme is not pointed.   However, this formal derived scheme is pointed modulo the ideal $I$ in $R$, because the curving vanishes modulo $I$.  This pointing modulo $I$ is given by a map of dg $R$-algebras
$$
C^\ast(\g) \to R / I.
$$

\begin{definition}
Let $\g,\g'$ be curved $L_\infty$ algebras over $R$ is a map 
$$
C^\ast(\g') \to C^\ast(\g)
$$
of commutative pro-nilpotent dg $R$-algebras, with the property that the diagram
$$
\xymatrix{
C^\ast(\g') \ar[r] & C^\ast(\g) \ar[dl] \\
R / I  & \\
}
$$
commutes. 
\end{definition}
In the case that $I = 0$, then $\g$ and $\g'$ are not curved, and such a map is the same as an $L_\infty$-map $\g \to \g'$.  More generally, a map $\g \to \g'$ gives rise to a map of ordinary $L_\infty$-algebras when we reduce modulo $I$. 

A map $\g \to \g'$ can be described by a sequence of maps
$$
\phi_n : \Sym^n (\g[1]) \to \g'[1]
$$
for $n \ge 0$, such that $\phi_0$ vanishes modulo $I$ and such that the usual identity for an $L_\infty$ map holds. 
\begin{definition}
A map $\g \to \g'$ is an equivalence if, modulo $I$, it is a quasi-isomorphism of $L_\infty$ algebras.
\end{definition}
\begin{remark}
We assume that $\g$ are finitely-generated projective modules over the graded ring $R^{\sharp}$, and thus we can recover $\g$ from the $R^{\sharp}$-linear dual of $\g$.  If we do not assume that $\g$ is finitely generated (but continue to assume that $\g$ is flat over $R^{\sharp}$, it is better to use the coalgebra $C_\ast(\g)$ in place of the algebra $C^\ast(\g)$ in the definition of a map of curved $L_\infty$-algebras over $R$. 
\end{remark}

Once we have this definition, it is straightforward to modify our definition of $R$-family of elliptic  $L_\infty$ algebras.
\begin{definition}
Let $R$ be as above.  An $R$-family of curved elliptic $L_\infty$ algebras on $M$ is a graded bundle $L$ of $R^\sharp$ on $M$, whose sheaf of sections $\L$ is equipped with the structure of curved $L_\infty$ algebra over $R$, where the curving vanishes modulo the maximal ideal $I$, and where the structure maps are polydifferential operators. 
\end{definition}

\section{Mapping problems as elliptic moduli problems}
Many field theories of interest in mathematics and physics have, for their space of fields, the space of maps between two manifolds.  In this section I will outline how one can put field theories of this nature into the framework of elliptic moduli problems.  

\begin{definition}
An \emph{elliptic ringed space} is a manifold $M$, equipped with a sheaf $\A$ of commutative differential graded algebras over the sheaf $\Omega^\ast_M$, with the following properties.
\begin{enumerate}
\item $\A$ is concentrated in finitely many degrees.
\item Each $\A^i$ is a locally free sheaf of modules for $\Omega^0_M$ of finite rank.
\item The differential $\d$ on $\A$ makes $\A$ into an elliptic complex.
\item We are given a map of dg $\Omega^\ast_M$-algebras $\A \to \cinfty_M$.  
\end{enumerate}
We will let $\mscr{I} \subset \A$ be the ideal which is the kernel of the map $\A \to \cinfty_M$.
\end{definition}
Note that, because each $\A^i$ is a locally-free sheaf of modules over $\Omega^0_M = \cinfty_M$, $\A$ must be the sheaf of sections of a finite-rank graded vector bundle on $M$.  

We can discuss elliptic ringed spaces over $\R$ or $\C$; an elliptic ringed space over $\C$ is defined as above, except that we work over the sheaf of dg algebras $\Omega^\ast_M \otimes_{\R} \C$.

Here are some examples of elliptic ringed spaces.  
\begin{enumerate}
\item Let $M$ be any manifold.  Then letting $\A = \Omega^\ast_M$ gives an elliptic ringed space which we refer to as $M_{dR}$.  
\item Let $M$ be a complex manifold.  Then there is an elliptic ringed space $M_{\dbar}$ over $\C$, with $\A = \Omega^{0,\ast}(M)$, where the differential is the operator $\dbar$.  The product is simply the usual wedge product of forms.   The homomorphism from the de Rham complex to the Dolbeault complex is the identity on $\Omega^{0,\ast}(M)$ and sends $\Omega^{> 0, \ast}(M)$ to zero.  Thus, the Dolbeault complex is the quotient of the de Rham complex by the ideal generated by $\Omega^{1,\ast}(M)$. 

\item Let $M$ be any complex manifold, and let $R$ be any finite rank holomorphic bundle of graded Artinian algebras on $M$.  Then, $\Omega^{0,\ast}(M,R)$ defines an elliptic ringed space. 
\item  As a special case of the last example, let $M$ be a complex manifold and let $E$ be a holomorphic vector bundle on $M$.  We can define an elliptic ringed space which we right as $E_{\dbar}[1]$ (or as just $E[1]$) with underlying manifold $M$, and  sheaf of algebras
$$
\mscr{A} = \Omega^{0,\ast}( M, \Sym^\ast ( E^\vee[-1] ) ,
$$
where $\Sym^\ast(E^\vee[-1] )$ indicates the free sheaf of graded-commutative algebras generated by $E^\vee$ in degree $1$.  The differential on this dga is just $\dbar$.
\item  Let $M$  be a $4$-manifold with a conformal structure.  Then, the complex 
$$
\mscr{A} = \left\{ \Omega^0 (M) \xto{\d} \Omega^1(M) \xto{\d_-} \Omega^2_-(M) \right\}
$$
gives $M$ the structure of an elliptic ringed space.  Here, the product structure is defined by thinking of $\mscr{A}$ as the quotient of the de Rham complex if $M$ by the differential ideal generated by $\Omega^2_{+}(M)$. 
\end{enumerate}

\subsection{}
We will show how to construct elliptic moduli problems from elliptic ringed spaces.  

To start with, we will explain how to construct the formal moduli spaces of maps from an elliptic ringed space to a formal derived space.   

Recall that, for any $L_\infty$ algebra $\g$, there is a formal moduli problem $B \g$ which assigns to an Artinian dg ring $(R,m)$ the simplicial set $\op{MC}(\g \otimes m)$ of solutions to the Maurer-Cartan equation in $\g \otimes m$. 

If $A$ is a commutative dga, we can think of $L_\infty$ algebra $A \otimes \g$ as describing the formal moduli problems of maps $\Spec A \to B \g$, completed near the constant maps with values the base point of $B \g$.   

Given any finite-dimensional $L_\infty$ algebra $\g$, and any elliptic ringed space $(M, \A)$, we can define an elliptic moduli problem $(M, \A \otimes \g)$.  We will think of this as describing the space of maps from $(M,\A)$ to $B \g$. 

Since the formal neighbourhood of any point in a derived stack is described by an $L_\infty$ algebra, this construction shows that, for any derived stack $Y$ and any elliptic ringed space $(M,\A)$, the space of maps $(M,\A)$ to $Y$ formally completed near a constant map to a point $y \in Y$ is described by an elliptic $L_\infty$ algebra on $M$.  

\section{Global mapping problems and $L_\infty$ spaces}
In this section we will briefly sketch a language which allows us to describe the elliptic moduli problem describing a quite wide class of mapping problems, for instance, that between complex manifolds.  We have seen how to describe mapping problems from an elliptic ringed space to a formal derived scheme; the challenge is to globalize this description.  It is not essential to understand this section in order to follow the rest of the paper.

A formal derived scheme is described by an $L_\infty$ algebra.  Our global construction of a mapping problem will take as target an ``$L_\infty$ space''.
\begin{definition}
An $L_\infty$ space is a manifold $X$ with a sheaf of curved $L_\infty$ algebra $\g$ over the de Rham complex $\Omega^\ast_X$, where the curving vanishes modulo the ideal $\Omega^{>0}_X$. 

Let $(X,\g_X)$ and $(Y,\g_Y)$ be $L_\infty$-spaces.  A map $(X,\g_X) \to (Y,\g_Y)$ is a smooth map $f: X \to Y$ together with an $\Omega^\ast_X$-linear curved $L_\infty$ map
$$
f_\ast : \g_X \to f^\ast \g_Y = f^{-1} \g_Y \otimes_{f^{-1} \Omega^\ast_Y}\Omega^\ast_X.
$$
Such a map is an equivalence if the map $f :X \to Y$ is a diffeomorphism and if $f_\ast$ is an equivalence of curved $L_\infty$ algebras. 
\end{definition}
More details on the theory of $L_\infty$ spaces are presented in \cite{Cos11a}.  

A more standard approach to derived geometry is to work with spaces equipped with a sheaf of differential graded commutative algebras.   The theory of $L_\infty$ spaces is Koszul dual to this more standard approach.  If $(X,\g)$ is an $L_\infty$ space, let $C^\ast(\g)$ be the $\Omega^\ast_X$-linear Chevalley-Eilenberg cochain complex of $\g$.  This forms a sheaf of dg rings on $X$.   

For example, if $X$ is a complex manifold, then it is shown in \cite{Cos11a} that there is an $L_\infty$ space $(X,\g_X)$, with underlying manifold $X$, such that $C^\ast(\g_X)$ is quasi-isomorphic to the sheaf of holomorphic functions on $X$.  More precisely, let $\mscr{J}(\Oo_X)$ be the sheaf of smooth sections of the bundle of jets of holomorphic functions on $X$.  This is a bundle with a flat connection, so that we can define the de Rham complex $\Omega^\ast_X(\mscr{J}(\Oo_X))$.  The sheaf of $L_\infty$ algebras $\g_X$ is constructed so that there is an isomorphism of sheaves of $\Omega^\ast_X$-algebras 
$$
C^\ast(\g_X) \iso \Omega^\ast_X(\mscr{J}(\Oo_X)).
$$

If $(X,\g)$ is an $L_\infty$ space, we can reduce the curved $L_\infty$ algebra $\g$ modulo $\Omega^{> 0}_X$ to get a sheaf $\g^{red}$ of (non-curved) $L_\infty$ algebras over the sheaf $\cinfty_X$ of smooth functions on $X$.  In particular, $\g^{red}$ is a cochain complex over smooth vector bundles on $X$.  In the case that $\g_X$ encodes the complex structure on $X$, $\g_X^{red}$ is the complex tangent bundle $T^{1,0}_X$.  

\subsection{Global mapping problems}
Now let $(M,\A)$ be an elliptic ringed space, and $(X,\g)$ be an $L_\infty$ space.  We are interested in defining a notion of map from $(M,\A)$ to $(X,\g)$.  Such a map will, in particular, be a map of smooth manifolds $\phi : M \to X$. 

If $\phi : M \to X$ is a smooth map, we will let
$$
\phi^\ast \g = \phi^{-1} \g \otimes_{\phi^{-1}\Omega^\ast_X} \A.
$$
Thus, $\phi^\ast \g$ is a curved $L_\infty$ algebra over $\A$, whose curving vanishes modulo the ideal $\mscr{I} \subset \A$. 

\begin{definition}
A map $(M,\A) \to (X,\g)$ consists of the following data.  
\begin{enumerate}
\item A smooth map $\phi : M \to X$.
\item A solution $\alpha$ to the Maurer-Cartan equation in $\phi^\ast \g$, which vanishes modulo the ideal $\mscr{I} \subset \A$.
\end{enumerate}
\end{definition}

As an example,  the following lemma is proved in \cite{Cos11a}.
\begin{lemma}
Let $M$ and $X$ be complex manifolds, and let $\A = \Omega^{0,\ast}_M$ be the Dolbeaut resolution of the structure sheaf of $M$, and let $\g_X$ denote the curved $L_\infty$ algebra over $\Omega^\ast_X$ which encodes the complex structure of $X$.  

Then a map from $(M,\Omega^{0,\ast}_M)$ to $(X,\g_X)$ is the same as a holomorphic map from $M$ to $X$.
\end{lemma}

This formalism allows us to write down easily the elliptic $L_\infty$ algebra on a complex manifold $M$ controlling deformations of a fixed holomorphic map $\phi : M \to X$. 

The lemma implies that the curving of $\phi^\ast \g_X$ vanishes precisely when $\phi$ is holomorphic.  Thus, when $\phi$ is holomorphic, $\phi^\ast \g_X$ is a cochain complex of sheaves of $\Omega^{0,\ast}_M$-modules.  There is an isomorphism of dg $\Omega^{0,\ast}_M$-modules
$$
\phi^\ast \g_X \iso \Omega^{0,\ast}(M, \phi^\ast T X [-1] ) .
$$
Further, if $\mscr{J}(\Oo_X)$ denotes bundle of jets of holomorphic functions on $X$, we have an isomorphism of sheaves of differential graded $\Omega^{0,\ast}_M$-modules 
$$
C^\ast(\phi^\ast\g_X) \iso \Omega^{0,\ast}_M(\phi^\ast \mscr{J}(\Oo_X).
$$
Lie algebra cochains of $\phi^\ast \g_X$ are taken linearly over $\Omega^{0,\ast}_M$. 

A Maurer-Cartan element of $\phi^\ast \g_X$ (with coefficients in an Artinian dg ring $(R,m)$ is then the same as a map of $\Omega^{0,\ast}_M$-algebras
$$
\Omega^{0,\ast}(M, \phi^\ast J ( \Oo_X ) ) \to \Omega^{0,\ast}_M \otimes m.
$$
This is the same as a deformation of the holomorphic map $\phi$.

Note that since the $L_\infty$ algebra is $\Omega^{0,\ast}_M$-linear, locally on $M$ the $L_\infty$ structure is encoded by a holomorphically-varying family of $L_\infty$ structures on the holomorphic bundle $\phi^\ast TX[-1]$.  Globally, we can view the $L_\infty$ structure on $\Omega^{0,\ast}(M, \phi^\ast T X[-1])$ as encoding a homotopical version a holomorphic $L_\infty$ structure on the holomorphic bundle $\phi^\ast TX [-1]$.

Note that, for any holomorphic vector bundle $E$ on $X$, $\Omega^\sharp_X \otimes_{\cinfty_X} E$ has the structure of a curved $L_\infty$ module over $\g_X$.  This curved $L_\infty$ structure is characterized up to contractible choice by the property that $C^\ast(\g_X, E)$ coincides with the $\Omega^\ast(X, J(E))$, the de Rham complex of $X$ with coefficients in the $\cinfty$ bundle underlying the bundle of jets of holomorphic sections of $E$.  

It follows that $\Omega^{0,\ast}(M, \phi^\ast E)$ is equipped with the structure of $L_\infty$ module over the $L_\infty$ algebra $\phi^\ast \g_X = \Omega^{0,\ast}(M, T X[-1])$.  

The semi-direct product $L_\infty$ algebra
$$
\Omega^{0,\ast}(M, \phi^\ast TX [-1] \oplus \phi^\ast E[-1]) 
$$
controls deformations of pairs $(\phi, s)$ where $\phi : M \to X$ is a holomorphic map, and $s$ is a section of $\phi^\ast E$ (where we are deforming near $s = 0$). 

\section{Principal bundles on elliptic ringed spaces}

We are also interested in elliptic moduli problems describing principal bundles on elliptic ringed spaces.  For example, a principal bundle $G$-bundle on $M_{dR}$ will be a flat $G$-bundle on $M$; and if $G$ is a complex Lie group and $M$ a complex manifold, a principal $G$-bundle on $M_{\dbar}$ will be a holomorphic $G$-bundle on $M$.   The reader who is happy to accept that there is a reasonable notion of principal bundle on an elliptic ringed space should skip this section.

The definition for a general group is a little involved, so we will start with the definition for $GL(n)$.  
\begin{definition}
Let $(M,\A)$ be an elliptic ringed space over $\R$ or $\C$.  A rank $n$ vector bundle on $M$ is a sheaf $\E$ of dg modules over the dg ring $\A$, which, as a sheaf of graded modules over the sheaf of graded $\A^\sharp$ given by $\A$ without the differential, is locally free of rank $n$. 
\end{definition}
Note if $\E$ is a rank $n$ vector bundle on $(M,\A)$, then $\E / \mscr{I}$ is a locally free sheaf of rank $n$ over $\cinfty_M$, and so define a rank $n$ vector bundle on $M$.  

Let us list some examples.
\begin{enumerate}
\item A vector bundle on $M_{dR}$ is a vector bundle on $M$ with a flat connection,
\item If $M$ is a complex manifold, a vector bundle on $M_{\dbar}$ is a holomorphic vector bundle on $M$.
\item If $M$ is again a complex manifold, a vector bundle on $T[1] M_{\dbar}$ is a Higgs bundle on $M$.  Indeed, the sheaf of algebras on $M$ describing $T[1] M_{\dbar}$ is $\Omega^{\ast,\ast}_M$ equipped with the differential $\dbar$. If $\E$ is a vector bundle on $T[1] M_{\dbar}$, then $\E$ is isomorphic to $\Omega^{\ast,\ast}(M,V)$ for some rank $n$ holomorphic vector bundle on $M$; but with a differential of the form $\dbar + \phi$, where $\phi \in \Omega^{1,0}(M, \op{End}(V))$.  The condition that the differential squares to zero means that $\dbar \phi = 0$ and $[\phi,\phi] = 0$.  
\end{enumerate}

\subsection{}
Let us now discuss the definition of a general principal $G$-bundle on an elliptic ringed space $(M,\A)$.  Because of lack of space, I will be a little terse.

To motivate our definition, let us recall the definition of a connection and of a flat connection on a principal $G$-bundle on a manifold.

In what follows, if $\pi : P \to M$ is a principal $G$-bundle, and $\mscr{E}$ is a sheaf on $M$ of modules over $\cinfty_M$, we use the notation $\pi^\ast \mscr{E}$ to denote the sheaf of $\cinfty_P$-modules
$$
\pi^\ast \mscr{E} = \pi^{-1} \mscr{E} \otimes_{\pi^{-1} \cinfty_M} \cinfty_P.
$$
\begin{definition}
Let $G$ be a real Lie group. Let $P \to M$ be a principal $G$-bundle on a manifold $M$.  Then a connection on $P$ is a $G$-equivariant and $\cinfty_P$-linear map
$$
\eta: \Omega^1_P \to \pi^\ast \Omega^1_M
$$
which splits the natural map
$$
\pi^\ast \Omega^1_M \to \Omega^1_P.
$$
\end{definition}
By composing with the de Rham differential on $P$, such a connection induces a derivation $\d_\eta$ on the bundle of graded algebras $\pi^\ast \Omega^\ast_M$, by the formula
$$
\d_\eta ( f \otimes \omega) = \eta(\d f ) \wedge \omega + f \d \omega
$$
for a local section $\omega$ of $\pi^{-1} \Omega^\ast_M$ and $f$ of $\cinfty_P$. 
\begin{definition}
A connection $\eta$ is flat if $\d_\eta^2 = 0$. 
\end{definition}
We will adapt this definition to define the notion of principal bundle on an elliptic ringed space $(M,\A)$.  Thus, suppose $\pi : P \to M$ is a principal $G$-bundle, and $(M,\A)$ is an elliptic ringed space over $\R$.

Since $\A$ is a sheaf of algebras over $\Omega^\ast_M$, each graded component $\A^i$ is a sheaf of modules for $\cinfty_M$.  Thus, we can define a sheaf of $\cinfty_P$-modules
$$
\pi^\ast \A^i = \pi^{-1} \A^i \otimes_{\pi^{-1} \cinfty_M} \cinfty_P. 
$$
Note that the natural map $\Omega^1_M \to \A^1$ induces a map $\pi^\ast \Omega^1_M \to \pi^\ast \A^1$.
\begin{definition}
If $G$ is a real Lie group, and $P \to M$ is a principal $G$-bundle on $M$, then an $\A$-connection on $P$ is a $G$-equivariant $\cinfty_P$-linear map
$$
\eta: \Omega^1_P \to \pi^\ast \A^1
$$
whose restriction to the subsheaf $\pi^\ast \Omega^1_M$ is the natural map $\pi^\ast \Omega^1_M \to \pi^\ast \A^1$.
\end{definition}
Note that such an $\A$-connection on $P$ induces a differential operator $\d_\eta : \cinfty_P \to \pi^\ast \A^1$, obtained by composing the de Rham differential on $P$ with $\eta$.  This operator extends uniquely to a derivation $\d_\eta$ of the sheaf of graded algebras $\pi^\ast \A$ by the Leibniz rule
$$
\d_\eta ( f \otimes a ) = \eta( \d f ) a + f \d_\A a
$$
for all local sections $f$ of $\cinfty_P$ and $a$ of $\pi^{-1} \A^i$. Here $\d_\A$ indicates the differential on $\A$. 

\begin{definition}
An $\A$-connection on $P$ is flat if the derivation $\d_\eta$ of $\pi^\ast \A$ is of square zero.  
\end{definition}
Note that in this situation, $\pi^\ast \A$, with the differential $\d_\eta$, is a sheaf of differential graded algebras over the sheaf of dgas $\pi^{-1} \A$.  

If the bundle $P$ is trivialized, so that $P = M \times G$, then there is a natural flat $\A$-connection on $P$ given by composing the projection map $\Omega^1_P \to \pi^\ast \Omega^1_M$ with the natural map $\Omega^1_M \to \A$.  In this case, we  can identify $\pi^\ast \A$ as
$$
\pi^\ast \A = \A \boxtimes \cinfty_G.
$$
and the operator $\d_\eta$ is defined by
$$
\d_\eta (a \boxtimes f ) = (\d_\A f) \boxtimes g. 
$$

\begin{definition}
Let $(M,\A)$ be an elliptic ringed space over $\R$, and let $G$ be a real Lie group.   Then a principal $G$-bundle on $(M,\A)$ is a principal $G$-bundle on $M$ equipped with a flat $\A$-connection.
\end{definition}
Recall that $M_{dR}$ denotes the elliptic ringed space $(M,\Omega^\ast_M)$.  It is clear from the definition that a principal $G$-bundle on $M_{dR}$ is the same thing as a $G$-bundle on $M$ with a flat connection. 

\subsection{}
The definition is slightly different in the cases when $G$ is complex.  Let $(M,\A)$ be an elliptic ringed space over $\C$, and let $\pi : P \to M$ be a principal bundle for a complex Lie group $G$.   We will let $\Oo_P$ denote the sheaf of smooth functions on $P$ which are holomorphic on each fibre, and we will let $\Omega^1_{P,\dbar}$ denote the sheaf of $1$-forms on $P$ which are holomorphic $(1,0)$ forms when restricted to each fibre.  We will let $\Omega^k_{P,\dbar}$ denote the $\Oo_P$-linear exterior power of the sheaf $\Omega^1_{P,\dbar}$.  Finally, we will let $\Omega^\ast_{P,\dbar}$ denote the de Rham complex built from the sheaves $\Omega^k_{P,\dbar}$. 

If $\E$ is a sheaf on $M$ of $\cinfty_M$ modules, we will let $\pi^\ast \E$ denote the sheaf of $\Oo_P$ modules
$$
\pi^\ast \E = \pi^{-1} \E \otimes_{\pi^{-1} \cinfty_M} \Oo_P.
$$
Thus, in particular, we have sheaves $\pi^\ast \A^i$ of $\Oo_P$ modules. 

\begin{definition}
If $G$ is a complex Lie group, $(M,\A)$ is an elliptic ringed space over $\C$, and $P \to M$ is a principal $G$-bundle, then an $\A$-connection on $P$ is a $G$-equivariant $\Oo_P$-linear map
$$
\eta :  \Omega^1_{P,\dbar} \to \pi^\ast \A^1
$$
whose restriction to $\pi^\ast \Omega^1_M$ is the natural map $\pi^\ast \Omega^1_M \to \pi^\ast \A^1$.

Such a connection $\eta$ is flat if the derivation $\d_\eta$ on $\pi^\ast \A$ constructed as before has square zero.   
\end{definition}
\begin{definition}
If $G$ is a complex Lie group, a principal $G$-bundle on a complex elliptic ringed space $(M,\A)$ is a principal $G$-bundle on $M$ equipped with a flat $\A$-connection. 
\end{definition}
\begin{lemma}
If $M$ is a complex manifold, then a principal $G$-bundle on the elliptic ringed space $M_{\dbar} = (M,\Omega^{0,\ast}_M)$ is the same as a holomorphic principal $G$-bundle on $M$. 
\end{lemma}
\begin{proof}
$P \to M$ is a bundle where the fibre and the base are both complex manifolds.  Suppose that we have a complex structure on the total space $P$ which is $G$-equivariant and compatible with the complex structures on the fibre and the base.

Let $\Omega^1_{P,hol}$ be the sheaf of holomorphic $1$-forms on $P$, and, as before, let $\Omega^1_{P,\dbar}$ be the sheaf of $1$-forms which are holomorphic $1,0$ forms on each fibre.  

Recall that we use the notation $\Oo_P$ to denote the sheaf of functions on $P$ which are holomorphic on all fibres.  If $f \in \Oo_P$ is a local section, then the complex structure on $P$ allows us to define $\dbar f \in \Omega^{0,1}_P$.  Since $f$ is holomorphic on fibres, $\dbar f$ will actually land in the sheaf $\pi^\ast \Omega^{0,1}_M \subset \Omega^{0,1}_P$.  

Thus, a complex structure on $P$ induces a differential operator 
$$
\Oo_P \to \pi^\ast \Omega^{0,1}_M.
$$
The sheaf $\Omega^1_{P,\dbar}$ of $1$-forms on $P$ which are holomorphic $(1,0)$-forms on each fibre is the sheaf of K\"ahler differentials of $\Oo_P$.  The universal property of K\"ahler differentials thus gives us an $\Oo_P$-linear map 
$$
\Omega^1_{P,\dbar} \to \pi^\ast \Omega^{0,1}_M
$$
and so an $\A$-connection as desired. Of course, since $\dbar^2 = 0$ on $P$, this $\A$-connection is flat.

The converse is straightforward.
\end{proof}

\subsection{}
Let $(M,\A)$ be an elliptic ringed space over the field $\mbb{F}$, which is either $\R$ or $\C$. Given a principal $G$-bundle $(P,\eta) \to (M,\A)$, and a representation $V$ of $G$. one can define an associated sheaf $V_{P,\A}$ of $\A$-modules on $M$ as follows.   We give $\pi^\ast \A$ the differential $\d_\eta$.  Then, $\pi^\ast \A \otimes_{\mbb{F}} V$ is a $G$-equivariant sheaf of $\pi^{-1} \A$ modules on $V$.  We can then define a sheaf $V_{P,\A}$ on $U$ by defining
$$
V_{P,\A}(U) =  \left( \pi^\ast \A (U) \otimes_{\mbb{F}} V \right)^G. 
$$
In this way, for example, we construct a sheaf $\A( \g_P )$ of dg Lie algebras on $M$, over the dg algebras $\A$, associated to the adjoint representation of  $G$.  Note that, if $V_P$ denotes the sheaf of $\cinfty_M$ modules of sections of the adjoint vector bundle on $M$ associated to $P$, then $V_{P,\A}$ is isomorphic to $\A \otimes_{\cinfty_M} V_P$ equipped with a differential coming from $\eta$.  

\begin{lemma}
In this situation, to give a deformation of the flat $\A$-connection $\eta$ on the fixed principal $G$-bundle $P$ is the same as to give a Maurer-Cartan element
$$
\alpha \in (\A( \g_P ))^1.
$$ 
\end{lemma}
\begin{proof}
I will give the proof for the real case; the complex case is similar.  Suppose that $\eta'$ is another $\A$-connection on the principal bundle $P \to M$.  Then, $\eta' - \eta$ is a $G$-equivariant $\cinfty_P$ linear map
$$
\Omega^1_{P} \to \pi^\ast \A^1
$$
which is zero on $\pi^\ast \Omega^1_M$.  Thus, if $\Omega^1_\pi$ refers to the sheaf relative $1$-forms for the map $\pi : P \to M$, $\eta' - \eta$ is a $G$-equivariant map
$$
\Omega^1_\pi \to \pi^\ast \A^1.
$$
This is the same as a $\cinfty_M$-linear map
$$
\g^\vee_P \to \A^1
$$
where $\g^\vee_P$ is the sheaf of sections of the co-adjoint vector bundle on $M$ associated to $P$.  

Thus, we have seen that $A$-connections form a torsor for $\A^1 \otimes \g_P$.  It is straightforward to calculate that the condition for an $\A$-connection to be flat is the same Maurer-Cartan equation in the dg Lie algebra $\A( \g_P )$.  (Recall that $\A( \g_P ) = \g_P \otimes_{\cinfty_M} \A$ with a differential coming from $\eta$).   
\end{proof}

From this observation, we see that to every principal $G$-bundle $(P,\eta)$ on $(M,\A)$, we can construct an elliptic $L_\infty$ algebra $\A( \g_P )$, and that this elliptic $L_\infty$ algebra controls the deformations of $(P,\eta)$.  We have already seen special cases of this construction. When $\A = \Omega^\ast_M$, we have seen that the elliptic $L_\infty$ algebra $\Omega^\ast(M, \g_P$ controls deformations of a flat principal $G$ bundle $P$, and that in the complex case, $\Omega^{0,\ast}(M, \g_P)$ controls deformations of a holomorphic principal $G$ bundle $P$. 

Of course, the statement that the elliptic $L_\infty$ algebra $\A(\g_P)$ controls deformations of the principal $G$-bundle $(P,\eta)$ requires proof. Since the proof is identical to the proof of the more familiar statements concerning flat $G$-bundles or holomorphic $G$-bundles, I will omit it.  Alternatively, since in this paper we are mostly interested in formal moduli problems, the reader can simply take the Maurer-Cartan moduli problem associated to $\A(\g_P)$ as a definition of the formal moduli space of $G$-bundles on $(M,\A)$.  

\subsection{}
Finally, I will briefly consider one further example.  Let $M$ be a complex manifold, and let $E$ be a vector bundle on $M$.  Let us consider the complex elliptic ringed space $T[1] M_{\dbar}$ with sheaf of algebras $\A = \Omega^{\ast,\ast}(M)$, with differential $\dbar$. The following lemma is easy to verify from the above discussion. 
\begin{lemma}
A principal $G$-bundle on $T[1]M_{\dbar}$ is the same as a Higgs bundle on $M$, that is, a holomorphic principal $G$-bundle $P$ on $M$ together with an element
$$
\phi \in \Omega^{1,0}(M, \g_P ) 
$$
satisfying 
\begin{align*}
\dbar \phi &= 0\\\
[\phi, \phi] &= 0 .
\end{align*}
\end{lemma}

\addcontentsline{toc}{part}{Field theories}
\section{The classical Batalin-Vilkovisky formalism}
Before I get to giving a definition of a perturbative classical field theory in the language of elliptic $L_\infty$ algebras, I will explain a little about the general Batalin-Vilkovisky formalism for classical field theories.  

Let us start by discussing the classical Batalin-Vilkovisky formalism in a finite-dimensional toy model (which we can think of as a $0$-dimensional classical field theory).    Our model for the space of fields is a finite-dimensional smooth manifold manifold $M$.  The ``action functional'' is given by a smooth function $S \in \cinfty(M)$.      Classical field theory is concerned with solutions to the equations of motion.  In our setting, the equations of motion are given by the subspace $\op{Crit}(S) \subset M$.  Our toy model will not change if $M$ is a smooth algebraic variety or a complex manifold, or indeed a smooth formal scheme.  Thus we will write $\Oo(M)$ to indicate whatever class of functions (smooth, polynomial, holomorphic, power series) we are considering on $M$.

If $S$ is not a nice function, then this critical set can be highly singular.  The classical Batalin-Vilkovisky formalism tells us to take, instead the \emph{derived} critical locus of $S$.     (Of course, this is exactly what a derived algebraic geometer \cite{Lur09b,Toe06,CioKap01} would tell us to do as well).  

The critical locus of $S$ is the intersection of the graph 
$$\Gamma( \d S) \subset T^\ast M$$
with the zero-section of the cotangent bundle of $M$.  Algebraically, this means that we can write the algebra $\Oo(\op{Crit}(S))$ of functions on $\op{Crit}(S)$ as a tensor product
$$
\Oo(\op{Crit}(S)) = \Oo( \Gamma ( \d S) ) \otimes_{\Oo( T^\ast M ) } \Oo (M) .
$$
Derived algebra geometry tells us that the derived critical locus is obtained by replacing this tensor product with a derived tensor product.  Thus, the derived critical locus of $S$ (which we denote $\op{Crit}^h(S)$ is an object such that
$$
\Oo(\op{Crit}^h(S)) = \Oo( \Gamma ( \d S) ) \otimes^{\mbb{L}}_{\Oo( T^\ast M ) } \Oo (M) .
$$
In derived algebraic geometry, as in ordinary geometry, spaces are determined by their algebras of functions.  In derived geometry, however, one allows differential-graded algebras as algebras of functions (normally one restricts attention to differential-graded algebras concentrated in non-positive cohomological degrees). 

We will take this derived tensor product as a definition of $\Oo(\op{Crit}^h(S))$. 
\subsection{}
It is convenient to consider an explicit model for the derived tensor product.  By taking a standard Koszul resolution of $\Oo( M)$ as a module over $\Oo(T^\ast M)$, one sees that $\Oo(\Crit^h(S))$ can be realized as the complex
$$
\Oo(\Crit^h(S)) \simeq \dots \xto{\vee \d S} \Gamma (M, \wedge^2 T M ) \xto{\vee \d S} \Gamma (M, TM ) \xto{\vee \d S} \Oo(M).
$$
In other words, we can identify $\Oo(\Crit^h (S))$ with functions on the graded manifold $T^\ast[-1] M$, equipped with the differential given by contracting with $\d S$.  

Note that 
$$
\Oo( T^\ast[-1] M ) = \Gamma (M ,  \wedge^\ast TM )
$$
has a Poisson bracket of cohomological degree $1$, called the Schouten-Nijenhuis bracket.  This Poisson bracket is characterized by the fact that if $f, g  \in \Oo(M)$ and $X, Y \in \Gamma (M, T M)$, then 
\begin{eqnarray*}
\{X,Y\} = [X, Y] & \{X, f\} =  X f & \{f, g \} = 0 
\end{eqnarray*}
(the Poisson bracket between other elements of $\Oo(T^\ast[-1] M)$ is inferred from the Leibniz rule). 

The differential on $\Oo(T^\ast [-1] M)$ corresponding to that on $\Oo(\Crit^h(S))$ is given by 
$$
\d \phi = \{S, \phi\}
$$
for $\phi \in \Oo( T^\ast [-1] M)$.  

\subsection{}

The derived critical locus of any function is a dg manifold equipped with a symplectic form of cohomological degree $-1$.  We call such an object a \emph{$0$-symplectic} dg manifold.     In the Batalin-Vilkovisky formalism, the space of fields always has such a symplectic structure.  However, one does not require that the space of fields arises as the derived critical locus of a function.

\section{Classical field theories}

We would like to consider classical field theories in the BV formalism.  For us, such a classical field theory is specified by a $0$-symplectic elliptic moduli problem (that is, equipped with a symplectic form of cohomological degree $-1$).

We defined the notion of formal elliptic moduli problem on a manifold $M$ using the language of $L_\infty$ algebras.  Thus, in order to give the definition of a classical field theory, we need to understand the following question: what extra structure on an $L_\infty$ algebra $\g$ endows the corresponding formal moduli problem with a symplectic form?

The answer to this question was given by Kontsevich \cite{Kon93}.  Given a pointed formal moduli problem $\mc{M}$, the associated $L_\infty$ algebra $\g_{\mc{M}}$ has the property that
$$
\g_{\mc{M}} = T_p \mc{M} [-1].
$$
Further, we can identify geometric objects on $\mc{M}$ in terms of $\g_{\mc{M}}$ as follows.

\begin{quote}
\begin{tabular}{c c}
$C^\ast(\g_{\mc{M}})$ & The algebra $\Oo(\mc{M})$ of functions on $\mc{M}$ \\
$\g_{\mc{M}}$-modules  & $\Oo_{\mc{M}}$-modules \\
$C^\ast(\g_{\mc{M}}, V )$ & the $\Oo_{\mc{M}}$ module $\Gamma(\mc{M}, V)$\\
The $\g_{\mc{M}}$-module $g_{\mc{M}}[1]$ & $T \mc{M}$ 
\end{tabular}
\end{quote}
Following this logic, we see that the complex of two-forms on $\mc{M}$ can be identified with $C^\ast(\g_{\mc{M}}, \wedge^2 ( \g_{\mc{M}}^\vee [-1] ) )$. 

However, on a symplectic formal manifold, one can always choose Darboux coordinates.   Changes of coordinates on $\mc{M}$ correspond to $L_\infty$ isomorphisms on $\g_{\mc{M}}$.   In Darboux coordinates, the symplectic form has constant coefficients, and thus can be viewed as a $\g_{\mc{M}}$-invariant element of $\wedge^2 ( \g_{\mc{M}}^\vee [-1])$. 

Note that the usual Koszul rules of signs imply that
$$
\wedge^2 (\g_{\mc{M}}^\vee [-1]) = \Sym^2 (\g^\vee_{\mc{M}} ) [-2].
$$
To give a $\g_{\mc{M}}$-invariant element of $\Sym^2 (\g^\vee_{\mc{M}} )$ is the same as to give an invariant symmetric bilinear form on $\g_{\mc{M}}$.

Thus, we arrive at the following principle:
\begin{quote}
To give a formal pointed derived moduli problem with a symplectic form of cohomological degree $k$ is the same as to give an $L_\infty$ algebra with a symmetric, invariant, and non-degenerate pairing of cohomological degree $k-2$.  
\end{quote}

We will define a classical field theory to be an elliptic $L_\infty$ algebra equipped with a non-degenerate invariant pairing of cohomological degree $-3$.  Let us first define what it means to have an invariant pairing on an elliptic $L_\infty$ algebra. 
\begin{definition}
Let $M$ be a manifold, and let $E$ be an elliptic $L_\infty$ algebra on $M$.  An invariant pairing on $E$ of cohomological degree $k$ is a symmetric vector bundle map
$$
\ip{-,-}_E : E \otimes E \to \op{Dens}(M) [k]
$$
satisfying some additional conditions.
\begin{enumerate}
\item Non-degeneracy: we require that this pairing induces a vector bundle isomorphism
$$
E \to E^\vee \otimes \op{Dens}(M) [k].
$$
\item Invariance:  let $\E_c$ denotes the space of compactly supported sections of $E$.  The pairing on $E$ induces an inner product on $\E_c$, defined by
\begin{align*}
\ip{-,-} : \E_c \otimes \E_c &\to \R \\
\alpha \otimes \beta & \to \int_M \ip{\alpha,\beta}.
\end{align*}
We require that this is an invariant pairing on the $L_\infty$ algebra $\E_c$. 
\end{enumerate}
\end{definition}
Recall that a symmetric pairing on an $L_\infty$ algebra $\g$ is called invariant if, for all $n$, the linear map
\begin{align*}
\g^{\otimes n+1} &\to \R \\
\alpha_1 \otimes \dots \otimes \alpha_{n+1} & \mapsto \ip{l_n(\alpha_1,\dots,\alpha_n), \alpha_{n+1}} 
\end{align*}
is graded anti-symmetric in the $\alpha_i$. 

\begin{definition}
A formal pointed elliptic moduli problem with a symplectic form of cohomological degree $k$ on a manifold $M$ is an elliptic $L_\infty$ algebra on $M$ with an invariant pairing of cohomological degree $k-2$.
\end{definition}

\begin{definition}
A (perturbative) classical field theory on $M$ in the BV formalism is a formal pointed elliptic moduli problem on $M$ with a symplectic form of cohomological degree $-1$.  
\end{definition}

\subsection{}
Suppose that $\mscr{L}$ is an elliptic $L_\infty$ algebra with an invariant pairing of cohomological degree $k-2$ on a manifold $M$, in the sense described above.  Then, if $M$ is compact, the pairing sets up a quasi-isomorphism between $\mscr{L}(M)$ and the continuous linear dual $\mscr{L}(M)^\vee$, with a shift.    Since the differential on $\mscr{L}$ is elliptic, $\mscr{L}(M)$ has finite dimensional cohomology. Thus, $\mscr{L}(M)$ describes a finite-dimensional formal moduli problem (in the ordinary sense), together with a symplectic form of degree $k$.

One needs to be a little careful in interpreting the invariant pairing on $\mscr{L}$ on non-compact open subsets $U$ of $M$.  If $U$ is such a subset, then the invariant pairing on $\mscr{L}$ \emph{does not} give a quasi-isomorphism between $\mscr{L}(U)$ and its continuous linear dual.

Rather, if $\mscr{L}_c(U)$ denotes the compactly supported sections of $\mscr{L}$, the invariant pairing gives a quasi-isomorphism
$$
\mscr{L}(U) \iso \mscr{L}_c(U)^\vee.
$$

One should interpret this as follows.  $\mscr{L}(U)$ describes a (possibly infinite-dimensional) formal pointed moduli problem, whose tangent complex at the base point is $\mscr{L}(U)$.   Note that $\mscr{L}_c(U) \subset \mscr{L}(U)$ is an $L_\infty$ ideal: any higher bracket at least one of whose inputs is compactly supported will yield a compactly supported section of $\mscr{L}(U)$.

In the dictionary between formal geometry and $L_\infty$ algebras, $L_\infty$ ideals correspond to foliations.  Thus, we see that $\mscr{L}_c(U)$ gives a foliation on the formal moduli problem corresponding to $\mscr{L}(U)$.  Two points of the formal moduli problem $B \mscr{L}(U)$ associated to $\mscr{L}(U)$ are on the same leaf if they coincide outside of a compact set.  

Let us denote the sub-bundle of the tangent bundle of $B \mscr{L}(U)$ corresponding to this foliation as  $T_c B \mscr{L}(U) \subset T B \mscr{L}(U)$.

The pairing between $\mscr{L}(U)$ and $\mscr{L}_c(U)$ gives rise to an isomorphism between $T_c B \mscr{L}(U)$ and the dual to $T B \mscr{L}(U)$.  In other words, $B \mscr{L}(U)$ is equipped with a kind of ``leaf-wise'' symplectic structure, pairing a tangent vectors along a leaf with an arbitrary tangent vector.

This leafwise symplectic structure can also be thought of as a Poisson structure on $B \mscr{L}(U)$ satisfying a weak version of non-degeneracy: namely, that the map $T^\ast B \mscr{L}(U) \to T B \mscr{L}(U)$ arising from the Poisson tensor gives an isomorphism $T^\ast B \mscr{L}(U) \to T_c B \mscr{L}(U)$.  

\section{Quantization of classical field theories}
\label{section_factorization}
In \cite{CosGwi11}, we develop an approach to quantum field theory which starts with the definition of classical field theory given above.  This work has two parts: we give a system of axioms for a quantum field theory, and we have a theorem stating that we can construct quantum field theories starting with a classical field theory by obstruction theory.

Let me sketch the axioms we use.
\begin{definition}
Let $M$ be a manifold.  A \emph{prefactorization algebra} $\mc{F}$ on $M$ is the assignment of a cochain complex $\mc{F}(U)$ to every open subset $U \subset M$, together with product cochain maps maps
$$
m_{U_1,\dots,U_n}^V : \mc{F}(U_1) \otimes \dots \otimes \mc{F}(U_n) \to \mc{F}(V)
$$ 
whenever $U_1,\dots,U_n$ are disjoint open subsets of $V$.  This product map must be independent of the ordering chosen on the $U_i$'s.

In addition, a certain associativity condition must be satisfied.  If $U_{ij}$ (for $i = 1,\dots,n$ and $j = 1,\dots, m_i$) are disjoint opens in $V_i$ and $V_i$ are disjoint opens in $W$, then the diagram
$$
\xymatrix{
\otimes_{i = 1}^n \otimes_{j = 1}^{m_i} \mc{F}(U_{ij}) \ar[r] \ar[d] &  \mc{F}(W) \\
\otimes_{i = 1}^n V_j \ar[ur] &
}
$$
commutes.
\end{definition}
\begin{remark}
A prefactorization algebra can be more succinctly defined to be an algebra over the colored operad whose colors are open subsets of $M$, and where the set of morphisms from a $n$-tuple of colors $\{U_1,\dots,U_n\}$ to $V$ consists of a point if the $U_i$ are disjoint in $V$, and is empty otherwise. 
\item A factorization algebra is a prefactorization algebra satisfying a certain ``local-to-global'' property, saying that the value of $\mc{F}$ on a large open subset $V$ is determined (in a specified way) by the values of $\mc{F}$ on the sets in a sufficiently fine cover of $V$.\end{remark}

\begin{definition}
Suppose we have a classical field theory defined, as above, by an elliptic $L_\infty$ algebra $\L$ with an invariant pairing. Define the prefactorization $\Obs^{cl}_{\L}$ of classical observables of the theory to be the prefactorization algebra which assigns to $U$ the Chevalley-Eilenberg cochain complex complex
$$
\Obs^{cl}_{\L}(U) = C^\ast(\L(U)).
$$
\end{definition}
\begin{remark}
As explained in detail in \cite{CosGwi11}, we need to use a completed tensor product and continuous linear dual in the definition of Chevalley-Eilenberg cochain complex.
\end{remark}
We can view $C^\ast(\L(U))$ as the dg algebra of functions on the formal moduli problem $B \L(U)$. Thus, $C^\ast(\L(U))$ should be thought of as the algebra of functions on the derived moduli space of solutions to the equations of motion of our field theory.

We have seen above that $B \L(U)$ has a Poisson structure with a Poisson tensor of degree $1$. We would therefore expect that the commutative algebra $\Obs^{cl}_{\L}(U)$ has a Poisson bracket.  This is true, but in a homotopical sense:
\begin{lemma}
There is a sub-factorization algebra $\til{\Obs}^{cl}_{\L}$ of $\Obs^{cl}_{\L}$ which is equipped with a commutative product and a Poisson bracket of cohomological degree $1$, and is such that the inclusion
$$
\til{\Obs}^{cl}_{\L}(U) \to \Obs^{cl}_{\L}(U)
$$
is a quasi-isomorphism for every open subset $U$ of $M$. The factorization structure maps are obtained by applying the Chevalley-Eilenberg cochain functor to the restriction maps of $L_\infty$ algebras
$$
\L(V) \to \L(U_1) \oplus \dots \oplus \L(U_n)
$$
defined for any inclusion $U_1,\dots,U_n \into V$. 
\end{lemma}
We call a commutative dg algebra with a Poisson bracket of cohomological degree $1$ a $P_0$ algebra.  Thus, $\til{\Obs}^{cl}_{\L}(U)$ is a $P_0$-factorization algebra.  The $P_0$ structure and factorization structure are compatible: all the factorization product maps are maps of $P_0$-algebras. 

\begin{definition}
A quantization of the $P_0$-factorization algebra $\til{\Obs}^{cl}_{\L}$ is a factorization algebra $\Obs^q$ over $\C[[\hbar]]$ which, modulo $\hbar$, is equipped with a quasi-isomorphism of factorization algebras to $\til{\Obs}^{cl}_{\L}$; and such that, to leading order in $\hbar$, the failure of $\Obs^q$ is govenerned by the Poisson bracket on $\til{\Obs}^{cl}_{\L}$.
\end{definition}
\begin{remark}
We have only sketched the definition; details are given in \cite{CosGwi11}. In fact, this is a sketch of the definition of a \emph{weak quantization} of a $P_0$ factorization algebra.  The related notion of \emph{strong quantization} has stronger compatbilities between the structure of $\Obs^q$ and the Poisson structure on $\til{\Obs}^{cl}_{\L}$.
\end{remark}
The main result of \cite{CosGwi11} is the following.
\begin{theorem*}
A quantization, in the sense of \cite{Cos11}, of a classical field theory $\L$, leads to a factorization algebra $\Obs^q$ quantizing the factorization algebra $\Obs^{cl}_{\L}$ of classical observables of the the theory.
\end{theorem*}
Since the results of \cite{Cos11} allow one to construct quantum field theories from classical ones using obstruction theory, this result provides a rich source of quantum factorization algebras.  In particular, in this paper, we show that minimal twists of supersymmetric gauge theories admit unique quantizations on $\C^2$.  This yields a factorization algebra quantizing the classical observables of these theories. 

\section{Cotangent field theories}
The basic way symplectic manifolds arise in geometry is, of course, as cotangent bundles.  Thus, given any elliptic moduli problem, we can construct a classical field theory as a shifted cotangent bundle.  Let us explain this construction in detail.

Let $L$ be an elliptic $L_\infty$ algebra on a manifold $X$; and let $\mc{M}_{L}$ be the associated elliptic moduli problem. 

Let $L^!$ be the bundle $L^\vee \otimes \op{Dens}(X)$.  Note that there is a natural pairing between compactly supported sections of $L$ and sections of $L^!$. 

Recall that we use the notation $\L$ to denote the space of sections of $L$; we will let $\L^!$ denote the space of sections of $L^!$.

\begin{definition}
Let us define $T^\ast [k] \mc{M}_L$ to be the elliptic moduli problem associated to the elliptic $L_\infty$ algebra $L \oplus L^! [k-2]$. 

This elliptic $L_\infty$ algebra has a pairing of cohomological degree $k-2$.
\end{definition}
The $L_\infty$ structure on the space $\L \oplus \L^! [k-2]$ of sections of the direct sum bundle $L \oplus L^![k-2]$ arises from the natural $\L$-module structure on $\L^!$.
\begin{definition}
Let $\mc M$ be an elliptic moduli problem.  Then, the \emph{cotangent field theory} associated to $\mc M$ is the $0$-symplectic elliptic moduli problem $T^\ast[-1] \mc M$. Explicitly, if $\mc M$ is described by the elliptic $L_\infty$ algebra $\L$, then $T^\ast[-1] \mc M$ is described by $\L \oplus \L^![-3]$. 
\end{definition}
As we will see, many important theories in physics and mathematics are cotangent theories.  However, field theories are not normally presented in the language we use: instead, one is normally given an action functional and a gauge group.  In this language, one can recognize a cotangent field theory as follows.  

Suppose that we have a field theory on a manifold which has a space of fields which is the sections of the direct sum of two vector bundles $E$ and $F$.  Sections of $E$ will be denoted by $A$ and of $F$ by $B$.  Let us suppose that we have an action functional of the form
$$
S = \int_{M} B \Phi(A)
$$
where $\Phi : E \to F^! = F^\vee \otimes \op{Dens}_M$ is some local polynomial map (local meaning that the Taylor components of $F$ are polydifferential operators).  Let us also suppose that we have a gauge Lie algebra which is the space of sections of some bundle $C$, which acts on the space of fields preserving the decomposition into $E$ and $F$.  Let us use the notation $\mc{C}, \mc{E}, \mc{F}$ for sections of $C,E,F$. Let us further assume that the equation $\Phi(A) = 0$ is elliptic modulo gauge.  Then, we have a cotangent field theory.

Indeed, the elliptic $L_\infty$ algebra constructed by taking the derived critical locus of the action functional on the derived quotient of the space of fields by the gauge group is of the form $\mc{L} \oplus \mc{L}^![-3]$, where
$$
\mc{L} = \mc{C} \oplus \mc{E}[-1] \oplus \mc{F}^![-2].
$$
Theories of this nature are often called in the physics literature $BF$ theories, or $\beta-\gamma$ systems, or $b-c$ systems. 

The phase space of a cotangent field theory is particularly simple. Recall the \emph{phase space} of a classical field theory is the space of jets of solutions to the equations of motion along a compact codimension $1$ manifold. The phase space is always symplectic.   The phase space of a cotangent field theory is always a cotangent bundle.  

\subsection{}
Next I will give the definition of an $L_\infty$ action on a classical field theory.   Recall that, if $\g$ is an $L_\infty$ algebra and $\mscr{L}$ is an elliptic $L_\infty$ algebra on a space $X$, then a $\g$-action on $\mscr{L}$ is a family of elliptic $L_\infty$ algebras over $C^\ast(\g)$, which specializes to $\mscr{L}$ modulo the maximal ideal $C^{>0}(\g)$.  The same definition applies to classical field theories.

Let $R$ be a differential graded algebra, and let $\mscr{L}$ be an $R$-family of elliptic $L_\infty$ algebras.   Recall that this means that we have a graded bundle $L$ of $R^\sharp$-modules on $X$, whose sheaf $\mscr{L}$ of sections is equipped with a differential making it into a sheaf of dg $R$-modules, and with an $R$-linear $L_\infty$ structure.  We will let
$$
L^! = L^\vee \otimes \op{Dens}_X
$$
where $L^\vee$ is the $R^\sharp$-linear dual of $L$.  We will let $\mscr{L}^!$ denote the sheaf of sections of $L^!$.  This has a natural structure of sheaf of dg modules over $R$, with an $L_\infty$ action of $\mscr{L}$. 
\begin{definition}
An invariant pairing of degree $k$ on an $R$-family of elliptic $L_\infty$ algebras $\mscr{L}$ is an $R$-linear isomorphism 
$$
\mscr{L} \iso \mscr{L}^! [k]
$$
of sheaves of $\mscr{L}$-modules, which is symmetric as before.
\end{definition}
\begin{definition}
Let $\g$ is an $L_\infty$ algebra, and let $\mscr{L}$ be a classical field theory on a space $X$. Thus $\mscr{L}$ is an elliptic $L_\infty$ algebra on $X$ with an invariant pairing $\mscr{L} \iso \mscr{L}^![-3]$ of degree $-3$.   Then a $\g$-action on $\mscr{L}$ is a family of elliptic moduli problems $\mscr{L}^{\g}$ on $X$, flat over the base ring $C^\ast(\g)$, equipped with an invariant pairing of degree $-3$, which specializes to $\mscr{L}$ modulo the maximal ideal $C^{>0}(\g)$ of $C^\ast(\g)$. 
\end{definition}
If $\mscr{L}$ is an elliptic $L_\infty$ algebra on $X$ with an action of $\g$, then the cotangent field theory $T^\ast[-1] \mscr{L}$ also has a natural action of $\g$, compatible with the invariant pairing.  

\section{Examples of cotangent field theories}
Many classical field theories of interest in mathematics and physics arise as cotangent theories.  In this section we will list some examples.

In order to make the discussion more transparent, I will normally not explicitly describe the elliptic $L_\infty$ algebra related to an elliptic moduli problem; instead, I will simply define the elliptic moduli problem in terms of the geometric objects it classifies.  In all examples, it is straightforward using the techniques we have discussed so far to write down the elliptic $L_\infty$ algebra describing the formal neighbourhood of a point in any of the elliptic moduli problems we will consider.

\subsection{Self-dual Yang-Mills theory}
Let $X$ be an oriented $4$-manifold equipped with a conformal class of a metric.  Let $G$ be a compact Lie group.   Let $\mc M(X, G)$ denote the elliptic moduli problem parametrizing principal $G$-bundles on $X$ with a connection whose curvature is self-dual.

Then, we can consider the cotangent theory $T^\ast[-1] \mc M(X,G)$.   This theory is known in the physics literature as \emph{self-dual Yang-Mills theory}.  

Let us describe the $L_\infty$ algebra of this theory explicitly. Observe that the elliptic $L_\infty$ algebra describing the completion of $\mc M(X,G)$ near a point $(P,\nabla)$ is 
$$
\Omega^0(X, \g_P ) \xto{\d} \Omega^1( X, \g_P ) \xto{\d_-} \Omega^2_-(X, \g_P )
$$ 
where $\g_P$ is the adjoint bundle of Lie algebras associated to the principal $G$-bundle $P$. 

Thus, the elliptic $L_\infty$ algebra describing $T^\ast[-1] \mc M$ is given by the diagram
$$
\xymatrix{
\Omega^0(X, \g_P ) \ar[r]^{\d}  & \Omega^1( X, \g_P ) \ar[r]^{\d_-} \ar@{}[d]^{\bigoplus} & \Omega^2_-(X, \g_P ) \ar@{}[d]^{\bigoplus} & \\
& \Omega^2_-(X, \g_P ) \ar[r]^{\d} & \Omega^3( X, \g_P ) \ar[r] & \Omega^4(X, \g_P )
}
$$
This is a standard presentation of the fields of self-dual Yang-Mills theory in the BV formalism.  Indeed, it is obtained by applying the BV construction to the space of fields $\Omega^1(X, \g_P) \oplus \Omega^2_-(X,\g_P)$, with action functional
$$
S(A,B) = \int F(A)_- \wedge B
$$
(where $A$ is the connection and $B$ is the anti-self dual $2$-form), and gauge Lie algebra $\Omega^0(X, \g_P)$. 

Ordinary Yang-Mills theory arises as a deformation of the self-dual theory.   The deformation is given by simply deforming the differential in the dg Lie algebra presented in the diagram above by including a term which is multiplication by a constant $c$ from the copy of $\Omega^2_-$ situated in degree $1$ to the copy in degree $2$. The constant $c$ is the coupling constant of the theory.  

\subsection{Curved $\beta-\gamma$ systems}
Let $E$ be an elliptic curve and let $X$ be a complex manifold.   Let $\mc M(E, X)$ denote the elliptic moduli problem parametrizing holomorphic maps from $E \to X$.  As before, there is an associated cotangent field theory $T^\ast [-1] \mc M(E,X)$.

This field theory is called by physicists the curved $\beta-\gamma$ system. It arrises as a twist of the $\sigma$-model with $(0,2)$-supersymmetry \cite{Wit05, Kap05}.  (In \cite{Cos10a}, I called this theory a ``a holomorphic Chern-Simons theory''. Although from the point of view of derived geometry this terminology is reasonable, people have found it confusing so I would prefer not to use it).

The case we are considering here is that of maps from an elliptic curve $E$ to the cotangent bundle of a smooth variety $X$. There is an isomorphism
$$
T^\ast[-1] \mc{M}(E,X) = \mc{M}(E,T^\ast X).
$$
Indeed, the definition of $T^\ast[-1]$ implies that $T^\ast[-1]\mc{M}(E,X)$ is the elliptic moduli problem associated to holomorphic maps $f : E \to X$ together with a holomorphic section of $K_{E} \otimes f^\ast T^\ast X$. Since the canonical bundle is trivial, this is the same data as a holomorphic map to $X$.
 
This theory has an interesting role in both mathematics and physics.  For instance, the factorization algebra associated to this theory (using the techniques of \cite{CosGwi11}) is believed \cite{Wit05, Kap05} to be an incarnation of the chiral differential operators of $X$.    Also, it was shown in \cite{Cos11a} that the partition function of this theory (at least, the part which discards the contributions of non-constant maps to $X$) is the Witten genus of $X$.

\section{The $A$- and $B$-models of mirror symmetry}

In this section we will first describe the $\tfrac{1}{2}$-twisted $A$- and $B$-models, and then describe how the fully twisted $A$- and $B$-models arise as deformations of the $\tfrac{1}{2}$-twisted models.

\subsection{}
Both the $\tfrac{1}{2}$-twisted models we will be discussing are cotangent theories built from holomorphic maps between graded complex manifolds.  Recall that the space of maps between any elliptic ringed space and any $L_\infty$ space defines an elliptic moduli problem.  The source elliptic ringed space will be one of the following:
\begin{align*} 
\Sigma_{\dbar} &= (\Sigma, \Omega^{0,\ast}_{\Sigma} ) \\ 
T[1] \Sigma_{\dbar} &= (\Sigma, \Omega^{0,\ast}(\Sigma, \Oo_\Sigma \oplus K_\Sigma[-1] ) ) \\
T^\ast [1] \Sigma_{\dbar}  &= (\Sigma, \Omega^{0,\ast}(\Sigma, \Oo_\Sigma \oplus T_\Sigma[-1] ))\\
\Sigma_{dR} &= (\Sigma, \Omega^\ast_\Sigma \otimes \C ).
\end{align*}
Let $X$ be a complex manifold, and let $\g_X$ be the curved $L_\infty$ algebra over $\Omega^\ast_X$ encoding the complex structure of $X$. Recall that $\g_X$ has the property that the sheaf of $\Omega^\ast_X$-linear cochains $C^\ast(\g_X)$ are quasi-isomorphic to Dolbeaut complex $\Omega^{0,\ast}_X$.  

The target $L_\infty$ space be one of the following:
\begin{align*}
X_{\dbar} &= (X, \g_X ) \\
T[1] X_{\dbar} &= (X, \g_X \oplus \g_X[1] )\\
T^\ast[1] X_{\dbar} &= (X, \g_X \oplus \g_X^\vee ).
\end{align*}
Recall that points in the elliptic moduli problem of maps $\Sigma_{\dbar} \to X_{\dbar}$ are just holomorphic maps from $\Sigma$ to $X$.  We rewrite holomorphic maps in terms of the $L_\infty$ space $(X,\g_X)$ because this language allows us to give a concrete description of the formal neighbourhood of every holomorphic map in the derived moduli space of maps.  (Another important reason for using this language is that it is well-suited to a perturbative analysis of the quantum theory).
 
\subsection{}
Let us first discuss the $\tfrac{1}{2}$-twisted $B$-model.  Let $\Sigma$ be a Riemann surface, and let $X$ be a complex manifold. 
\begin{definition}
 The $\tfrac{1}{2}$-twisted $B$-model is the cotangent theory to the elliptic moduli problem on $\Sigma$ describing maps
$$
\phi : \Sigma_{\dbar} \to T^\ast[1] X_{\dbar}
$$
to the shifted cotangent bundle of $X$ (which we view as a graded complex manifold).
\end{definition}

Let us fix a holomorphic map $\phi : \Sigma \to X$.  We will describe the elliptic $L_\infty$ algebra on $\Sigma$ describing the classical field theory near $\phi$.     If $\phi : \Sigma \to X$ is a holomorphic map, the pull back 
$$
\phi^\ast \g_X = \phi^{-1} \g_X \otimes_{\phi^{-1} \Omega^\ast_X} \Omega^{0,\ast}_\Sigma
$$
is an ordinary (non-curved) $L_\infty$ algebra, which controls deformations of the holomorphic map $\phi$.  Also, as a sheaf of $\Omega^{0,\ast}_\Sigma$-modules,
$$
\phi^\ast \g_X = \Omega^{0,\ast}(\Sigma, \phi^\ast T X [-1] ) .
$$
This identification equips $\Omega^{0,\ast}(\Sigma, \phi^\ast T X[-1] )$ with an $\Omega^{0,\ast}_\Sigma$-linear $L_\infty$ structure.  In what follows, we will identify $\phi^\ast \g_X$ in this way. 

Deformations of $\phi$ as a map to $T^\ast [1] X$ are described by the elliptic $L_\infty$ algebra 
$$
\Omega^{0,\ast}( \Sigma, \phi^\ast T X [-1] \oplus  \phi^\ast T^\ast X ).
$$
The $L_\infty$ structure here is a semi-direct product $L_\infty$ algebra, arising from the natural $L_\infty$ action of $\Omega^{0,\ast}(\Sigma, \phi^\ast T X [-1] ) $ on $\Omega^{0,\ast} ( \Sigma, \phi^\ast T^\ast X )$. 

The cotangent theory to this elliptic $L_\infty$ algebra is then 
$$
\Omega^{0,\ast}( \Sigma, \phi^\ast T X [-1] \oplus  \phi^\ast T^\ast X \oplus K_\Sigma \otimes \phi^\ast T^\ast X [-1] \oplus K_\Sigma \otimes \phi^\ast T X [-2]  ).
$$
As before, the $L_\infty$ structure is a semi-direct product structure arising from the action of $\Omega^{0,\ast}(\Sigma, \phi^\ast T X [-1] )$ on the other summands. 

Note that we can also view this field theory as the cotangent field theory to the elliptic moduli problem described by 
$$
\Omega^{0,\ast}( \Sigma, \phi^\ast T X [-1] \oplus K_\Sigma\otimes \phi^\ast T X [-2] ).
$$
The latter elliptic moduli problem can be interpreted as the space of maps $T[1] \Sigma_{\dbar} \to X_{\dbar}$.

Following this reasoning, the full elliptic moduli problem (i.e.\ including the cotangent directions) for the $\tfrac{1}{2}$-twisted theory can also be interpreted as describing the space of holomorphic maps from $T[1] \Sigma$ to $T^\ast[1] X$.   The degree $-1$ symplectic form on this mapping spaces arises via the AKSZ formalism by transgressing the degree $1$ symplectic form on $T^\ast[1] X_{\dbar}$ using the degree $-1$ volume form on $T[1] \Sigma_{\dbar}$.  

\subsection{}
Next let us discuss the half-twisted $A$-model, with target $X$.  The factorization algebra associated to this field theory is conjectured \cite{Kap05} to be the chiral de Rham complex of $X$ \cite{GorMalSch00}.

\begin{definition}
The $\tfrac{1}{2}$-twisted $A$-model is to be the cotangent theory for the elliptic moduli problem of maps
$$
\Sigma_{\dbar} \to T[1] X_{\dbar}. 
$$
\end{definition}
If we perturb around a given holomorphic map $\phi :\Sigma \to X$, as above, then the elliptic $L_\infty$ algebra describing this mapping problem is 
$$
\Omega^{0,\ast}(\Sigma, \phi^\ast T[-1] X \oplus \phi^\ast T X )
$$
where $\phi^\ast T X $ is an $L_\infty$ module over $\phi^\ast T [-1] X$.  

The corresponding cotangent theory is described by the elliptic $L_\infty$ algebra
$$
\Omega^{0,\ast} (  \Sigma, \phi^\ast T[-1] X \oplus \phi^\ast T  X \oplus K_\Sigma \otimes \phi^\ast T^\ast X[-2] \oplus  K_\Sigma \otimes \phi^\ast T^\ast X[-3]).
$$

\subsection{}
Let us now consider the fully twisted $A$- and $B$-models.  We will start with the fully-twisted $B$-model.  

When we discuss supersymmetric gauge theories, we will see that the a twist of a supersymmetric field theory is given by a $\C^\times$ equivariant family of field theories over $\C$; which at the origin specializes to the original theory, and elsewhere to the twisted theory.  I will only discuss supersymmetry in detail in $4$ dimensions; in $2$ dimensions, I will only describe the twisted theories, without giving a derivation in terms of supersymmetry. 

\begin{definition}
The (fully twisted) $B$-model with source a Riemann surface $\Sigma$, and target a complex manifold $X$, is the cotangent theory to the elliptic moduli problem of maps
$$\Sigma_{dR} \to X_{\dbar}.$$
\end{definition}

Note that the un-derived version of the space of maps from $\Sigma_{dR}$ to $X_{\dbar}$ is the space of constant maps.  We can describe the derived moduli space of such maps as the $L_\infty$ space  $(X, \g_X \otimes \Omega^\ast(\Sigma) )$. Note that this is quasi-isomorphic to the $L_\infty$ space $(X,\g_X \otimes H^\ast(\Sigma,\C) )$, because there is a quasi-isomorphism of commutative dg algebras $\Omega^\ast(\Sigma) 
\simeq H^\ast(\Sigma)$. 

The corresponding cotangent theory is described by the $L_\infty$ space
$$
(X, (\g_X \oplus \g_X^\vee[2]) \otimes \Omega^\ast(\Sigma)).
$$
This cotangent theory (i.e. the fully-twisted $B$-model) can be interpreted as the space of maps 
$$
\Sigma_{dR} \to T^\ast[1] X_{\dbar}.
$$

In order to justify the relationship between this theory and the $\tfrac{1}{2}$-twisted $B$-model, we need to exhibit this theory as the general fibre of a $\C^\times$-equivariant family of theories over $\C$.

For $t \in \C$, let us define an elliptic ringed space $\op{Rees}_t (\Sigma_{dR})$ by
$$
\op{Rees}_t (\Sigma_{dR}) = (\Sigma, \Omega^{\ast,\ast}_\Sigma, \dbar + t \partial ).
$$
As $t$ varies, this defines the Rees family of algebras associated to the  Hodge filtration on $\Omega^\ast_\Sigma$.  At $t = 0$, $\op{Rees}_t(\Sigma_{dR})$ is $T[1] \Sigma_{\dbar}$, and for $t \neq 0$, $\op{Rees}_t(\Sigma_{dR})$ is $\Sigma_{dR}$.  

By considering the cotangent theory associated to the elliptic moduli problem of maps
$$
\op{Rees}_t(\Sigma_{dR}) \to X_{\dbar}
$$
we find a $\C^\times$-equivariant family of theories over $\C$.   At $t = 0$, this family of elliptic moduli problems specializes to the cotangent theory associated to maps $T[1] \Sigma_{\dbar} \to X$ and we have seen above that this cotangent theory is the $1/2$ twisted $B$-model.

\begin{remark}
When I claim that ``this is the $B$-model'' all I mean is that this is the classical field theory which a physicist would call the $B$-model.  Later we will see how, if one quantizes this theory, one finds a projective volume form on the derived space of maps from $\Sigma_{dR}$ to $X_{\dbar}$.   Integrals against this volume form should produce the correlators of the $B$-model topological field theory. 
\end{remark}

\subsection{}
Next, let us describe the fully-twisted $A$-model. From the point of view of perturbative quantum field theory, the fully-twisted $A$-model is not very interesting: we will see that the elliptic $L_\infty$ algebra describing this theory on a surface $\Sigma$ has vanishing cohomology. 

We have seen that the $\tfrac{1}{2}$-twisted $A$-model is the cotangent theory for holomorphic maps from $\Sigma$ to $T[1] X$.  Note that the sheaf on $X$ of holomorphic functions on $T[1] X$ is the holomorphic de Rham algebra of $X$, with zero differential. Of course, this is quasi-isomorphic to the sheaf $\Omega^{\ast}_X$ of $\cinfty$ de Rham complex, equipped with the differential $\dbar$. 

To define the fully-twisted $A$-model, we will deform this sheaf of rings into the de Rham complex of $X$.  Thus, we define
$$
\op{Rees}_t(X_{dR})
$$
to be the $L_\infty$ space
$$
(X, \g_X \xto{t \times \op{Id} } \g_X[-1] ) .
$$
Note that the $\Omega^\ast_X$-linear cochain complex of the $L_\infty$ algebra $\g_X[1] \xto{t \times \op{Id} } \g_X$ is quasi-isomorphic to $\Omega^\ast_X$ with differential $\dbar + t \partial$. 
The deformation parameter $t$ is best taken to be a formal parameter.  
\begin{definition}
The fully twisted $A$-model is the family of theories over $\C[[t]]$ obtained as  the cotangent theory to the elliptic moduli problem of maps 
$$
\Sigma_{\dbar} \to \op{Rees}_t( X_{dR} ) .
$$
\end{definition}
If we perturb near a given holomorphic map $\phi : \Sigma \to X$, then the elliptic $L_\infty$ algebra on $\Sigma$ (linear over $\C[[t]]$) describing this mapping problem is 
$$
\Omega^{0,\ast} (\Sigma, \phi^\ast T X \xto{t \times \op{Id} } \phi^\ast T X [-1] ) \otimes \C[[t]]. 
$$
That is, the sheaf $\phi^\ast T X \oplus \phi^\ast T X [-1]$ is equipped with a differential, which is $t$ times the identity map from $\phi^\ast T X$ in degree $0$ to $\phi^\ast T X $ in degree $1$.  Note that if we invert $t$, this sheaf of elliptic $L_\infty$ algebras has no cohomology.  Thus, perturbative quantum field theory has nothing interesting to say about the fully-twisted $A$-model.

When we discuss quantization later, we will see that the quantization of a cotangent theory to an elliptic moduli problem leads to a volume form on that elliptic moduli problem.  The $A$-model is essentially the cotangent theory to the de Rham stack of the moduli space of holomorphic maps from $\Sigma$ to $X$.  A volume form on any space $Z$ is a section of the Grothendieck-Serre dualizing sheaf. The dualizing sheaf of $Z_{dR}$ should be the topological Verdier dualizing sheaf of the topological space $Z$, and a section of this is a homology class on $X$.  

Thus, we would hope that the volume form arising from quantizing the $A$-model should be the virtual fundamental class of the moduli space of holomorphic maps from $\Sigma$ to $X$.  

More generally, we can see from this picture that local observables are what one would expect. In general, classical observables of a field theory on $\Sigma$ are functions on the solutions to the equations of motion on some open subset $U$ of $\Sigma$. In this case, if $U = D$ is a disc, then the equations of motion are the de Rham stack of the space of holomorphic stack from $D$ to $X$. Therefore, observables on $D$ should be the de Rham cohomology of this mapping space.  Local observables -- supported on a point in $\Sigma$ -- are obtained by using a formal disc in place of $D$.  Since the space of maps from a formal disc to $X$ is homotopy equivalent to $X$, we find that local observables are $H^\ast(X)$, as expected.    One expects that the operator product of local observables gives the quantum cup product on $H^\ast(X)$. 

One might hope that a complete construction of the $\tfrac{1}{2}$-twisted $A$-model will give us an interesting refinement of the theory of Gromov-Witten invariants.  Indeed, local observables of the $\tfrac{1}{2}$-twisted $A$-model are the chiral de Rham complex of $X$, which is infinite dimensional. One might hope that correlation functions of such observables (if they could be defined non-perturbatively, which is not at all obvious) should generalize Gromov-Witten invariants.

\subsection{}
Of course, there are many variants of the $\tfrac{1}{2}$-twisted $A$- and $B$-models.   Given any surface $\Sigma$ with a line bundle $L$, one can consider the cotangent theory of the elliptic moduli problem of holomorphic maps from the graded complex manifold $L[1]$ to $X$.  When $L = K_\Sigma$, this is the $\tfrac{1}{2}$-twisted $B$-model, and when $L$ is trivial, this is the $\tfrac{1}{2}$-twisted $A$-model.

\section{Quantization of cotangent field theories}
In this section I'll say a little bit about the geometric meaning of a quantization of a cotangent field theory. Quantization is interpreted in the sense of \cite{Cos11} and \cite{CosGwi11}; I briefly sketched the approach developed in \cite{CosGwi11} in section \ref{section_factorization}.    The main result sketched in this section is the following.
\begin{proposition}
Let $\mc M$ denote a formal elliptic moduli problem on a manifold $X$, and let $T^\ast[-1] \mc M$ be the associated cotangent field theory.   Then, a quantization of the cotangent field theory yields a volume form (defined up to multiplication by a scalar) on the formal derived space $\mc M(X)$. 
\end{proposition}
This result is a version of a theorem proved by Koszul \cite{Kos85}.  A closely related result is proved in \cite{BeiDri04}.

In order to separate the analytic technicalities from more conceptual issues, I will start by proving this result in a finite-dimensional context, that is, when the space-time manifold $X$ is a point.  Then I will indicate how the statements need to be modified for the general case.  This modification is not difficult, and uses techniques developed in great detail in \cite{Cos11} and \cite{CosGwi11}.

Thus, let $\g$ be a finite dimensional differential graded $L_\infty$ algebra, equipped with an invariant pairing of degree $-3$.  Such a $\g$ describes a formal pointed derived space with a symplectic form of degree $-1$.   Let $C^\ast (\g)$ denote the pro-nilpotent differential graded algebra of cochains on $\g$.  The invariant pairing on $\g$ endows $C^\ast(\g)$ with a Poisson bracket of cohomological degree $1$.  

\begin{definition}
A $P_0$ algebra is a commutative differential graded algebra equipped with a Poisson bracket of cohomological degree $1$.

The $P_0$ operad is the operad in the category of cochain complexes whose algebras are $P_0$ algebras.
\end{definition}

\subsection{}
Our approach to quantization parallels the deformation-quantization interpretation of quantum mechanics.  In ordinary deformation quantization, one starts with commutative algebra $A$ with a Poisson bracket of degree $0$.  This encodes the classical mechanical system we start with.  The problem of quantization is then interpreted as the problem of deforming $A$ into a flat family of associative algebras $\til{A}$ over $\C[[\hbar]]$, which reduces to $A$ modulo $\hbar$, and has the property that, to first order in $\hbar$, the failure of $\til{A}$ to be commutative is measured by the Poisson bracket on $A$. 

\subsection{}
Before introducing the notion of quantization of a dga with a Poisson bracket of degree $1$, I will show how the usual deformation-quantization story can be expressed using the language of operads.  

\begin{definition}
Let $P_k$ be the operad whose algebras are commutative dg algebras with a Poisson bracket of cohomological degree $1-k$.  

Thus, a $P_1$-algebra is a Poisson algebra in the usual sense.
\end{definition}
\begin{remark}
This terminology was suggested by Jacob Lurie: the point is that the $P_k$ operad is closely related to the operad $E_k$ of little $k$-discs. Note that $E_1$ is equivalent to the associative operad.  
\end{remark}
Next, we will construct an operad $BD_1$ over $\C[[\hbar]]$ with the property that, modulo $\hbar$, $BD_1$ is isomorphic to the Poisson operad; and that when we invert $\hbar$, $BD_1$ is isomorphic to the associative operad tensored with $\C((\hbar))$.  This operad was constructed by Ed Segal \cite{Seg10}.

\begin{definition}
Let $V$ be a cochain complex flat over $\C[[\hbar]]$.  A $BD_1$ structure on $V$ consists of an associative product $\star$, on $V$, and a Lie bracket $[-,-]$ on $V$, both $\C[[\hbar]]$ linear and of degree $0$; such that the following additional relations hold.
\begin{align*}
a \star b - (-1)^{\abs{a} \abs{b}} b \star a &= \hbar [a,b]\\ 
[a \star b, c] = a \star [b,c ] + (-1)^{\abs{b}\abs{c}} [a,c] \star b.
\end{align*}

Let $BD_1$ be the operad over $\C[[\hbar]]$ whose algebras are $BD_1$ algebras.
\end{definition}
Note that, modulo $\hbar$, $BD_1$ is the ordinary Poisson operad $P_1$. Indeed, the relations in $BD_1$ are precisely those ensuring that the product $\star$ is commutative, and that the bracket $[-,-]$ is a Poisson bracket. When $\hbar$ is inverted, $BD_1$ is the associative operad. One can check that $BD_1$ is flat over $\C[[\hbar]]$. 

\begin{lemma}
Suppose that $A$ is a differential graded $P_1$ algebra. Then, deformation quantizations of $A$ are the same as lifts of $A$ to a $BD_1$ algebra $\til{A}$.
\end{lemma}
\begin{proof}  Indeed, let $\til{A}$ be a deformation quantization of $A$.  If $\star$ denotes the product on $\til{A}$, then the bracket $a \star b - (-1)^{\abs{a} \abs{b}} b \star a$ lands in the image of $\hbar$.  Thus, we can define a Lie bracket on $\til{A}$ by
$$
[a,b] = \hbar^{-1} \left( a \star b - (-1)^{\abs{a} \abs{b}} b \star a\right).
$$
It is clear that, with the product $\star$ and bracket $[-,-]$, $\til{A}$ defines a $BD_1$ algebra which reduces modulo $\hbar$ to the Poisson algebra $A$.   The converse is clear.  
\end{proof}
One way to think about this construction is as follows.  We can think of an operad $P$ as the universal multi-category containing a $P$-algebra.  Thus, the operad $P_1$ is, in this sense, the universal Poisson algebra.  The operad $BD_1$ plays the role of the universal deformation quantization. 

\subsection{}
We will follow this operadic approach when formulating the notion of quantization of a $P_0$ algebra.    
\begin{definition}
Let $P_0$ denote the operad whose algebras are $P_0$-algebras as above.  The BD operad is the differential graded operad over $\C[[\hbar]]$ which, as a graded operad, is simply $P_0 \otimes \C[[\hbar]]$; but where the differential is given by 
$$
\d \ast = \hbar \{-,-\}.
$$
\end{definition}

Note that $\op{BD}$ is a flat family of operads over $\C[[\hbar]]$, which reduces mod $\hbar$ to $P_0$.  Further, when we invert $\hbar$, the operad $\op{BD}$ becomes contractible: the cohomology of $\op{BD}(n)[\hbar^{-1}]$ vanishes when $n > 0$, and the cohomology when $n = 0$ is one dimensional, corresponding to the unit element of a $P_0$ algebra.  
\begin{definition}
A quantization of a $P_0$ algebra $A$ is a BD algebra $\til{A}$, flat over $\C[[\hbar]]$, which reduces mod $\hbar$ to the $P_0$ algebra $A$.
\end{definition}

In section \ref{section_factorization}, we explained that the observables of a classical field theory form a $P_0$ factorization algebra, and explained that the observables of a quantum field theory form a factorization algebra over $\C[[\hbar]]$ which reduces modulo $\hbar$ to the observables of the classical field theory.  
\begin{definition}
Suppose we have a classical field theory on a manifold $M$. Let $\til{\Obs}^{cl}_{\L}(U)$ be the $P_0$ factorization algebra of classical observables.  Then, a \emph{strong quantization} of this $P_0$ factorization algebra is a $BD_0$ factorization algebra $\Obs^q$ with a quasi-isomorphism of $P_0$ algebras $\Obs^q\op{mod} \hbar \simeq \til{\Obs}^{cl}_{\L}$. 
\end{definition}
\begin{remark}
In \cite{CosGwi11}, we prove a general theorem allowing one to construct a factorization algebra of quantum observables for every quantum field theory in the sense of \cite{Cos11}.  This implies that we can quantize the factorization algebra of classical observables using obstruction theory. However, we have not proved that the factorization algebra of quantum observables has a $BD_0$ structure (although we conjecture that this is the case).  Instead, our quantum factorization algebra satisfies a weaker compatibility condition with the $P_0$ structure on classical observables.  
\end{remark}

\subsection{}
The main result of this section will relate quantizations of certain $P_0$ algebras to ``projective volume forms'' on formal moduli problems.  Thus, I need to explain what I mean by a projective volume form on a formal moduli problem.  

\begin{definition}
Let $X$ be a complex manifold.  A \emph{projective volume form} on $X$ is a flat connection on the canonical bundle $K_X$.  Equivalently, it is a trivialization of the $\Oo_X^\times / \C^\times$-torsor associated to  $K_X$.
\end{definition}
Note that what we call a projective volume form is \emph{not} the same as a volume form on $X$ up to scalar multiplication.  Locally, the two notions coincide.  Globally, however, the flat connection on $K_X$ may have non-trivial monodromy: this provides an obstruction to lifting a projective volume form to a volume form. 
\begin{lemma}
A projective volume form on $X$ is the same as a right $D_X$-module structure on $\Oo_X$. 
\end{lemma}
\begin{proof}
If $M$ is a right $D_X$-module, then $M \otimes K_X^{-1}$ is a left $D_X$-module.  Thus, a right $D_X$-module structure on $\Oo_X$ induces a left $D_X$-module structure on $K_X^{-1}$, that is, a flat connection on $K_X^{-1}$; and so a flat connection on $K_X$.    The converse is immediate.   
\end{proof}

\subsection{}
We are interested in projective volume forms on formal moduli problems.    I will follow a very helpful suggestion of Nick Rozenblyum, and \emph{define} a projective volume form on a formal moduli problem to be a right $D$-module structure on the structure sheaf.  The reason for this approach is that I don't know how to define the canonical sheaf of a formal moduli problem; presumably, the correct definition would involve some version of Grothendieck-Serre duality.  

Let us introduce some notation related to formal moduli problems.  Let $\g$ be an $L_\infty$ algebra (without an invariant pairing).  We will let $B \g$ denote the corresponding formal moduli problem; thus, $\Oo(B \g)$ will refer to the dga of cochains on $\g$, and so on. 

Let $\op{Vect}( B \g)$ be the dg Lie algebra of vector fields on $B \g$, that is, 
$$\op{Vect}( B \g) = C^\ast(\g, \g[1]) = \op{Der} (\Oo(B \g) ).$$ 
Let us define the associative algebra of differential operators $D (B \g)$ to be the free associative algebra generated over $\Oo(B \g)$ by $X \in \op{Vect}( B\g)$ subject to the usual relations:
\begin{align*}
X \cdot f - f \cdot X = (X f) \\
f \cdot X = f X 
\end{align*}
where $\cdot$ denotes the associative product in $D (B \g)$, and juxtaposition indicates the action of $\op{Vect}(B \g)$ on $\Oo(B \g)$ or the $\Oo(B \g)$-module structure on $\op{Vect}( B \g)$. 

\begin{definition}
A projective volume form on $B \g$ is a right $D(B\g)$-module structure on $\Oo( B \g)$. 
\end{definition} 

\subsection{}
Let $T^\ast [-1] B \g$ denote the formal moduli problem $B ( \g \oplus \g^\vee[-3])$.  Note that $T^\ast [-1]  B \g$ has a symplectic form of degree $-1$, so that $\Oo( T^\ast[-1] B \g)$ is a commutative dga equipped with a Poisson bracket of degree $1$.   In other words, $\Oo( T^\ast [-1] B \g )$ is a $P_0$ algebra.

Note that there's a $\C^\times$ action on $T^\ast [-1] B \g$ by scaling the cotangent fibres, that is, by scaling the $\g^\vee [-3]$ in the Lie algebra $\g \oplus \g^\vee [-3]$.    Under this $\C^\times$ action on $\Oo( T^\ast [-1] B \g )$, the Poisson bracket has weight $-1$, and the product has weight $0$.

One way to say this is to observe that there is a $\C^\times$ action on the $P_0$ operad, where the Poisson bracket has weight $-1$; and that $\Oo (T^\ast [-1] B \g)$ is a $\C^\times$-equivariant algebra. 

Note that the operad $BD_0$ also has a $\C^\times$, where the parameter $\hbar$ has weight $1$, the product has weight $0$ and the bracket has weight $-1$.  Thus, we can talk about $\C^\times$-equivariant quantizations of a $\C^\times$-equivariant $P_0$ algebra. 

\subsection{}
Possible right $D(B \g)$-module structures on $\Oo( B \g)$ form a simplicial set, as we can consider such objects in families over the algebra $\Omega^\ast(\tr^n)$ of forms on the $n$-simplex.

We can also define a simplicial set of quantizations of $\Oo(T^\ast[-1] B\g)$.   By the graded along of the Darboux lemma, $\Oo( T^\ast[-1] B \g)$ can not be deformed as a graded $P_0$ algebra (without a differential).  Thus, any quantization of $\Oo( T^\ast[-1] B \g)$ is given by a $BD_0$ structure on $\Oo( T^\ast[-1] B \g)[[\hbar]]$ with fixed underlying graded $P_0$ algebra.  Such a $BD_0$ structure is entirely specified by the differential, which must be compatible with the $P_0$ structure in the sense described above, and which must agree with the given differential modulo $\hbar$. 

We can thus define the simplicial set of quantizations by saying that the $n$-simplices are families of $BD_0$ structures on $\Oo( T^\ast[-1] B \g) \otimes \Omega^\ast(\tr^n) [[\hbar]]$, with fixed underlying $P_0$ structure, and linear over $\Omega^\ast(\tr^n)$. 

\begin{proposition}
There is a natural homotopy equivalence between the simplicial set of right $D(B \g)$-structures on $\Oo(B \g)$ and that of $\C^\times$-equivariant quantizations of the $P_0$ algebra $\Oo(T^\ast [-1] B \g)$.  
\end{proposition}
\begin{proof}
Suppose we have a right $D(B\g)$-module structure on $\Oo(B \g)$.  If $V \in D(B \g)$ and $f \in \Oo(B \g)$, we will let $f \rho(V) \in \Oo(B \g)$ be the result of applying $V$ to $f$ using the right $D$-module structure.  Note that, by definition, for $g \in \Oo(B \g) \subset D(B \g)$, $f \rho(g) =  fg $.  Thus, for $X \in \op{Vect}(B \g)$, 
$$f \rho (X)   = 1 \rho(f) \rho(X) = 1 \rho( f X).$$
Thus, the entire action is determined by a linear map 
\begin{align*}
\Phi : \op{Vect}( B \g)  & \to \Oo ( B \g) \\
\Phi( X) = 1 \rho(X). 
\end{align*}
Note that the relations in $D (B \g)$ imply that $X \cdot f = (Xf )  + (-1)^{\abs{f} \abs{X}} f X$.  It follows that
\begin{equation}
\Phi (f X) - f \Phi(X) =  -  (-1)^{\abs{f} \abs{X}} (X f )      \in \Oo(B \g) \tag{$\dagger$} \label{eqn_dagger}.
\end{equation}

We will use the map $\Phi$ to define a quantization of $\Oo ( T^\ast [-1] B \g)$.  The underlying graded Poisson algebra of our quantization is $\Oo ( T^\ast [-1] B \g) [[\hbar]]$.  To describe the differential, let us introduce an auxiliary operator $\tr$ on $\Oo ( T^\ast [-1] B \g)$.  The operator $\tr$ is the unique order $2$ differential operator with the property that, for $f \in \Oo(B \g)$, and for $X \in \op{Vect}(B \g ) [1] \subset \Oo (T^\ast [-1] B \g)$, we have
\begin{align*}
\tr (f ) &= 0\\
\tr (X) &= \Phi(X).
\end{align*}
The fact that $\tr$ is well-defined follows from the fact that $\Phi$ is an order $1$ differential operator.   It is not hard to verify (from equation (\ref{eqn_dagger})) that the failure of $\tr$ to be a derivation is the Poisson bracket on $\Oo( T^\ast[-1] B \g)$.  Thus, we define the differential on our $BD_0$ algebra to be $\d + \hbar \tr$, where $\d$ is the usual differential on $\Oo( T^\ast[-1] B \g)$. 

Let us now consider the converse.  The simplicial set of quantizations we are considering has, for $n$-simplices, $BD_0$ structures on $\Oo(T^\ast[-1] B \g) \otimes \Omega^\ast(\tr^n) [[\hbar]]$ with fixed underlying graded $P_0$ algebra.   The $\C^\times$ invariance properties of the quantization force the differential to be of the form $\d + \hbar \tr$, where $\d$ is the given differential on $\Oo(T^\ast[-1] B \g)$, and $\tr$ is some operator mapping $\Gamma(B \g, \wedge^i T B \g ) \to \Gamma(B \g, \wedge^{i-1} T B \g)$.  The operator $\tr$ is determined uniquely by its behaviour on $\op{Vect}(B \g)$; restricted to this subspace, it must be a cochain map 
$$
\Phi_{\tr} : \op{Vect}( B \g) \to \Oo( B\g)
$$
satisfying the axiom in (\ref{eqn_dagger}).  

Thus, we have set up a bijection of simplicial sets between right $D(B \g)$-module structures on $\Oo ( B \g)$ and $\C^\times$-invariant quantizations. 
\end{proof}

\subsection{}
In this section we will discuss the $P_0$ structure on the classical observables of a cotangent field theory, and explain how quantization of a cotangent field theory leads to a projective volume form on the moduli of solutions to the equations of motion. 

Now, suppose we have a classical field theory on a compact manifold $M$, given by an elliptic $L_\infty$ algebra $\mscr{L}$ equipped with an invariant pairing of degree $-3$.  Let $\mscr{L}(M)$ denote the global sections of $\mscr{L}$; note that, since $M$ is finite dimensional  and the differential on $\mscr{L}$ is elliptic, $\mscr{L}(M)$ has finite dimensional cohomology.  

Let $C^\ast(\mscr{L}(M))$ denote the Chevalley-Eilenberg cochain complex of the topological $L_\infty$ algebra $\mscr{L}(M)$.  By definition, $C^\ast(\mscr{L}(M))$ is a commutative differential graded algebra; one would expect that the invariant pairing on $\mscr{L}(M)$ induces a Poisson bracket of degree $1$ on $C^\ast(\mscr{L}(M))$.  

It turns out that a little work is require to produce this Poisson bracket, because of analytic difficulties inherent in the infinite-dimensional nature of $\mscr{L}(M)$.  In \cite{Cos11}, it is shown how to construct a canonical, up to contractible choice, $P_0$ algebra structure on $C^\ast(\mscr{L}(M))$.  (The results of \cite{Cos11} assume an additional technical condition, namely the existence of a gauge fixing condition.  This condition is easy to verify in all examples we consider here).

However, in order to avoid having to discuss any infinite-dimensional issues, I will explain how the results of \cite{Cos11} yield structures on the cohomology of $\mscr{L}(M)$.
\begin{lemma}
There is a canonical, up to contractible choice, $L_\infty$ structure on $H^\ast(\mscr{L}(M))$ for which the pairing on $H^\ast(\mscr{L}(M))$ is invariant.  Further, this $L_\infty$ algebra is  equipped with an $L_\infty$-equivalence to $\mscr{L}(M)$.
\end{lemma}
This $L_\infty$ structure is given by the  familiar homotopical transfer of structures. 

The Chevalley-Eilenberg  cochain complex $C^\ast(H^\ast(\mscr{L}(M)))$ is thus equipped with a $P_0$ structure, arising from the invariant pairing on $H^\ast(\mscr{L}(M))$.

In \cite{Cos11}, a definition of \emph{quantization} of a classical field theory is presented.  Part of the data of a quantization of a classical field theory is a quantization of the $P_0$ algebra $C^\ast(H^\ast(\mscr{L}(M)))$, in the operadic sense discussed earlier.    Our discussion about the relationship between $BD_0$ algebras and projective volume forms now shows the following.
\begin{lemma}
Suppose that $\mscr{L}_0(M)$ is an elliptic moduli problem on $M$, and that  $\mscr{L}(M) = \mscr{L}_0(M) \oplus \mscr{L}^!_0(M)[-3]$ is the corresponding cotangent field theory.

Then, a $\C^\times$-invariant quantization of this cotangent theory (using the definitions of \cite{Cos11}) yields a projective volume form on the formal moduli problem $\mscr{L}_0(M)$.  
\end{lemma}

The key concept in the definition of quantization presented in \cite{Cos11} is that of \emph{locality}.  This is reflected in the fact that possible quantizations of a classical field theory form a sheaf on $M$: so that the projective volume form on the moduli problem $\mscr{L}_0(M)$ is built up from local data on $M$. 

In the algebraic language presented in \cite{CosGwi11}, locality is reflected in the fact that the a quantization provides a $BD_0$ algebra\footnote{In the current incarnation of \cite{CosGwi11}, the theorem is a little weaker than I have stated: we have not yet constructed the $BD_0$ algebra on $\Obs^q(U)$ for each open subset $U$, we have only constructed these as cochain complexes.  We have, however, constructed the $BD_0$ structure on global observables $\Obs^q(M)$.  The stronger statement is work in progress . } quantizing not just $C^\ast(\mscr{L}(M))$, but also $C^\ast(\mscr{L} (U))$, for each open subset $U \subset M$.   We refer to this $BD_0$ algebra as the complex of quantum observables $\Obs^q(U)$.   As $U$ varies, the quantum observables form what we call a \emph{factorization algebra}.

\newcommand{\Spin}{\op{Spin}}
\newcommand{\PT}{\mbb{PT}}
\renewcommand{\P}{\mbb{P}}
\addcontentsline{toc}{part}{Supersymmetry}
\section{Basics of supersymmetry}
Many of the field theories of interest in mathematics arise as twists of supersymmetric field theories.  In this section I will say what it means for a field theory in $4$-dimensions to be supersymmetric,  and explain the twistor-space construction of certain supersymmetric gauge theories on $\R^4$  developed in \cite{BoeMasSki07,Wit04}. 

Before we start, I should say a few words about the gradings used in supersymmetry.  Supersymmetric field theories have two gradings: one by $\Z/2$, and one by $\Z$.  The first grading is the number of fermions, and the second is a cohomological grading, called in the physics literature the ``ghost number''.   Both gradings contribute to signs: if we move an element $\alpha$ of bidegree $(a_1,a_2)$ past an element $\beta$ of bidegree $(b_1,b_2)$, we introduce a sign of $(-1)^{(a_1 + a_2) ( b_1 + b_2) }$. 

When we deal with such bi-graded cochain complexes, the differential is of degree $(0,1)$.  In other words, the differential only affects the cohomological degree and not the fermionic degree.  We call such an object a \emph{super cochain complex}.

There are two possible ways of shifting a super-cochain complex: we can reverse the fermionic grading, or we can shift the cohomological grading.  If $V$ is a super-cochain complex, we will let $\pi V$ denote the same complex with fermionic grading reversed and cohomological grading unchanged.  We will let $V[-1]$ denote the same complex with fermionic grading unchanged, and cohomological grading shifted by one. 

We would like to adapt the definition of classical field theory given earlier to this bigraded context.    The main point to remember is that all algebraic structures we consider \emph{preserve} the fermion degree, and have the same cohomological degree as they do in the world of ordinary cochain complexes.   Thus, a \emph{super $L_\infty$ algebra} is a super cochain complex $V$, equipped with maps $l_n : V^{\otimes n} \to V$ of bidegree $(0,2-n)$, satisfying the usual $L_\infty$ relation.  

The relationship between $L_\infty$ algebras and formal moduli problems exists in the super context as well, where one defines a super formal moduli problem to be a functor on the category of super Artinian dg algebras. 

All of the definitions we gave earlier (like that of an elliptic $L_\infty$ algebra) can be defined without any difficulty in the super context.
\begin{definition}
Let $M$ be a manifold. A \emph{perturbative classical field theory with fermions} on $M$ is a super elliptic $L_\infty$ algebra $\mscr{L}$ on $M$ together with an invariant pairing on $\mscr{L}$ of bidegree $(0,-3)$.
\end{definition}
Given any super elliptic $L_\infty$ algebra $\mscr{L}$ on $M$, corresponding to a sheaf of formal super moduli problems on $M$, one can, as before, construct a classical field theory $\mscr{L} \oplus \mscr{L}^![-3]$.  This classical field theory will be called the \emph{cotangent field theory} associated to $\mscr{L}$. 

\subsection{}
First, let us give the definition of supersymmetry.  I will concentrate on dimension $4$ in Euclidean signature.   Suppose that $\mc M$ is a classical field theory on $\R^4$.  Thus, $\mc M$ is a formal elliptic moduli problem on $\R^4$, equipped with a symplectic form of degree $-1$.

Let us suppose that the classical field theory $\mc M$ is invariant under the group $\op{Spin}(4) \ltimes \R^4$, the double cover of the group of Euclidean symmetries of $\R^4$.   This means that this group acts on $\mc M(\R^4)$ in such a way that, if $g \in \op{Spin}(4) \ltimes \R^4$, the action of $g$ on $\mc M(\R^4)$ takes $\mc M(U) \subset \mc M(\R^4)$ to $\mc M(g(U))$. 

We will further assume that this action differentiates to an action of the Lie algebra $\mf{so}(4) \ltimes \R^4$ on $\mc M(U)$, for each $U \subset \mc M$.  

We will define a supersymmetric field theory on $\R^4$ to be a field theory equipped with an action of a certain super Lie algebra called the super Euclidean Lie algebra, extending the given action of the Euclidean Lie algebra.   

In Euclidean signature, this larger Lie algebra is only defined over $\C$, and not over $\R$.  Thus, to talk about supersymmetry in Euclidean signature, we need to use elliptic moduli problems which are defined over $\C$.  (For an elliptic moduli problem $\mc M$ to be defined over $\C$ just means that the corresponding $L_\infty$ algebra $\g_{\mc M}$ is a complex Lie algebra).

\subsection{}
Recall that there is an isomorphism of groups
$$
\Spin(4) \iso SU(2) \times SU(2).
$$
We will refer to the two copies of $SU(2)$ as $SU(2)_+$ and $SU(2)_-$.  

Let $\mc S_+$ denote the $2$ complex dimensional fundamental representation of $SU(2)_+$, endowed with the trivial $SU(2)_-$ action.  Thus, $\mc S_+$ is a complex representation of $\Spin(4)$.  Define $\mc S_-$ in the same way. 

Let $V$ denote the defining $4$-dimensional real representation of $SO(4)$.  There is an isomorphism of complex $\Spin(4)$ representations
$$
V_\C = V \otimes_\R \C \iso \mc{S}_+ \otimes \mc{S}_-.
$$

Given any complex vector space $W$, we can define a super-translation Lie algebra based on $W$.  The dimension of $W$ will be the number of supersymmetries.

The super-translation Lie algebra $T^{W}$ is defined to be the super Lie algebra
$$
T^{W} = V_\C \oplus \Pi ( \mc S_+ \otimes W \oplus \mc S_- \otimes W^\vee).
$$
Thus, the even part of $T^W$ is $V_\C$, and the odd part is $ \mc S_+ \otimes W \oplus \mc S_- \otimes W^\vee$.

The only non-trivial bracket on $T^{W}$ is between $\mc S_+ \otimes W$ and $\mc S_- \otimes W^\vee$.  If $\phi : \mc S_+ \otimes \mc S_- \to V_\C$ is the natural map, then the bracket is defined by the formula
$$
[s \otimes w, s' \otimes w'] = \phi ( s \otimes s') \ip{w,w'}
$$
for $s \in \mc S_+$, $s' \in \mc S_-$, $w \in W$, $w' \in W^\vee$. 

We will often use the notation 
$$
T^{\mscr{N}=k} = T^{\C^k}
$$
to refer to the super-translation Lie algebra associated to $\C^k$.  The cases of interest are when $k = 1,2,4$. 

Note that there is a natural action of $\op{Spin}(4)$ on $T^{W}$.  Thus, the complexified Lie algebra $\mf{so}(4,\C)$ acts on $T^{W}$, so that we can define the semi-direct product $\mf{so}(4,\C) \ltimes T^{W}$. This is the (complexified) super Euclidean Lie algebra.  

Also, the group $GL(W)$ acts on $T^{W}$, in a way commuting with the natural action of $\op{Spin}(4)$.  
\begin{definition}
A field theory on $\R^4$ with $\mscr{N}=k$ supersymmetries is a $\op{Spin}(4)  \ltimes \R^4$-invariant super elliptic moduli problem $\mc M$ defined over $\C$ with a symplectic form of cohomological degree $-1$; together with an extension of the action of the complexified Euclidean Lie algebra $\mf{so}(4,\C) \ltimes V_\C$ to an action of the complexified super-Euclidean Lie algebra $\mf{so}(4,\C) \ltimes T^{\mscr{N}=k}$.

Given any complex Lie subgroup $G \subset GL(k,\C)$, we say that such a supersymmetric field  theory has $R$-symmetry group $G$ if the group $G$ acts on the theory in a way covering the trivial action on space-time $\R^4$, and compatible with the action of $G \subset GL(k,\C)$ on $T^{\mscr{N}=k}$. 
\end{definition}
\begin{remark}
Although in this paper we are mostly content with a complex space of fields, reality conditions are important in physics. One way they appear is in the path integral: when we have a complex space of fields with an action functional which is holomorphic, the path integral is performed over a contour, which is a real slice of the space of fields.  In perturbation theory, the contour is irrelevant; however, it appears when we want to include contributions from non-trivial solutions to the equations of motion (e.g. instantons). For example, we have seen that quantization of a cotangent theory leads to a (projective) volume form on the space of solutions to the equations of motion, which in the case of a complex theory will be a holomorphic volume form on a complex manifold.  The partition function should be the integral of this volume form; in order to perform this integral, we need to choose a real slice of the space of solutions to the equations of motion. 
\end{remark}
\subsection{Examples}
Now that we have a definition of a supersymmetric classical field theory, I should give some examples. 

Supersymmetric gauge theories are difficult to write down explicitly on real space.  The more supersymmetry, the harder it is to write down a supersymmetric theory.  For the $\mscr{N}=1$ theory, one can use the ``superspace formalism'' \cite{DelFre99} to construct supersymmetric field theories.  The fields of the theories are geometric objects (such as connections) on the supermanifold $\R^{4 \mid 4}$, satisfying certain constraint equations.  For $\mscr{N}=2$ theories, there is also a superspace formalism, but the constraints the fields need to satisfy become more complicated.  For the $\mscr{N}=4$ theory, no superspace formulation is known. 

On the other hand, it has become clear in recent years \cite{Wit04,BoeMasSki07} that twistor theory provides a simple uniform construction of all supersymmetric gauge theories.  

For this reason, I will describe the twistor-space formulation of supersymmetric gauge theories. I will start by briefly recalling the classical Penrose-Ward correspondence, relating holomorphic bundles on twistor space and anti-self-dual bundles on $R^4$.  Then, I'll describe the super-twistor space in some detail, and state the super analogue of the Penrose-Ward corresponence.  

\subsection{}  
Recall \cite{WarWel91} that the twistor space $\PT$ of $\R^4$ can be identified with the complement of a $\P^1$ on $\P^3$, or equivalently with the total space of $\Oo(1) \oplus \Oo(1) \to \P^1$.

The group of Euclidean (and even conformal) symmetries of $\R^4$ acts on twistor space.  Let us rewrite twistor space in a way which makes this action evident.  

As before, let $\mc S_+$ and $\mc S_-$ denote the spinor representations of $\Spin(4) = SU(2) \times SU(2)$, defined as the pullback of the fundamental representations of the two copies of $SU(2)$.  The more invariant definition of the twistor space is that it is 
$$
\PT = \Oo(1) \otimes \mc S_- \to \P(\mc S_+),
$$ 
the total space of the rank $2$ vector bundle $\Oo(1) \otimes \mc S_-$ over $\P(\mc S_+)$. In this presentation, the action of the complex group 
$$SL(\mc S_-) \times SL(\mc S_+) = \op{Spin}(4,\C)$$ on $\PT$ is evident.

Next, we need to construct an action of the group of translations of $\R^4$ on twistor space $\PT$.  This action will preserve the fibration $p : \PT \to \P(\mc S_+)$.  Note that the relative tangent bundle to $p$ is $p^\ast \Oo(1) \otimes \mc S_-$.  There is a natural map
$$
H^0 ( \P(\mc S_+), \Oo(1) \otimes \mc S_- ) \to H^0(\PT, p^\ast (\Oo(1) \otimes \mc S_- ). 
$$
There is a canonical isomorphism
$$
H^0 ( \P(\mc S_+), \Oo(1) \otimes \mc S_- ) = \mc S_+ \otimes \mc S_-
$$
and $\mc S_+ \otimes \mc S_-$ is (as a $\op{Spin}(4)$ representation) the complexification of $\R^4$.  In this way, we have construct a $\op{Spin}(4)$-equivariant map
$$
\C^4 \to H^0 (\PT, T \PT ) 
$$
and so the desired action of translations on $\R^4$ on the twistor space.

\subsection{}
Next, we will construct the twistor fibration, which is a non-holomorphic fibration $\PT \to \R^4$. 

Note that the we can identify the space of sections of the bundle $\Oo(1) \otimes \mc S_- \to \P(\mc S_+)$ with the space of $\P^1$'s in $\PT$ which project isomorphically onto $\P(\mc S_+)$.  We have seen above that this space of sections is a copy of $\C^4 = \R^4 \otimes \C$.

It follows that, for each $x \in \R^4$, there is a corresponding $\P^1_x \in \PT$, namely the image of the section corresponding to 
$$
x \in \R^4 \subset \R^4 \otimes \C = \mc S_+ \otimes \mc S_-.
$$
This $\P^1$ is called the \emph{twistor line} corresponding to $x \in \R^4$.

A standard lemma is the following.
\begin{lemma}
For $x,y \in \R^4$, the twistor lines $\P^1_x$, $\P^1_y$ are disjoint.  Further,  every point in twistor space is in a unique $\P^1_x$ for some $x \in \R^4$.
\end{lemma}
\begin{proof}
Note that $\P^1_0$ is the zero section of the bundle $\Oo(1)\otimes \mc S_- \to \P(\mc S_+)$.  We need to verify that, for each $x \in \R^4$, the section $s_x$ of $\Oo(1) \otimes \mc S_-$ has no zeroes. 

To see this, observe that -- since $\mc S_+$ and $\mc S_-$ both have symplectic structures -- the space $\mc S_+ \otimes \mc S_-$ of sections of $\Oo(1) \otimes \mc S_-$ has a symmetric and non-degenerate inner product.  Recall that $\mc S_+ \otimes \mc S_- = \R^4 \otimes \C$; this inner product is the complexification of the Euclidean inner product on $\R^4$.   Let us choose a Darboux basis $\alpha,\beta$ for $\mc S_-$. Let us write our section $s_x$ of $\Oo(1) \otimes \mc S_-$ as 
$$s_x = f \alpha +  g\beta,$$ where $f,g \in H^0(\P(\mc S_+), \Oo(1))$.  Such a section has a zero if and only if $f$ and $g$ are proportional.  If $f$ and $g$ are proportional, $\ip{s_x, s_x} = 0$, where the inner product is the one on $\mc S_+ \otimes \mc S_-$ inherited from the symplectic forms on $\mc S_+$ and $\mc S_-$.  

Now, the subspace $\R^4 \subset \C^4$ is positive-definite with respect to this inner product, so the section corresponding to any $x \in \R^4$ has no zeroes.

It is an easy exercise to verify that every point in twistor space lies in the twistor line for some $x \in \R^4$.  
\end{proof}
  
Thus, there is a non-holomorphic fibration -- the \emph{twistor fibration} -- 
$$\pi : \PT \to \R^4$$ 
with the property that $\pi^{-1}(x) = \P^1_x$ for each $x \in \R^4$.  

\subsection{}
We are interested in twistor space (and its supersymmetric generalizations) as a tool to construct supersymmetric gauge theories.  We will discuss the non-supersymetric case first.

Recall that the Penrose-Ward correspondence \cite{WarWel91} states that there is a natural bijection between vector bundles on $\R^4$ equipped with anti-self-dual connections $\R^4$ , and holomorphic vector bundles on $\PT$ which are trivial on every twistor fibre.   A refined version of the Penrose-Ward correspondence \cite{BoeMasSki07,Mov08} is the following.
\begin{theorem}
The anti-self-dual Yang-Mills theory on $\R^4$ (i.e. the cotangent theory to the elliptic moduli problem of anti-self-dual instantons) is equivalent to the cotangent theory for holomorphic vector bundles on $\PT$, which are trivial on every twistor fibre. 
\end{theorem}
A proof of this is presented in \cite{BoeMasSki07} and \cite{Mov08}.    

Let me give a more precise version of the statement.  Let $G$ be a semi-simple algebraic group.  Let $P \to \PT$ be a principal $G$ bundle, and suppose that for every twistor fibre $\P^1 \subset \PT$, the $G$-bundle $P \mid_{\P^1}$ is trivial (but \emph{not} trivialized).  

The Penrose-Ward correspondence tells us that, associated to $P$, there is a $\cinfty$ $G$-bundle $\mscr{T}(P)\to \mbb{\R}^4$, equipped with a connection whose curvature is anti-self-dual.  

Note that we are considering the \emph{complex} anti-self-duality equations for a complex connection on a complex principal bundle.  In order to find solutions to the real anti-self-duality equation, one needs to put some reality structures on the holomorphic principal bundle $P \to \PT$.  We will, however, only consider the complex case.  

Given $P \to \PT$, we can consider the elliptic moduli problem of deformations of $P$. This is described by the elliptic Lie algebra
$$
\mscr{L}_{hol} = \Omega^{0,\ast} ( \PT, \g_P )
$$ 
where $\g_P$ adjoint bundle of Lie algebras associated to $P$.  The cotangent theory to this elliptic moduli problem is described by the elliptic Lie algebra
$$
T^\ast[-1] \mscr{L}_{hol} = \Omega^{0,\ast} (\PT, \g_P ) \oplus \Omega^{3,\ast} ( \PT, \g_P^\vee ) . 
$$

If $U \subset \R^4$, let $\PT(U) \subset \PT$ be the inverse image of $U$ under the twistor fibration $\pi : \PT \to \R^4$.    Our elliptic moduli problem on $\PT$ gives rise to a sheaf of dg Lie algebras on $\R^4$, which assigns to $U$ the dg Lie algebra 
$$
\pi_\ast \mscr{L}_{hol}  =   \Omega^{0,\ast}( \PT(U), \g_P ).
$$
This is not quite an elliptic Lie algebra on $\R^4$, because it is not built from sections of a finite-rank vector bundle on $\R^4$.  

However, it \emph{is} quasi-isomorphic to an elliptic Lie algebra on $\R^4$.     Recall that the moduli of deformations of the anti-self-dual Yang-Mills bundle $\mscr{T}(P)$ on $\R^4$ can be described by the elliptic Lie algebra
$$
\mscr{L}_{ASD} = \Omega^0(\R^4, \g_{\mscr{T}(P)} ) \to 
\Omega^1(\R^4, \g_{\mscr{T}(P)} ) \to 
\Omega^2_{+}(\R^4, \g_{\mscr{T}(P)} ) .
$$
Then, it is shown in \cite{BoeMasSki07,Mov08} that there is a homotopy equivalence of sheaves of Lie algebras on $\R^4$
$$
\mscr{L}_{SD} \simeq \pi_\ast \mscr{L}_{hol}.
$$
This homotopy equivalence extends to the cotangent theories: there is a homotopy equivalence
$$
T^\ast [-1] \mscr{L}_{SD} \simeq \pi_\ast T^\ast[-1] \mscr{L}_{hol},
$$
where $T^\ast [-1] \mscr{L}_{SD} $ is the elliptic Lie algebra describing the shifted cotangent bundle to the moduli of anti-self-dual bundles. 

This shows that the anti-self-dual Yang-Mills theory on $\R^4$ is equivalent to the cotangent theory for holomorphic bundles on the twistor space $\PT$.

\subsection{}
Next, we will discuss the supersymmetric version of this story, which will allow us to construct supersymmetric gauge theories. 

I gave the definition of a supersymmetric field theory on $\R^4$.  However, one can also define supersymmetric field theories on other spaces equipped with an action of the translation group $\R^4$: if $X$ is a manifold with such an action, then a supersymmetric field theory on $X$ is a field theory on $X$, invariant under the action of $\R^4$, together with an extension of the infinitesimal action of $\R^4$ to an action of the appropriate super-translation Lie algebra.    We will construct supersymmetric field theories on twistor space.

The supersymmetric extension of the Penrose-Ward correspondence states that there is an equivalence between holomorphic vector bundles on a graded complex manifold -- the super-twistor space -- and the anti-self-dual versions of supersymmetric gauge theories on $\R^4$.  In \cite{BoeMasSki07,Wit04} it is shown that this correspondence can be lifted to a homotopy equivalence of sheaves of $L_\infty$ algebras\footnote{These authors didn't use this language: they stated the result as an equivalence between classical field theories.  However, since the $L_\infty$ structure is encoded in the classical action functional, their results prove this statement.} on $\R^4$, where one sheaf is the elliptic $L_\infty$ algebra describing solutions to the super-symmetric anti-self-duality equations, and the other sheaf is the push-forward from the twistor space of the elliptic Lie algebra describing holomorphic bundles on the appropriate super-twistor space. 

The twistor space formulation of anti-self-dual supersymmetric gauge theory is far more transparent than the real-space formulation.  Thus, we will take the twistor space formulation as a definition.  

There are three versions of supersymmetric gauge theory: those with $\mscr{N}=1,2,4$ supersymmetry.  Thus, there are three corresponding super-twistor spaces.  
 
\subsection{}
We will present a uniform construction of super-twistor spaces.  For any $k \ge 0$, let us define a holomorphic super-manifold
$$
\PT^{\mscr{N}=k} = \Pi (\Oo(1) \otimes \C^k) \to \PT.
$$
Thus, $\PT^{\mscr{N}=k}$ is the complex $\Z/2$ graded manifold which is the total space of the purely odd bundle $\Oo(1) \otimes \C^k$   Let 
$$
\Oo_{\mscr{N}=k} = \wedge^\ast (\Oo(-1) \otimes \C^k)
$$
be the structure sheaf of $\PT^{\mscr{N}=k}$, viewed as a sheaf of super algebras on $\PT$.     Note that $\wedge^i (\Oo(-1) \otimes \C^k)$ is in fermion degree $i \mod 2$, and cohomological degree $0$.

Let $P \to \PT$ be a principal $G$-bundle.    Let $\g_P$ be the adjoint bundle of Lie algebras on $\PT$.   Let us consider the super elliptic moduli problem of holomorphic bundles on $\PT^{\mscr{N} = k}$ which are deformations of the bundle $\pi^\ast P$.

The corresponding elliptic Lie algebra is 
$$
\mscr{L}_{\mscr{N} = k } = \Omega^{0,\ast} (\PT, \Oo_{\mscr{N}=k} \otimes_{\Oo_{\PT}} \g_P ).
$$
Here, we have a mixture of cohomological and fermionic degrees: $\Omega^{0,i}(\PT, \wedge^j ( \Oo(-1) \otimes \C^k  ) \otimes \g_P )$ is in cohomological degree $i$ and fermionic degree $j \mod 2$. 

\begin{lemma}
The graded complex manifold $\PT^{\mscr{N}=k}$ has a natural action of the super-translation Lie algebra $T^{\mscr{N}=k}$.  Further, the composed projection
$$
\PT^{\mscr{N}=1} \to \PT \to \R^4
$$
is compatible with the action of the translation Lie algebra of $\R^4$ on all the three spaces. 

Further, this action is compatible with the natural action of $\op{Spin}(4,\C) \times GL(k,\C)$ on $\PT^{\mscr{N}=k}$. 
\end{lemma}
\begin{proof}
Let us change to more invariant notation. Let $W$ be a complex vector space of dimension $k$.  The super Lie algebra $T^{W}$ was defined earlier.  We will let $\PT^{W}$ refer to the total space of $\Pi ( \Oo(1) \otimes W)$ over $\PT$.   If $W = \C^k$ then, then $\PT^{W}$ is what we called $\PT^{\mscr{N}=k}$ above.

Thus, $\PT^{W}$ is the total space of the $\Z/2$-graded bundle 
$$
\Oo(1) \otimes \left( \mc S_  \oplus \Pi W \right) 
$$
over $\P(\mc S_+)$. 

We will use this presentation to define linear maps
\begin{align*}
\Pi \mc S_- \otimes W^\vee  &\to \op{Vect} (\PT^{W} ) \\
\Pi \mc S_+ \otimes W & \to \op{Vect} (\PT^{W} ).
\end{align*}
Every endomorphism of the bundle $\Oo(1) \otimes \left( \mc S_- \oplus \Pi W \right)$ yields a vector field on $\PT$.   Clearly, endomorphisms of the super vector space $\mc S_- \oplus \Pi W$ yield endomorphisms of this super vector bundle.   Since 
$$\Pi W^\vee \otimes \mc S_- \subset \op{End} ( \mc S_- \oplus \Pi W) $$
we get a map
$$
\Pi W^\vee \otimes \mc S_- \to \op{Vect} ( \PT^{W} ) .
$$ 

Next, we will construct the map
$$\Pi \mc S_+ \otimes W \to \op{Vect} ( \PT^{W} ).$$
Any section of the super vector bundle $\Oo(1) \otimes \left( \mc S_- \oplus \Pi W \right)$ yields a vector field on the total space of this bundle.  This gives us a map 
$$
H^0 \left( \P(\mc S_+), \Oo(1) \otimes \left(  \mc S_- \oplus \Pi W \right) \right) \to \op{Vect} ( \PT^{W} ) .
$$
In particular, we get a map
$$
H^0 (\P(\mc S_+),  \Oo(1) \otimes \Pi W ) = \Pi \mc S_+ \otimes W \to \op{Vect} ( \PT^{W}) .
$$
It's not difficult to verify that these maps satisfy the relations necessary to define an action of the super-translation Lie algebra $T^{W}$ on $\PT^{W}$, in a way compatible with the map $\PT^{\mscr{N}=1} \to \R^4$.   
\end{proof}

\subsection{}
Now we are ready to define the anti-self-dual versions of the $\mscr{N}=1$ and $\mscr{N}=2$ supersymmetric gauge theories. 

\begin{definition}
The anti-self-dual $\mscr{N}=1$ (respectively, $\mscr{N}=2$) gauge theory is the cotangent field theory for the super elliptic moduli problem on $\PT$ describing holomorphic bundles on the $\mscr{N}=1$ (respectively, $\mscr{N}=2$) twistor space $\PT^{\mscr{N}=1}$ (respectively, $\PT^{\mscr{N}=2}$). 
\end{definition}
Since the super-translation Lie algebra $T^{\mscr{N}=k}$ acts on the super-twistor space $\PT^{\mscr{N}=k}$, it is clear that this construction yields a super-symmetric field theory in the sense we described earlier   Further,  this theory has the largest possible $R$-symmetry $GL(k,\C)$.

\begin{remark}
I should emphasize that, in practise, the entire $R$-symmetry group will not act on the quantum theory.  As Anton Kapustin explained to me, in the ordinary $\mscr{N}=1$ gauge theory with no matter fields, the classical $R$-symmetry group $\C^\times$ has an anomaly at the quantum level: there,  only a discrete cyclic subgroup acts.
\end{remark}

Recall that $\Oo_{\mscr{N}=k}$ indicates the pushforward of the structure sheaf of $\PT^{\mscr{N}=k}$ to $\Oo_{\PT}$.  Let 
$$
\Oo^{\vee}_{\mscr{N}=k} = \Hom_{\Oo_{\PT}} (\Oo_{\PT}^{\mscr{N}=k}, \Oo_{\PT} ).
$$

By definition, the elliptic Lie algebra describing this cotangent theory is
$$
T^\ast[-1] \mscr{L}_{\mscr{N}=k} = \Omega^{0,\ast}\left( \PT,  \Oo_{\mscr{N}=k} \otimes_{\Oo_{\PT}} \g_P \oplus \Oo^\vee_{\mscr{N}=k} \otimes_{\Oo_{\PT}} K_{\PT} \otimes_{\Oo_{\PT}} \g^\vee_P \right).
$$
Here, $k = 1,2$ as above. The full $\mscr{N}=1$ or $\mscr{N}=2$ gauge theory is, as is described in \cite{BoeMasSki07}, given by a $1$-parameter family of deformations of the anti-self-dual theory.   

\subsection{}
The self-dual theories on $\R^4$ are obtained by dimensional reduction from the twistor space to $\R^4$. It is natural to ask for an explicit finite-rank model for the theories on $\R^4$.  This analysis was performed in detail in \cite{BoeMasSki07}, we will only sketch some aspects.    If $\mscr{L}^{N=k}$ is the elliptic Lie algebra on $\mbb{PT}$ describing a supersymmetric field theory, then we can push forward to $\R^4$ to find a sheaf $\pi_\ast \mscr{L}^{N=k}$ of dg Lie algebras on $\R^4$.  Let 
$$\mscr{L}_{\R^4}^{N=k} \subset \pi_\ast \mscr{L}^{N=k}$$
be the sub-sheaf constisting of those sections which are haromic on every twistor fibre. Hodge theory tells us that this inclusion is a quasi-isomorphism, and homological perturbation theory gives us an $L_\infty$-structure on $\mscr{L}_{\R^4}^{N=k}$ which is quasi-isomorphic to $\pi_\ast \mscr{L}^{N=k}$.  

The fibre of $\mscr{L}^{N=k}_{\R^4}$ can be calculated by calculating harmonic sections of various smooth bundles on $\mbb{P}^1$. Let's consider, for example, the part of $\mscr{L}^{N=k}_{\R^4}$ arising from $\Omega^{0,\ast}(\mbb{PT}, \Oo_{\PT})$.  Elements of this complex restrict, on $\mbb{P}^1$, to the Dolbealut complex of $\mbb{P}^1$ with coefficients in the exterior algebra of the complex conjugate of the conormal bundle.  Thus, we find that this summand of $\mbb{L}^{\mc{N}=k}_{\R^4}$ contributes
$$
H^\ast_{\dbar}( \mbb{P}^1, \Sym^\ast ( \Oo(1) [-1] \oplus \Oo(1)[-1]) ). 
$$
In degree $1$, we find $H^0(\mbb{P}^1, \Oo(1)^2) = \C^4$; under the action of $\op{Spin}(4)$, this is the vector representation.  So, as expected, we find that $\mscr{L}^{N=k}_{\R^4}$ contains $\Omega^1_{\R^4}$ in cohomological degree $1$. Similarly, in degree $2$, we find $H^0(\mbb{P}^1, \Oo(2))$, which is the fibre of $\Omega^2_{-}$ at the orgin in $\R^4$.  In sum, we have
$$
\cinfty(\R^4) \otimes H^\ast_{\dbar} (\mbb{P}^1, \Sym^\ast (\Oo(1)[-1] \oplus \Oo(1)[-1] ) = \Omega^0 \oplus \Omega^1[-1] \oplus \Omega^2_{-} [-2].
$$
Similarly, find that $\Omega^{0,\ast}(\mbb{PT}, \Oo(-4)$ contributes
$$
H^\ast_{\dbar}(\mbb{P}^1, \Oo(-4) \otimes \Sym^\ast ( \Oo(1)[-1] \oplus \Oo(1)[-1] )).
$$
We can calculate that 
$$
\cinfty(\R^4) \otimes H^\ast_{\dbar}(\mbb{P}^1, \Oo(-4) \otimes \Sym^\ast ( \Oo(1)[-1] \oplus \Oo(1)[-1] )) = \Omega^2_-[-1] \oplus \Omega^3[-2] \oplus \Omega^4[-3].
$$
In this way we find that $\Omega^{0,\ast}(\mbb{PT}, \Oo \oplus \Oo(-4))$ contributes precisely the field content of self-dual Yang-Mills on $\R^4$.  

One can go a bit further, and calculate the solutions to the linearized equations of motion on $\R^4$, using basic results about the Penrose transform.  For example, in the $N=1$ case, we see that the cohomology of the fermionic part of $\mscr{L}^{N=k}_{\R^4}$ on an open subset $U$ of $\R^4$ is $H^\ast(\pi^{-1}U, \Oo(-1)) \oplus H^\ast(\pi^{-1}(U), \Oo(-3))$.  Penrose tells us that $H^\ast(\pi^{-1}U, \Oo(-1))$ is isomorphic to the cohomology of the complex
$$
\mc{S}_+(U) \xto{\dirac} \mc{S}_-(U)
$$
where $\mc{S}_+(U)$ is in degree $1$ and $\mc{S}_-(U)$ is in degree $2$.  Similarly, $\Oo(-3)$ contributes 
$$
\mc{S}_-(U) \xto{\dirac} \mc{S}_+(U).
$$
From this, we conclude that the fermionic part of $H^1(\L_{\R^4}^{\mscr{N}=1}(U))$ consists of (adjoint-valued) spinors $\psi_{+}, \psi_-$ on $U$ satisfying $\dirac \psi_- = \dirac \psi_+ = 0$.   This is what we find for the linearized equations of motion of the fermionic fields of $\mscr{N}=1$ gauge theory, as expected.  

\subsection{}
So far, we have constructed the $\mscr{N}=1, 2$ gauge theories on twistor space.  Next, we will construct the $\mscr{N}=4$ gauge theory.  The construction of this theory does not quite follow the patter set by the $\mscr{N}=1,2$ theories.  

We have seen that $\PT^{\mscr{N}=4}$ has an action of $T^{\mscr{N}=4}$.  Let $\mscr{L}_{\mscr{N}=4}$ be the elliptic Lie algebra describing bundles on $\PT^{\mscr{N}=4}$, near a particular $G$-bundle $P$ pulled back from $\PT$.  Thus, if $\g_P$ is the adjoint bundle of Lie algebras on $\PT$, we have
$$
\mscr{L}_{\mscr{N}=4} = \Omega^{0,\ast}( \PT, \Oo_{\mscr{N}=4} \otimes_{\Oo_{\PT}} \g_P  ).
$$
With the $\mscr{N}=2$ and $\mscr{N}=1$ theories, we took the cotangent theory to the moduli of $G$-bundles on the appropriate super-twistor space.  With the $\mscr{N}=4$ theory, we don't need to do this.
\begin{lemma}
$\mscr{L}_{\mscr{N}=4}$ theory has an invariant pairing of cohomological degree $-3$ (and fermion degree $0$), so that the corresponding elliptic moduli problem is a classical field theory. 
\end{lemma}
\begin{proof}
Indeed,
$$
\Oo_{\mscr{N}=4} = \Sym^\ast \left(  \Pi \Oo_{\PT}(-1) \otimes \C^4  \right) .
$$
Thus, there's a map
$$
\Oo_{\mscr{N}=4} \to \wedge^4( \Oo(-1)_{\PT} \otimes \C^4)   = \Oo_{\PT}(-4) = K_{\PT}
$$
of super vector bundles on $\Oo_{\PT}$.

The invariant pairing on the Lie algebra $\g_P$ thus gives us a map of super vector bundles on $\PT$
$$
\left(\Oo_{\mscr{N}=4} \otimes_{\Oo_{\PT}} \g_P\right)^{\otimes 2} \to K_{\PT}.
$$
Composing with the isomorphism
$$
\Omega^{0,3} ( \PT, K_{\PT} ) \iso \op{Dens}( \PT)
$$
gives the desired invariant pairing. 
\end{proof}
Of course, this lemma is simply observing that $\PT^{\mscr{N}=4}$ is a super Calabi-Yau manifold.  The anti-self-dual $\mscr{N}=4$ theory is the holomorphic Chern-Simons theory on this super Calabi-Yau.  

It is clear that the action of the super-translation Lie algebra $T^{\mscr{N}=4}$ on the $\mscr{N}=4$ theory is compatible with the action of the complexified Euclidean Lie algebra, so this is indeed a supersymmetric field theory. 

The $R$-symmetry group of the $\mscr{N}=4$ theory is \emph{not} $GL(4,\C)$, because we need to choose a trivialization of $\op{det} \C^4$ in order to write down the symplectic pairing.  Instead, the $R$-symmetry group is $SL(4,\C)$. 

\begin{remark}
Our conventions are slightly different to usual conventions in the physics literature.  Our theories have complex space of fields, and we are not concerned about reality conditions.  Thus, our $R$-symmetry group is $SL(4,\C)$.  In the physics literature, reality conditions are considered to be more important, so the $R$-symmetry group for the $\mscr{N}=4$ theory is usually taken to be $SU(4)$.  
\end{remark}

\begin{remark}
One can ask why we have not considered a theory with $\mscr{N}=3$ supersymmetry.  We could follow the pattern set by the $\mscr{N}=1$ and $\mscr{N}=2$ theories, and define the $\mscr{N}=3$ theory to be the cotangent theory to the moduli of holomorphic bundles on the graded manifold 
$$\Oo(1)[1]^{\oplus 2} \oplus \Oo(1)[-1]  \to \PT.$$
This cotangent theory will recover the $\mscr{N}=4$ theory.  Thus, the $\mscr{N}=3$ theory already has $\mscr{N}=4$ supersymmetry. 

\end{remark}

\section{Twisting supersymmetric field theories}
Many of the quantum field theories of interest in mathematics arise as twistings of supersymmetric field theories.  In this section I will describe the concept of twisting of a supersymmetric field theory, and analyze the twists of the supersymmetric gauge theories introduced in the previous section. 

Recall that a field theory on $\R^4$ (or on some other space naturally associated to $\R^4$, such as the twistor space $\PT$) has $\mscr{N}=k$ supersymmetry if it is equipped with an action of the super Lie algebra
$$
T^{\mscr{N}=k} \iso T^{W} = \Pi \left( \mc S_- \otimes W^\vee \oplus  \mc S_+ \otimes W \right)  \oplus V_\C,
$$
where $W$ is a complex vector space of dimension $k$ and where $V_\C = \C^4$ denotes the complexification of the abelian Lie algebra of translations on $\R^4$.

We also assume that our supersymmetric field theory is equipped with an action of an $R$-symmetry group $G_R \subset GL(W)$, in a way compatible with the natural action of $G_R$ on $T^{W}$.  

The basic idea of twisting is as follows.  Suppose we have an odd element $Q$ of $T^{W}$ which satisfies $[Q,Q] = 0$.    A physicist would say that the twisted theory is obtained by treating $Q$ as a BRST operator (that is, as a differential).    In our set-up, field theories are already differential-graded objects.  Therefore, we should imagine constructing the twisted theory by adding $Q$ to the differential on the elliptic moduli problem defining our classical field theory.  This naive idea leads to problems, because our differential should be of fermion degree $0$ and cohomological degree $1$, whereas $Q$ is of fermion degree $1$ and cohomological degree $0$.

In order to do this construction properly, we need some additional data.
\begin{definition}
\emph{Twisting data} for a supersymmetric field theory consists of a group homomorphism $\rho : \C^\times \to G_R$, an odd element $Q \in T^{W}$, such 
\begin{enumerate}
\item $[Q,Q] = 0$.
\item $\rho(t) (Q) = t Q$ for $t \in \C^\times$, and $\rho(t)$ indicates the action of $\C^\times$ on $T^{W}$ arising from the natural action of $G_R \subset GL(W)$ on $T^{W}$.
\end{enumerate}
\end{definition}

Suppose that $\mc{M}$ is a classical field theory on $\R^4$ (or on $\PT$) with $\mscr{N}=k$ supersymmetry.  Thus, $\mc{M}$ is a super elliptic moduli problem on $\R^4$ (or on $\PT$), with an invariant pairing of degree $-1$, and additionally equipped with a compatible action of $\op{Spin}(4)$ and of the super-Lie algebra $T^{W}$.  Let us suppose also that $\mc{M}$ is acted on by the $R$-symmetry group $G_R$, in a way compatible with the $G_R$ action on $T^{W}$. 

Then, a choice of twisting data $(\rho,Q)$ as above gives rise (via $Q$) to an action on $\mc{M}$ of the Abelian super-Lie algebra $\Pi \C$, and (via $\rho$)  to an action of the group $\C^\times$.  These two actions are compatible, and can be viewed as an action of the super-Lie group $\C^\times \ltimes \Pi \C$.  (Since $\Pi \C$ is a nilpotent Lie algebra, we can think of it as a super-Lie group).

\subsection{}
The twisting construction can be defined for any field theory  $\mc M$ on a manifold $X$ with an action of the super-group $\C^\times \ltimes \Pi \C$, covering the trivial action on $X$.   Morally, the procedure to construct a twisted theory is as follows. 
\begin{enumerate}
\item First take the homotopy fixed points of $\mc M$ under the action of $\Pi \C$. This gives a field theory over the Chevalley-Eilenberg Lie algebra cochain algebra $C^\ast(\Pi \C) = \C[[t]]$,  where $t$ has cohomological degree $1$ and fermion degree $1$.
\item Invert the parameter $t$ in our base ring, to give a field theory over $\C((t))$.  
\item Now, take $\C^\times$ invariants.  
\end{enumerate} 
Let us now explain in detail how to do this at the formal level, where $\mc M$ is a formal elliptic moduli problem, corresponding to an elliptic super $L_\infty$ algebra $\mscr{L}$,  equipped with an invariant pairing of cohomological degree $-3$.    As above, suppose that $\mscr{L}$ is equipped with an action of the super-group $\C^\times \ltimes \Pi \C$, covering the trivial action on $X$.  Thus, the sheaf $B \mscr{L}$ on $X$ of formal moduli problem is equipped with an action of $\C^\times \ltimes \Pi \C$ in a way preserving the base point. 
 
Let us assume that the action of $\C^\times$ on $\mscr{L}$ is linear.  Thus, this $\C^\times$ action will give $\mscr{L}$ an additional grading. 

Let us think of the action of $\Pi \C$ on $\mscr{L}$ as a Lie algebra action, and let us denote a basis of $\Pi \C$ by $Q$   Recall that, for us, an action of a Lie algebra $\g$ on a field theory is described by a family of field theories over $C^\ast(\g)$.  Thus, the action of $\Pi \C$ on $\mscr{L}$ is described by a family of elliptic $L_\infty$ algebras $\mscr{L}^{\Pi \C}$ over the base ring $C^\ast  ( \Pi \C ) = \C[[t]]$ where $t$ has a cohomological degree $1$, and fermion degree $1$.  This $L_\infty$ algebra (viewed as an $L_\infty$ algebra over $\C$) is the homotopy fixed points for the $\Pi \C$ action on $\mscr{L}$. 

Concretely, if the action of $\Pi \C$ on the $L_\infty$ algebra $\mscr{L}$ happens to be linear (as is the case for the field theories we are interested in), then 
$$
\mscr{L}^{\Pi \C} = \mscr{L}[[t]]
$$
with differential $\d_{\mscr{L}} + t Q$.    $\mscr{L}^{\Pi C}$ is the homotopy fixed points of $\mscr{L}$ under the action of $\Pi \C$.

Next, let us invert the parameter $t$, to giving us the elliptic $L_\infty$ algebra 
$$\mscr{L}^{\Pi \C} [t^{-1} ] = \mscr{L}^{\Pi \C} \otimes_{\C[[t]]} \C((t)).$$
Finally, we take the $\C^\times$ fixed point of this object. That is, the twisted theory is 
$$
\mscr{L}^{\op{Twisted}} = \left( \mscr{L}^{\Pi \C} [t^{-1} ]  \right) ^{\C^\times}.
$$
The invariant pairing on $\mscr{L}^{\op{Twisted}}$ is defined as follows.  Note that $\mscr{L}^{\Pi \C}[t^{-1}]$ has an invariant pairing valued in $\C((t))$.  This restricts to a pairing on $\mscr{L}^{\op{Twisted}}$ with values in the $\C^\times$-invariants of $\C((t))$, which is $\C$. 

\subsection{}
Let us make this construction more explicit.  Let us assume for a moment that the action of $Q \in \Pi \C$ on $\mscr{L}$ is linear (as will be the case in the examples of interest).  Let $\mscr{L}^k \subset \mscr{L}$ denote the subspace on which $s \in \C^\times$ acts by $s^k$.   For simplicity (so we don't have to discuss completions) let us assume that $\mscr{L}^k = 0$ for all but finitely many $k$.    Note that the operator $Q$ maps $\mscr{L}^k$ to $\mscr{L}^{k+1}$.  Then, we can identify
$$
\left( \mscr{L}^{\Pi \C} [t^{-1}] \right)^{\C^\times} = \oplus_{k \in \Z} \left( \Pi^{k}  \mscr{L}^k [-k] \right).
$$
Thus, $\mscr{L}^k$ is shifted up by cohomological degree $k$, and also has a shift of fermion degree by $k \op{mod} 2$.  This shift comes from the fact that $t^k$ has fermion degree $k \op{mod} 2$ and cohomological degree $k$.

The differential on this complex is arises from the ordinary differential $\d_{\mscr{L}}$ on $\mscr{L}$, together with the operator $Q$, which maps each $\mscr{L}^k$ to $\mscr{L}^{k+1}$.  Because of the shift of fermion and cohomological degrees, the operator $Q$ is now of fermion degree $0$ and cohomological degree $1$, as is required for a differential. 

\subsection{}
Let $\mscr{L}$ be a field theory on $X$ as above, with an action of $\C^\times \ltimes \Pi \C$.  Then there is a filtration on $\mscr{L}^{\op{Twisted}}$ by setting
$$
F^i \mscr{L}^{\op{Twisted}} = \oplus_{k \ge i} \mscr{L}^k t^k \subset \mscr{L}^{\op{Twisted}}.
$$
The associated graded coincides with $\mscr{L}$, except for a change of grading, where the summand $\mscr{L}^k$ is shifted by cohomological degree $k$ and by fermion degree $k$.  In other words, the associated graded coincides with $\mscr{L}$ twisted by the given action of $\C^\times$ and the trivial action of $\Pi \C$. 

In \cite{CosGwi11}, we analyze classical and quantum field theories in terms of the associated factorization algebra.  The factorization algebra for the classical field theory $\mscr{L}$ assigns, to an open subset $U \subset X$, the Chevalley-Eilenberg cochains $C^\ast(\mscr{L}(U))$.  Let us denote this factorization algebra by $\op{Obs}^{cl}_{\mscr{L}}$: it is the factorization algebra of \emph{classical observables} of the field theory. 

The group $\C^\times \ltimes \C$ acts on the factorization algebra $\Obs^{cl}_{\mscr{L}}$, and one can construct the factorization algebra for the twisted theory directly from this action.  Indeed, it is not difficult to verify that $\Obs^{cl}_{\mscr{L}^{\op{Twisted}}}$ can be computed by
$$
\Obs^{cl}_{\mscr{L}^{\op{Twisted}}} (U) = \left( \Obs^{cl}_{\mscr{L}}(U)^{\Pi \C} \otimes_{\C[[t]]} \C((t))  \right)^{\C^\times}.
$$
In other words, the observables for the twisted theory are obtained from the observables of the untwisted theory by first, taking the homotopy fixed points with respect to $\Pi \C$; then inverting the parameter $t$; and finally taking $\C^\times$-invariants.   (In this expression, a little care is needed with completions of topological vector spaces in order to make the equality exact). 

It follows from this expression that there is a spectral sequence computing the cohomology of the observables of the twisted theory from those of the untwisted theory, as follows.   Let $H^{i,j,k}  (\Obs^{cl}_{\mscr{L}}(U)))$ denote the cohomology in cohomological degree $i$, fermion degree $j$ and weight $k$ under the $\C^\times$ action.  Then we have a spectral sequence
$$
H^{i,j,k} (\Obs^{cl}_{\mscr{L}}(U)) \Rightarrow H^{i+k, j+k} ( \Obs^{cl}_{\mscr{L}^{\op{Twisted}}} (U)).
$$

If one can construct a quantization of the classical theory described by $\mscr{L}$, and if this quantization is compatible with the action of $\C^\times \ltimes \Pi \C$, then one automatically has a quantization of the twisted theory.  Further, the relation between observables of the twisted and untwisted theories described above holds at the quantum level.   It follows that one has a similar spectral sequence relating quantum observables of the twisted and untwisted theories.  

\section{Twisted supersymmetric gauge theories}
Next, we will give a detailed description of the theories obtained by twisting our anti-self-dual super-symmetric gauge theories on twistor space.    We will see that the twisting procedure yields natural ``holomorphic gauge theories'' on $\C^2$. 

We will start with the case of $\mscr{N}=1$ supersymmetry.  Let us choose an element $Q \in \mc{S}_+$.  Recall that the choice of such an element yields a complex structure on the linear space $\R^4$; indeed, the stabilizer of $Q$ in $\op{Spin}(4) = SU(2) \times SU(2)$ is $SU(2)$, so that $Q$ provides a reduction of structure group to $SU(2)$. 

Since the $R$-symmetry group of the $\mscr{N}=1$ theory is $\C^\times$, there is a unique map $\rho : \C^\times \to G_R^{\mscr{N}=1}$ under which $Q$ has weight $1$. 

Recall that the $\mscr{N}=1$ anti-self-dual theory is the cotangent theory to the elliptic moduli problem of holomorphic bundles on $\PT^{\mscr{N}=1} = \Pi \Oo(1) \to \PT$.  As before, let 
$$
\Oo_{\mscr{N}=1} =  \Pi \Oo_{\PT}(1) \oplus \Oo_{\PT}
$$
be the structure sheaf of $\PT^{\mscr{N}=1}$, viewed as a sheaf of super commutative algebras on $\PT$.  

Recall that 
$$\PT = \Oo(1)^{\oplus 2} \to \P(\mc S_+).$$
Thus, the element $Q \in \mc S_+$ yields a section on $\P(\mc S_+)$ of $\Oo(1)$, and so a section of $\Oo_{\PT}(1)$.  Thus, it yields a map
$$
Q : \Oo_{\PT}(-1) \to \Oo_{\PT}.
$$

Recall that the twisting procedure described above involves a shift in grading.  In this example, we shift $\Oo_{\PT}(-1)$ from being in fermion degree $1$ and cohomological degree $0$ to being in fermion degree $0$ and cohomological degree $-1$. 

Thus, the twisting procedure yields the sheaf $\Oo_{\mscr{N}=1}^Q$ of differential graded algebras
$$
\Oo_{\mscr{N}=1}^Q = \Oo_{\PT}(-1)[1] \xto{Q} \Oo_{\PT}.  
$$

The twisted $\mscr{N}=1$ theory is the cotangent theory to the elliptic moduli problem for holomorphic bundles on the differential graded complex manifold $\PT$ with structure sheaf $\Oo_{\mscr{N}=1}^Q$.  Explicitly, near the trivial bundle principal $G$-bundle on $\PT$, the twisted theory is the cotangent theory associated to the elliptic Lie algebra
$$
\Omega^{0,\ast} ( \PT,  \Oo_{\mscr{N}=1}^Q \otimes \g ).
$$

Note that this theory is concentrated entirely in fermion degree $0$, so we can consider it to be an ordinary (non-super) elliptic moduli problem.  This will be a feature of all the twisted theories we consider. 

We showed above that, in general, the twisted theory lives in a $\C^\times$-equivariant family of theories over $\mbb{A}^1$ whose value at the central fibre is the untwisted theory (with a different grading).    For this example, the family of theories is defined by the family of differential graded algebras
$$
\Oo_{\PT}(-1)[1] \xto{t Q} \Oo_{\PT}.  
$$
This has a natural $\C^\times$ action (of weight $0$ on $\Oo_{\PT}$ and weight $-1$ on $\Oo_{\PT}(-1)[1]$) making it into a $\C^\times$-equivariant family of theories over $\mbb{A}^1$. 

\subsection{}
Observe that there is a quasi-isomorphism of sheaves of algebras on $\PT$
$$
\Oo_{\mscr{N}=1}^Q \simeq \Oo_{Z(Q) } 
$$
where 
$$
Z(Q) \subset \PT
$$
is the copy of $\C^2$ realized as the zero locus of $Q \in \Gamma( \PT, \Oo(1))$.  The twistor projection $\PT \to \R^4$ yields a diffeomorphism
$$
Z(Q) \to \R^4.
$$
The complex structure on $\R^4$ induced by this diffeomorphism is, of course, the complex structure associated to $Q \in \mc S_+$. 

It follows that there is a homotopy equivalence of sheaves of dg Lie algebras 
$$
\pi_\ast \Omega^{0,\ast}( \Oo_{\mscr{N}=1}^Q  \otimes \g ) \simeq \Omega^{0,\ast} ( \Oo_{Z(Q)} \otimes \g),
$$
where $\pi : \PT \to \R^4$ is the twistor fibration. 

Thus, we have shown the following.
\begin{lemma}
The twisted self-dual $\mscr{N}=1$ theory is equivalent to the cotangent theory for the moduli space of holomorphic bundles on $\C^2 = \R^4$, holomorphic with respect to the complex structure determined by the choice of spinor $Q \in \mc S_+$.  
\end{lemma}
Note, however, that in this presentation the relationship between the twisted theory and the original theory has been obscured.    There is no filtration on the dg Lie algebra $\Omega^{0,\ast} ( \Oo_{Z(Q)} \otimes \g)$ corresponding to that on the homotopy equivalent dg Lie algebra $\pi_\ast \Omega^{0,\ast}( \Oo_{\mscr{N}=1}^Q  \otimes \g) $; the filtration can only be written down after passing to a larger (but homotopy equivalent) complex. 
 
\subsection{}
Let us now consider the twists of the $\mscr{N}=2$ and $\mscr{N}=4$ theories.   In each case we need to choose twisting data (for the $\mscr{N}=1$ theory, all choices are essentially equivalent).  

Recall that the $R$-symmetry group of the $\mscr{N}=2$ theory is $GL_2(\C)$.  Part of the twisting data is a group homomorphism $\rho : \C^\times \to GL_2(\C)$.  We will choose the homomorphism given by the matrix
$$
\begin{pmatrix}
t & 0 \\
0 & t
\end{pmatrix}.
$$

Recall that the odd part of $T^{\mscr{N}=2}$ is $\mc S_+ \otimes \C^2 \oplus \mc S_- \otimes \C^2$.   The other part of our twisting data is an odd element $Q$ of $T^{\mscr{N}=2}$ which satisfies $[Q,Q] = 0$ and which is of weight $1$ under the $\C^\times$ action.   The $\C^\times$ action gives $\mc S_+ \otimes \C^2$ weight $1$ and $\mc S_- \otimes \C^2$ weight $-1$.  Thus, the space of possible $Q$'s is $\mc S_+ \otimes \C^2$.  Let $\alpha \in \mc S_+$ be a spinor (corresponding to a complex structure on $\R^4$).  We will take our $Q$ to be a decomposable element
$$
Q = \alpha \otimes \begin{pmatrix} 1 \\ 0\end{pmatrix}.
$$

\begin{remark}
Note that the twisting data we choose is not generic.  Indeed, once we have chosen our homomorphism $\rho: \C^\times \to GL_2(\C)$, we see that possible $Q$'s is $\mc S_+ \otimes \C^2$.  The group $GL(2) \times \op{Spin}(4,\C)$ acts on this space; the action factors through $GL_2(\C) \times SL_2(\C)$.   Two twists which are related by an element of this symmetry group are equivalent. We are choosing our $Q$ to be a decomposable tensor, and thus in a lowest-dimensional non-zero orbit.   

We call twists of this form \emph{minimal twists}.  We will realize twists by more generic elements as being obtained by further twisting the minimally twisted theory.   
\end{remark}

We find that the twisted $\mscr{N}=2$ theory is the cotangent theory to the moduli of holomorphic $G$-bundles on the differential graded complex manifold with structure sheaf
$$
\Oo_{\PT^{\mscr{N}=2}}^{Q} = \Sym^\ast\left(  \Oo_{\PT}(-1) [1] \oplus \Oo_{\PT}(-1) [1]  \right)
$$
with differential induced from the map $Q : \Oo_{\PT}(-1)[1] \to \Oo_{\PT}$.

An argument similar to that we applied for the $\mscr{N}=1$ theory now shows the following. 
\begin{lemma}
The twisted $\mscr{N}=2$ theory is the cotangent theory for the elliptic moduli problem of principal $G$-bundles on the graded complex manifold $\underline{\C}[-1] \to \C^2$, where $\underline{\C}[-1]$ refers to the trivial line bundle put in degree $1$.
\end{lemma}
\subsection{}
Next, let us consider the $\mscr{N}=4$ theory.  The $R$-symmetry group in this case is $SL_4(\C)$.  We choose our homomorphism $\rho : \C^\times \to SL_4(\C)$ to be given by the matrix
$$
\begin{pmatrix}
t & 0 &0 &0 \\
0 & t & 0 & 0\\
0 & 0 & t^{-1} & 0 \\
0 & 0 & 0 & t^{-1}
\end{pmatrix}.
$$
The space of possible $Q$'s of weight $1$ under this twist is $\mc S_+ \otimes \C^2 \oplus \mc S_- \otimes \C^2$.  We take our $Q$ to be, as in the $\mscr{N}=2$ theory, a decomposable tensor
$$
Q = \alpha \otimes \begin{pmatrix} 1 \\ 0\end{pmatrix}
$$
for some $\alpha \in \mc S_+$. 

The following lemma is easy to verify.
\begin{lemma}
The twisted $\mscr{N}=4$ theory is the cotangent theory  for the elliptic moduli problem describing holomorphic $G$-bundles on the graded complex manifold $\underline{\C}^2[1] \to \C^2$.  
\end{lemma}

\subsection{}
It is worthwhile describing these field theories completely explicitly.  For the $\mscr{N}=1$ theory, near the trivial principal $G$-bundle, the elliptic Lie algebra with invariant pairing is 
$$
\mscr{L}_{\mscr{N}=1}^{Q} = \Omega^{0,\ast} ( \C^2,   \g \oplus \g^\vee[-1]),
$$
where $\g \oplus \g^\vee[-1]$ is made into a Lie algebra by the action of $\g$ on $\g^\vee[-1]$.  The invariant pairing is
$$
\ip{\phi \otimes A, \psi \otimes B } = \int_{\C^2} \phi \psi \d z_1 \d z_2 \ip{A,B}_{\mf{g}}
$$
where $A \in \g$, $B \in \g^\vee$, and $\phi, \psi \in \Omega^{0,\ast}(\C^2)$. 

The $\mscr{N}=2$ Lie algebra is
$$
\mscr{L}_{\mscr{N}=2}^{Q} = \Omega^{0,\ast} ( \C^2,   \g[\eps] \oplus \g^\vee[-2][\eps] ),
$$
where $\eps$ is a parameter of degree $-1$, and the pairing between $\g[\eps]$ and $\g^\vee[\eps]$ is given by combining the pairing between $\g$ and $\g^\vee$ with the trace map on $\C[\eps]$ defined by 
\begin{align*}
\op{Tr} : \C[\eps] &\to  \C\\\
\op{Tr}(\eps) &= 1.
\end{align*}

Finally, let us describe the $\mscr{N}=4$ Lie algebra.  Consider the algebra $A = \C[\eps_1,\eps_2]$, where $\eps_i$ are of degree $1$, equipped with the trace map of degree $-2$ defined by
$$
\op{Tr}( \eps_1 \eps_2 ) = 1.
$$
The $\mscr{N}=4$ Lie algebra can be written as 
$$
\mscr{L}_{\mscr{N}=4}^{Q} = \Omega^{0,\ast} ( \C^2,   \g \otimes A  \oplus \g^\vee [1] \otimes A ).
$$

Note that the $\mscr{N}=1,2$ and $4$ twisted theories all arise from taking a graded-commutative Frobenius algebra $A$, with a trace of degree $-k$, and then considering 
$$
\Omega^{0,\ast}( \C^2, \g \otimes A \oplus \g^\vee[k-1] \otimes A).  
$$

\subsection{}
A supersymmetric field theory, by definition, is one equipped with an action of a certain super Lie algebra.  In this section we will see that a twist of a twist of a supersymmetric gauge theory has some residual symmetries. 

Recall that a theory with $\mscr{N}=k$ supersymmetries has an action of the super Lie algebra 
$$
\left( \g_R \oplus \mf{sl}_2(\C) \oplus \mf{sl}_2(\C)  \right) \ltimes T^{\mscr{N}=k},
$$
where $\g_R$ is the Lie algebra of the appropriate $R$-symmetry group (which is $SL_4(\C)$ in the case $k = 4$, $GL_2(\C)$ in the case $k = 2$, or $\C^\times$ in the case $k = 1$). 

Let $(\rho, Q)$ be twisting data as above. Let  $\g_R^{\rho} \subset \g_R$ be the sub Lie algebra fixed under the restriction of the adjoint $G_R$ action under the homomorphism $\rho : \C^\times \to G_R$.  

Recall that the odd part of $T^{\mscr{N}=k}$ is
$$
\mc S_+ \otimes \C^k \oplus \mc S_- \otimes \C^k.
$$
The group $\C^\times$ acts on this space via the homomorphism $\rho : \C^\times \to G_R \subset GL_k(\C)$.  (Recall that the copy of $\C^k$ tensored with $\mc S_-$ is dual to the copy which is tensored with $\mc S_+$). 

In our examples, the space $\mc S_+ \otimes \C^k \oplus \mc S_- \otimes \C^k$ decomposes into weight $1$ and weight $-1$ subspaces under this $\C^\times$ action.   Let us introduce a $\Z$-grading on the space of odd elements of $T^{\mscr{N}=k}$, by saying that elements of weight $1$ are in cohomological degree $1$, and elements of weight $-1$ are in cohomological degree $-1$. 

Thus, with this grading, we find a $\Z$-graded Lie algebra
$$
\left( \g_R^{\C^\times} \oplus \mf{sl}_2(\C) \oplus \mf{sl}_2(\C) \right) \ltimes T^{\mscr{N}=k}.
$$ 
The element $Q \in T^{\mscr{N}=k}$ is, by assumption, in cohomological degree $1$. 

It is clear that every element of this $\Z$-graded Lie algebra which commutes with $Q$ acts on the twisted theory.  Further, symmetries of the form $[Q,X]$ act homotopically trivially.  Thus, we see that
\begin{lemma}
A theory twisted by a supercharge $Q$ acquires an action of the differential graded Lie algebra $\left( \mf{sl}_2 \oplus \mf{sl}_2 \oplus \g_R^{\C^\times}\right) \ltimes T^{\mscr{N}=k}$, with differential $Q$. 
\end{lemma}
Later we will see that the $\mscr{N}=4$ theory admits further twists.  These twists arise from elements of $H^1$ of this differential graded Lie algebra. 

\subsection{}
So far, we have described the anti-self-dual supersymmetric gauge theories via the twistor-space formulation, and we have described the twisted theories arising from the anti-self-dual theories.  We have not, however, described the full supersymmetric gauge theory. 
be
As explained in \cite{BoeMasSki07}, the full supersymmetric gauge theory also has a twistor space description: it is obtained by deforming the action for the anti-self-dual theory by adding a certain explicit term to the action.

Because the full supersymmetric gauge theory is acted on by the same supersymmetry group as the anti-self-dual theory, the twisting construction described above can be applied to the full theory.  \begin{proposition}
The deformation of the anti-self-dual theory into the full supersymmetric gauge theory does not change the minimally-twisted theory.
\end{proposition}

The minimally-twisted supersymmetric gauge theory we described above has an action of an $R$-symmetry group, and is also invariant under translations and the group $\op{GL}_2(\C)$ acting on $\C^2$.  The deformation of this theory we are interested in has the same symmetry group, inherited from the symmetries of the full supersymmetric gauge theory. 

In order to prove the result, we will verify that the minimally twisted theories we analyzed earlier admit no deformations with these symmetries.  
\begin{theorem}
Let $G$ be a simple algebraic group, and $\g$ its Lie algebra.   Let us consider the minimally-twisted $\mscr{N}=1,2,4$ supersymmetric gauge theories on $\C^2$, perturbing around the trivial $G$-bundle.   

Then, the cohomology of the complex of translation-invariant local functionals for the $\mscr{N}=1,2,4$ minimally twisted field theories, also invariant under the action of $\C^\times$ by dilation on $\C^2$ and the appropriate $R$-symmetry group, is trivial in degrees $0$ and $< -1$, and isomorphic to $H^5(\g)$ in degree $1$.  

For the $\mscr{N}=1$ theory, the degree $-1$ cohomology group is also trivial.  For the $\mscr{N}=2,4$ theories, the degree $-1$ cohomology group is $\C$, which corresponds to the Lie algebra of the center of the $R$-symmetry group acting on the twisted theory. 
\end{theorem}
The proof is presented in the appendix.  

I should remark that this theorem has an immediate consequence.
\begin{corollary}
The minimally-twisted supersymmetric gauge theories on $\C^2$ all admit a unique quantization, invariant under translation, dilation, and $R$-symmetry. 
\end{corollary}
\begin{proof}
Indeed, the obstruction to quantizing lies in $H^5(\g)$.  However, the outer automorphism group of $\g$ acts on everything, and the obstruction must be invariant under this symmetry.  Since, for any semi-simple Lie algebra $\g$, there are no elements of $H^5(\g)$ invariant under $\op{Out}(\g)$, we conclude that the obstruction must vanish.  
\end{proof}
This argument was also used in \cite{Cos11}, Chapter 6 to prove the existence of a quantization of ordinary Yang-Mills theory.  

\section{Twisted theories on a complex surface}
Next, I will explain how these twisted theories can be put on an arbitrary complex surface, and not just on $\C^2$.   

The twisted $\mscr{N}=1$ theory is the cotangent theory to the moduli problem of holomorphic bundles on $\C^2$.  This theory makes sense on any complex surface.
\begin{definition}
Let $X$ be a complex surface.  Then $\mscr{N}=1$ twisted supersymmetric gauge theory on $X$ is the cotangent theory to the moduli problem of holomorphic principal $G$-bundles on $X$.
\end{definition}
Thus, if $P \to X$ is a principal $G$-bundle, the elliptic Lie algebra describing the $\mscr{N}=1$ theory is 
$$
\mscr{L}_{\mscr{N}=1}(X) = \Omega^{0,\ast}(X, \g_P) \oplus \Omega^{0,\ast}(X, \g^\vee_P \otimes K_X [-1] ) .
$$

\subsection{}
The twisted $\mscr{N}=2$ and $\mscr{N}=4$ theories are defined to be the cotangent theories to the moduli of bundles on certain graded complex manifolds extending $\C^2$.  For the $\mscr{N}=2$ theory, the graded complex manifold is $\C[-1] \to \C^2$.    For the $\mscr{N}=4$ theory, the graded complex manifold is $\C^2[1] \to \C^2$.

Thus, the $\mscr{N}=2$ theory can be defined on any graded complex manifold which locally looks like $\C[-1] \to \C^2$.   This leads to the following definition.
\begin{definition}
Let $X$ be a complex surface, and let $L \to X$ be a line bundle.  Then the $L$-twisted $\mscr{N}=2$ theory is the cotangent theory to the moduli of holomorphic principal $G$-bundles on the graded complex manifold $L[-1] \to X$.

Let $X$ be a complex surface, and let $V \to X$ be a rank two vector bundle.  Then the $V$-twisted $\mscr{N}=4$ theory is the cotangent theory to the moduli of holomorphic principal $G$-bundles on the graded complex manifold $V[1] \to X$.
\end{definition}
\begin{remark}
In the physics literature, a universal choice of such a bundle $V$ is part of what is called a ``twist''.  Indeed, physicists consider the data of a twist to include an action of $\op{Spin}(4)$ on the theory under which the chosen supercharge $Q$ is invariant. Locally, $\op{Spin}(4)$ acts on the space of fields; globally, this means that fields are sections of a vector bundle which is associated to the $\op{Spin}(4)$-frame bundle and to some representaion of $\op{Spin}(4)$.  

We are considering holomorphic twists, which can not be $\op{Spin}(4)$-invariant. However, they are $\op{SL}_2$ and even $\op{GL}_2$ invariant, and they can be made $\op{GL}_2$-invariant in more than one way. On $\C^2$, different choices of $\op{GL}_2$-action on the space of fields lead, globally, to the fields being sections of different bundles. 
\end{remark}

\subsection{}
The most important examples are when the vector bundle is naturally associated to $X$.  For the $\mscr{N}=2$ theory, the example we will be interested in is when the line bundle $L$ is the trivial line bundle. We will refer to this as ``the'' twisted $\mscr{N}=2$ gauge theory.  

In all examples we have considered so far, we have only twisted by one supersymmetry operator.  We will refer to a theory twisted in this way as a minimally twisted theory.   In many examples, however, one can perform further twists.

For the $\mscr{N}=2$ gauge theory (with trivial line bundle $L$) the theory can be twisted further to give a topological theory.  This further twist is the classical field theory related to Donaldson theory; in the same was as the fully-twisted $A$-model is related to the theory of Gromov-Witten invariants.

The elliptic Lie algebra on $X$ describing the minimally twisted $\mscr{N}=2$ theory is
$$
\Omega^{0,\ast}(X, \g_P[\eps] \oplus K_X \otimes \g_P^\vee [\eps] [-2] )
$$ 
where $\eps$ is a parameter of degree $-1$.  The pairing arises from the natural pairing between $\g_P$ and $\g_P^\vee$ and the trace map 
\begin{align*}
\op{Tr} : \C[\eps] &\to \C\\
\op{Tr}(\eps) &= 1.
\end{align*}

\subsection{} 
For the $\mscr{N}=4$ theory, there are two natural choices of rank $2$ vector bundle $V$. One is when $V$ is trivial.  This version of the twisted $\mscr{N}=4$ theory was considered by Vafa and Witten in \cite{VafWit94}; it admits a further twist into a classical topological field theory, whose partition function is supposed to be the Euler characteristic of moduli spaces of holomorphic bundles. 

The other natural choice is when $V$ is the tangent bundle.  This version of the twisted $\mscr{N}=4$ theory was considered by Kapustin and Witten \cite{KapWit06}.  This is the only version of the minimally twisted $\mscr{N}=4$ theory we will consider from now on, and we will refer to it as ``the'' twisted $\mscr{N}=4$ theory.   This theory has a very familiar geometric interpretation: it is the cotangent theory to the derived moduli space of Higgs bundles on $X$.  Recall that a Higgs bundle on $X$ is a holomorphic principal $G$-bundle $P$ with an element $\phi \in H^0 ( X, \g_P \otimes T^\ast X)$ satisfying $[\phi, \phi ] = 0$. 
 
Kapustin and Witten consider theories where one twists by several supersymmetry operators, and not just by one.  The theory described above is the minimally twisted $\mscr{N}=4$ theory, where we have twisted only by a single supersymmetry operator.  Later we will consider further twists of the minimally twisted theory, which lead to the $\P^1$ of topological theories considered by Kapustin and Witten.  

The elliptic Lie algebra describing this twisted $\mscr{N}=4$ theory is 
$$
\mscr{L}_{\mscr{N}=4}(X) = \Omega^{\ast,\ast}(X, \g_P \oplus \g^\vee_P[1] ),
$$
where the differential is the $\dbar$ operator, and of course, $\Omega^{p,q}(X)$ is situated in degree $p+q$.

\section{The Kapustin-Witten family of twisted $N=4$ theories} 

So far we have constructed the minimal twists of the $\mscr{N}=1,2,4$ supersymmetric gauge theories, on a complex surface.  In this section I will show how the twisted $\mscr{N}=4$ theory we have constructed can be twisted further, to yield a $\mbb{P}^1$ of (classical) topological field theories.  

Let $G$ be a simple algebraic group, and let $P$ be a $G$ local system on $X$.  To describe the $\P^1$ of twisted theories where we perturb around $P$, I will just write down a $\P^1$ of elliptic Lie algebras on $X$. 

Let $(s,t) \in \C^2$.  Let us define an elliptic Lie algebra 
$$
\mscr{L}(s,t) = \left(   \Omega^{\ast,\ast}( X, \g_P \otimes \C[\eps ] \right)
$$
where $\eps$ is a parameter of degree $-1$.  The differential is
$$
\dbar + s \partial + t \dpa{\eps}.
$$
The Lie bracket on $\mscr{L}(s,t)$ is independent of $s$ and $t$. 

The elliptic Lie algebra $\mscr{L}(s,t)$ has an invariant pairing given by the formula
$$
\ip{\eps \alpha, \beta } = \int_X \ip{\alpha,\beta}_{\g}.
$$
Here $\alpha,\beta \in \Omega^{\ast,\ast}(X,\g_P)$, and $\ip{-,-}_\g$ is a chosen invariant pairing on the Lie algebra $\g$ of $G$ (which, since $G$ is simple, is unique up to scale).

When $s = t = 0$, this family of elliptic Lie algebras coincides with that describing the minimally twisted theory we considered earlier.

The elliptic Lie algebra $\mscr{L}(0,0)$ has a $\C^\times$ action, where an element
$$
\alpha \in \eps^r\Omega^{p,q} (X, \g_P) 
$$
has weight $p - 4 r$.  This $\C^\times$ action is easily seen to preserve the pairing.   The action therefore gives an isomorphism of classical field theories
$$
\mscr{L}(s,t) \iso \mscr {L}( \lambda s, \lambda^{4} t ) .
$$
Thus, the family of twisted theories is parametrized by a weighted $\P^1$. 

\subsection{}
Note that, when $t = 0$, the theory $\mscr{L}(1,0)$ is the cotangent theory to elliptic moduli problem of $G$ local systems on $X$.      As we will see shortly, this theory becomes, on dimensional reduction, the $B$-model with target the space of $G$-local systems on a curve.  We will this value of the parameter the \emph{$B$-model point}.

For $s = 1$ and $t \in \C$, the theory $\mscr{L}(1,t)$ is a ``twisted'' form of the cotangent theory to the moduli of $G$ local systems on $X$.   Let $\op{Loc}_G(X)$ denote the derived moduli space of $G$-local systems on $X$.  As always, I will only be precise at the formal level: the elliptic Lie algebra describing the formal neighbourhood of a $G$-local system $P$ is $\Omega^\ast(X, \g_P)$, with the de Rham differential. 

Note that the Poincar\'e pairing gives this elliptic Lie algebra a pairing of cohomological degree $-4$.  It follows that the derived moduli space $\op{Loc}_G(X)$ has a symplectic pairing of cohomological degree $-2$.    This symplectic pairing can be interpreted as a Poisson bracket of cohomological degree $2$.  The Poisson bivector $P$ is then a cohomological degree $0$ function on $T^\ast[-1] \op{Loc}_G(X)$, which is quadratic along the cotangent fibres.   The Jacobi identity implies that $\{P,P\}= 0$, where $\{-,-\}$ refers to the Poisson bracket on functions on $T^\ast[-1] \op{Loc}_G(X)$.  

Thus, we can deform the $0$-symplectic manifold $T^\ast[-1] \op{Loc}_G(X)$ by adding $t \{P,-\}$ to the differential.  This one-parameter family of $0$-symplectic manifolds (for $t \in \mbb{A}^1$) describes the Kapustin-Witten family of theories at the points $(1:t)$.  

Note that we can use the symplectic form on $\op{Loc}_G(X)$ to identify $T^\ast[-1] \op{Loc}_G(X)$ with $T[1] \op{Loc}_G(X)$.  Functions on $T[1] \op{Loc}_G(X)$ are forms on $\op{Loc}_G(X)$.  Under this identification, the operation $\{P,-\}$ (on functions on $T^\ast[-1] \op{Loc}_G(X)$)  becomes the de Rham differential.  Thus, we can think of the $0$-symplectic manifold describing the Kapustin-Witten theory at a point $(1:t)$ with $t \neq 0$ as being equivalent to the de Rham stack of $\op{Loc}_G(X)$.   With this identification, the symplectic form depends on $t$.  

\subsection{}
Finally, let us discuss the theory when $s = 0$.  The theory $\mscr{L}(0,1)$ described by elliptic Lie algebra
$$
\Omega^{\ast,\ast} ( X, \g_P [\eps ] )
$$
with differential $\dbar + \frac{\d}{\d \eps}$.  This elliptic Lie algebra is contractible; just like the elliptic $L_\infty$ algebra describing the fully-twisted $A$-model.   Thus, perturbation theory does not say anything about the theory with this parameter.  When we dimensionally reduce, this theory becomes the $A$-model with target the stack of Higgs bundles on a curve.
 
\section{Dimensional reduction}
In this section I will introduce the general idea of dimensional reduction.   Shortly we will apply this idea to relate the twisted $4$-dimensional gauge theories we have been studying to the $2$-dimensional field theories we discussed earlier: the various twists of the $A$- and $B$-models. 

Because of lack of space, I will be a little informal in the general discussion of dimensional reduction.  As I mentioned in the introduction, I will not attempt to give detailed definitions of global objects of derived algebraic geometry. I will, however, try to be more precise at the level of formal derived spaces. 

\subsection{}
The basic idea of dimensional reduction is very simple.  We have defined a (perturbative) field theory on a space $X$ to be a sheaf of (formal) derived spaces on $X$, together with a symplectic form.   If $f : X \to Y$ is a fibration, and $\mc M$ is a sheaf of formal derived spaces on $X$, then we can define a push forward sheaf $f_\ast \mc M$.    If $\mc M$ is a classical field theory -- that is, equipped with a symplectic form of degree $-1$ -- then so is $f_\ast \mc M$.  We call $f_\ast \mc M$ the \emph{dimensional reduction} of the field theory $\mc M$ on $X$. 

Let us consider a simple example.  Let $M$ and $N$ be complex manifolds.  Let $G$ be an algebraic group, and let $\op{Bun}_G (N)$ denote the (derived) moduli stack of $G$-bundles on $N$.  

Then, the derived moduli stack of holomorphic maps $M \to \op{Bun}_G(N)$ is the same as the derived moduli stack of holomorphic $G$-bundles on $M \times N$. 

Thus, we see that an elliptic moduli problem on $M \times N$ (that describing holomorphic $G$-bundles)  can be turned into an elliptic moduli problem on $M$ (that describing maps from $M$ to $\op{Bun}_G(N)$).

\subsection{}
Let us now give a formal definition of dimensional reduction.  We will work at the perturbative level, where a classical field theory is described by an elliptic $L_\infty$ algebra with an invariant pairing. 

Let $\pi : F \to M$ be a proper fibration of manifolds (so that $\pi$ has compact fibres).  Let $\L$ be an elliptic $L_\infty$ algebra on $F$. Let $L$ be the underlying graded vector bundle of $\L$. 

We would like to define the elliptic $L_\infty$ algebra on $M$ to be the sheaf-theoretic pushforward $\pi_\ast \L$.  This, however, does not obey the axioms I gave for an elliptic $L_\infty$ algebra, because $\pi_\ast \L$ does not arise as the sections of a finite-dimensional graded vector bundle on $M$.  

We can, instead, look for an elliptic $L_\infty$ algebra on $M$ which is quasi-isomorphic to $\pi_\ast \L$.  This gives a precise definition of a dimensional reduction of a formal elliptic moduli problem.
\begin{definition}
Let $\L$ be an elliptic $L_\infty$ algebra on $F$.  Then an elliptic $L_\infty$ algebra $\til{\L}$ on $M$ is a dimensional reduction of $\L$ if we are given a quasi-isomorphism of sheaves of $L_\infty$ algebras
$$
\til{\L} \simeq \pi_\ast \L.
$$
\end{definition}
Let us see how this works in case when $M$ and $N$ are Riemann surfaces, and the elliptic moduli problem we are considering is that of holomorphic $G$-bundles on $M \times N$.  

Let $P \to M \times N$ be such a $G$-bundle.  The elliptic Lie algebra controlling deformations of $P$ is $\Omega^{0,\ast}(M \times N, \g_P)$, where $\g_P$ denotes the adjoint bundle of Lie algebras on $M \times N$ associated to $P$. 

Dimensional reduction, in this case, means that we consider $P$ to be a map from $M$ to the moduli stack $\op{Bun}_G(N)$ of holomorphic $G$-bundles on $N$.  Note that because $N$ is a Riemann surface, $\op{Bun}_G(N)$ can be treated as an ordinary (non-derived) stack.  

Let $\phi : M \to \op{Bun}_G(N)$.   Let $T\op{Bun}_G(N)$ denote the tangent complex of $\op{Bun}_G(N)$.    There is an $L_\infty$ structure on $\Omega^{0,\ast}( M, \phi^\ast T \op{Bun}_G(N)) $ which controls deformations of the map $\phi$. 

In this example, the statement that the elliptic moduli problem on $M$ is dimensionally reduced from that on $M \times N$ means that there is a canonical equivalence between two sheaves of $L_\infty$ algebras on $M$.  The first sheaf of $L_\infty$ algebras sends $U \subset M$ to
$$
\g_1(U) = \Omega^{0,\ast}(U , \phi^\ast T \op{Bun}_G(N ) ).
$$
The second sheaf of $L_\infty$ algebras sends $U \subset M$ to 
$$
\g_2(U) = \Omega^{0,\ast}(U \times N, \g_P) 
$$
where $g_P$ is the principal $G$-bundle on $M \times N$ arising from the map $M \to \op{Bun}_G(N)$.  

The existence of such an equivalence of sheaves of $L_\infty$ algebras is automatic from the universal property of $\op{Bun}_G(N)$.   

In practise, however, in this and in other examples, there is no need to replace the sheaf $\g_2(U)$ of $L_\infty$ algebras on $M$ by a smaller sheaf.  The sheaf $\g_2(U)$ does not strictly conform to the definition of an elliptic $L_\infty$ algebra I gave earlier: it does not arise as the sections of a finite rank graded vector bundle on $M$.  However, there are no essential difficulties caused by working directly with a sheaf of $L_\infty$ algebras of the form $\g_2$.  
  
\subsection{}
The factorization algebra point of view \cite{CosGwi11} on perturbative quantum field theory gives a clean way to think about dimensional reduction.  Let $\pi : F \to M$ be a proper fibration of manifolds.  Let $\mscr{F}$ be a factorization algebra on $F$, in the sense of \cite{CosGwi11}.   If we are dealing with a classical field theory on $F$, then $\mscr{F}$ will be a commutative factorization algebra with a Poisson bracket of degree $1$.  If we are dealing with a quantum field theory, then $F$ will be a factorization algebra over $\R[[\hbar]]$. 

In either case, we can define a factorization algebra $\pi_\ast F$ on $M$ be setting
$$
(\pi_\ast F) (U) = F(\pi^{-1} (U))
$$
for an open subset $U \subset M$.    This pushforward factorization algebra describes the observables of the dimensionally reduced theory. 

\subsection{}
Many of the theories we have considered in this paper are the cotangent theories associated to elliptic moduli problems.  It is straightforward to verify, from the definitions given above, that dimensional reduction commutes with the operation of taking the cotangent theory associated to an elliptic moduli problem.  

\section{From $4$-dimensional gauge theories to $2$-dimensional $\sigma$-models}

In this section we will see how dimensional reduction of the various twisted $4$-dimensional gauge theories we have considered lead to $2$-dimensional $\sigma$-models with target various versions of the moduli stack of $G$-bundles on a Riemann surface.

\subsection{}
Let us start with the twisted $\mscr{N}=1$ gauge theory on a product $\Sigma_1 \times \Sigma_2$ of two Riemann surfaces.    The $4$-dimensional theory is the cotangent theory to the moduli of holomorphic $G$-bundles on $\Sigma_1 \times \Sigma_2$.  It follows that the $2$-dimensional theory (dimensionally reduced along $\Sigma_2$) is the cotangent theory to the moduli of holomorphic maps from $\Sigma_1$ to $\op{Bun}_G(\Sigma_2)$.  

Recall that, given any complex manifold $X$, the cotangent theory to the moduli of holomorphic maps from a Riemann surface $\Sigma$ to $X$ is known in the physics literature as a twisted $(0,2)$ $\sigma$-model.    It is believed \cite{Wit05,Cos10a} that the factorization algebra of quantum observables of this theory -- or at least, that part of the factorization algebra which only considers constant holomorphic maps to $X$ -- is the chiral differential operators of $X$.  

Thus, one expects that factorization algebra constructed from the twisted $\mscr{N}=1$ theory should be closely related to the chiral differential operators of $\op{Bun}_G(\Sigma_2)$; and that the partition function of this theory contains the Witten elliptic genus of the moduli stack $\op{Bun}_G(\Sigma_2)$. 

In \cite{Cos13}, the twisted $\mscr{N}=1$ theory is constructed on all complex surfaces $M$ with $c_1(M) = 0$. This paper analyzes the factorization algebra associated to a twist of the deformed $\mscr{N}=1$ theory, and shows that it is related to the Yangian. 

\subsection{}
Next, let us consider the minimally twisted $\mscr{N}=2$ theory.  Recall that this is the cotangent theory to the derived moduli space of $G$-bundles on a complex surface $M$, together with a holomorphic section of the adjoint bundle $\g_P$. 

Upon dimensional reduction along a curve $\Sigma_2$, this becomes the cotangent theory to the space of holomorphic maps from $\Sigma_1$ to the derived moduli stack of pairs 
$$
\{ (P,\phi) \mid P \in \op{Bun}_G(\Sigma_2), \ \phi \in H^0(\Sigma, \g_P) \} .
$$
(When I say the derived moduli stack of such pairs, it is implicitly assumed that the higher cohomology $H^i(\Sigma, \g_P)$ is included as part of the derived structure).  

Note that we can identify the tangent complex $T_P\op{Bun}_G(\Sigma_2)$ as
$$
T_P\op{Bun}_G(\Sigma_2) = H^\ast( \Sigma, \g_P)[1]. 
$$
Thus, the twisted $\mscr{N}=2$ gauge theory becomes, upon dimensional reduction, the cotangent theory to the space of holomorphic maps from $\Sigma_1$ to $T[-1] \op{Bun}_G(\Sigma_2)$. 

This agrees, as a $\Z/2$ graded theory, with the $\tfrac{1}{2}$-twisted $A$-model on $\op{Bun}_G$.    As I discussed earlier, the $\Z$ grading we gave to supersymmetric field theories is a little arbitrary;  thus, we can say that, after changing the grading on our minimally twisted $\mscr{N}=2$ theory, it dimensionally reduces to the $\tfrac{1}{2}$-twisted $A$-model on $\op{Bun}_G$. 

It is not difficult to see that the minimally twisted $\mscr{N}=2$ theory can be further twisted into a theory which dimensionally reduces to the fully-twisted $A$-model.     This further twist of the $\mscr{N}=2$ theory is the one considered by Witten \cite{Wit88a} in his study of Donaldson theory.

\section{Dimensional reduction of the $\mscr{N}=4$ theory}

In this section, we will show the following.
\begin{proposition}
The dimensional reduction of the minimally-twisted Kapustin-Witten theory on a product of two Riemann surfaces $\Sigma_1 \times \Sigma_2$ is the $\tfrac{1}{2}$-twisted $B$-model with target $T^\ast \op{Bun}_G( \Sigma_2 )$.
\end{proposition}
Since $T^\ast \op{Bun}_G(\Sigma_2)$ has a (holomorphic) symplectic form, the $\tfrac{1}{2}$-twisted $A$- and $B$-models with this target coincide. 
\begin{proof}
The minimally-twisted Kapustin-Witten theory is the cotangent theory to the derived moduli space of $G$-bundles on $T[1] (\Sigma_1 \times \Sigma_2 )$. The elliptic Lie algebra on $\Sigma_1 \times \Sigma_2$ describing this derived moduli space (near a given principal $G$-bundle) is
$$
\Omega^{\ast,\ast}(\Sigma_1 \times \Sigma_2, \g_P)
$$
with differential $\dbar$. 

Note that we can write
$$
T[1] (\Sigma_1 \times \Sigma_2) = (T[1] \Sigma_1) \times (T[1] \Sigma_2).
$$
Thus, when we dimensionally reduce along $\Sigma_2$, we find the cotangent theory to the elliptic moduli problem describing holomorphic maps
$$
T[1] \Sigma_1 \to \op{Bun}_G (T[1] \Sigma_2).
$$
For a general complex target $X$, the cotangent theory to the space of holomorphic maps $T[1] \Sigma \to X$ is the $\tfrac{1}{2}$-twisted $B$-model on $X$.  

Thus, it remains to verify that, for any Riemann surface $\Sigma$, 
$$
\op{Bun}_G(T[1] \Sigma) = T^\ast \op{Bun}_G(\Sigma).
$$
Note that $\op{Bun}_G(T[1] \Sigma)$ is, by definition, the derived moduli space of pairs $(P,\phi)$, where $P$ is a principal $G$-bundle on $\Sigma$ and $\phi$ is a section of $K_\Sigma \otimes \g_P$.  In other words, $\op{Bun}_G(T[1] \Sigma)$ is the derived moduli space of Higgs bundles on $\Sigma$.  It is well known that this moduli space describes the cotangent bundle to $\op{Bun}_G(\Sigma)$. 
\end{proof}

In a similar way, we see the following.
\begin{lemma}
The dimensional reduction of the fully twisted $\mscr{N}=4$ theory at the point $(1,0)$ in the $\mbb{P}^1$ of twists is the $B$-fully-twisted model with target $\op{Loc}_G(\Sigma_2)$. 

The dimensional reduction of the fully-twisted $\mscr{N}=4$ theory at the point $(0,1)$ is the fully-twisted $A$-model with target $T^\ast \op{Bun}_G(\Sigma_2)$.
\end{lemma}
\begin{proof}
Let us first prove the $B$-model statement.  At the point $(1,0)$ the fully-twisted $\mscr{N}=4$ theory is the cotangent theory to the moduli space of $G$-local systems on $\Sigma_1 \times \Sigma_2$.  When we dimensionally reduce, we find the cotangent theory to the space of locally constant maps $\Sigma_1 \to \op{Bun}_G(\Sigma_2)$, which is the $B$-model with target $\op{Bun}_G(\Sigma_2)$. 

Recall that the minimally-twisted $\mscr{N}=4$ theory becomes, upon dimensional reduction, the cotangent theory to the moduli of Higgs bundles on $\Sigma_1 \times \Sigma_2$.  We can also view this as the cotangent theory to the space of maps from $T[1] \Sigma_1 \to T^\ast \op{Bun}_G(\Sigma_2)$.  The $A$-model with target $X$ is a deformation of the cotangent theory of holomorphic maps from $T[1] \Sigma \to X$, deformed by introducing the de Rham differential on $T[1] \Sigma$.  It remains to verify that the deformation of the minimally-twisted $\mscr{N}=4$ theory to the twisted theory with parameter $(0,1)$ amounts to introducing the de Rham differential on $T[1] \Sigma_1$; this is straightforward.

\end{proof}
\section*{Appendix}
In this appendix I present a proof of a cohomology vanishing result, which allowed us to conclude that the twist of the full supersymmetric gauge theory coincides with the twist of the anti-self-dual theory.  This result also shows that our minimally twisted theories admit a unique quantization on $\C^2$, invariant under translation, dilation, and $R$-symmetry.  This result is related to a theorem proved in \cite{Cos11}, where I showed by a similar cohomological analysis that ordinary Yang-Mills theory can be quantized on $\R^4$.

\begin{theorem}
Let $\g$ be a simple Lie algebra.  The cohomology of the complex of translation-invariant local functionals for the $\mscr{N}=1,2,4$ minimally twisted field theories, with gauge Lie algebra $\g$, which are also invariant under the action of $\C^\times$ by dilation on $\C^2$ and the appropriate $R$-symmetry group, is trivial in degrees $0$ and $\le -2$.  In degree $1$ it coincides with $H^5(\g)$.  

For the $\mscr{N}=1$ theory, the degree $-1$ cohomology group is also trivial.  For the $\mscr{N}=2,4$ theories, the degree $-1$ cohomology group is $\C$, which corresponds to the Lie algebra of the center of the $R$-symmetry group acting on the twisted theory. 
\end{theorem}

\begin{proof}
General results from \cite{Cos11}, Chapter 5, allow one to identify the groups of translation-invariant local functionals with certain Lie algebra cohomology groups, as follows.   For each $k \ge 0$, consider the Lie algebra 
\begin{multline*}
\mscr{G}'_k = \g[[z_1,z_2,\br{z}_1, \br{z}_2, \d \br{z}_1, \d \br{z}_2, \eps_1,\dots,\eps_k]]  \\ \oplus \g[[z_1,z_2 \br{z}_1, \br{z}_2, \d \br{z}_1, \d \br{z}_2,\eps_1,\dots,\eps_k]] \d z_1 \d z_2 (\d \eps_1)^{-1} \dots (\d \eps_k)^{-n} [-1].
\end{multline*}
Here $\g$ is a fixed semi-simple Lie algebra; the parameters $\eps_i$ and $\d \zbar_i$ has cohomological degree $1$.   The differential on $\mscr{G}'$ is induced by the usual Dolbeaut operator, $\sum \d \zbar_i \frac{\d}{\d \zbar_i}$. The factor of $\d z_1 \d z_2 \prod( \d \eps_i)^{-1}$ is indicated to show how symmetries of $\C^2$, and the $R$-symmetry group, act on everything.

Note that $\mscr{G}'_k$ is acted on by the Abelian Lie algebra $\C^4$, with basis $\partial_i, \br{\partial}_j$, where $\partial_i$ acts by $\frac{\d}{\d z_i}$ and $\br{\partial}_j$ acts by $\frac{\d}{\d \zbar_i}$.

General results from \cite{Cos11} imply that the complex of translation-invariant deformations of our minimally twisted $\mscr{N}=1$, $\mscr{N}=2$ or $\mscr{N}=4$ theory is given by a Lie algebra cohomology group of the form
$$
C_\ast( \C^4,  C^\ast_{red}( \mscr{G}'_k ) \d z_1 \d z_2 \d \zbar_1 \d \zbar_2 ) . 
$$
for different values of $k$: $k = 0$ corresponds to $\mscr{N}=1$, $k = 1$ to $\mscr{N}=2$ and $k = 2$ to $\mscr{N}=4$. 

This expression indicates the Lie algebra chains of the Abelian Lie algebra $\C^4$ with coefficients in the reduced Lie algebra cochains of $\mscr{G}'_k$.  

Our aim is to show that there is nothing in $H^0$ of this complex which is invariant under the group $\C^\times \times \GL_k(\C)$, which acts as follows.  The factor of $\C^\times$ is the dilation symmetry of $\C^2$, and so acts on $z_i, \zbar_i, \d z_i, \d \zbar_i$ in the evident way.  The factor $GL_k(\C)$ is the $R$-symmetry group, which acts on the exterior algebra $\C[\eps_1,\dots,\eps_k]$ in the evident way.  

The $\C^\times$ action on the $\eps_i$ is taken to be trivial.   This is just a convention: if $\eps_i$ had weight $\lambda \in \Z$ under the action of $\C^\times$, then by conjugating by an isomorphism of  the group $\C^\times \times GL_2(\C)$ we could return to a situation where $\eps_i$ was preserved by the group $\C^\times$.  

Now, let $\mscr{G}_k \subset \mscr{G}'_k$ be the sub Lie algebra which has no $\zbar$'s and no $\d \zbar$'s.  Note that the inclusion $\mscr{G}_k \subset \mscr{G}'_k$ is a quasi-isomorphism.  It follows that we can compute the cohomology groups of interest using $\mscr{G}_k$ in place of $\mscr{G}'_k$.  

Note that the elements $\dbar_i$ in the Abelian Lie algebra $\C^4$ act trivially on $\mscr{G}_k$.  We can rewrite our complex as 
$$
C_\ast(\C^2, C^\ast_{red}(\mscr{G}_k ) \d z_1 \d z_2 ) \otimes \left( \C[\dbar_1, \dbar_2] \d \zbar_1 \d \zbar_2   \right)  
$$
where in the algebra $\C[\dbar_1, \dbar_2]$ the generators $\dbar_1, \dbar_2$ are given degree $-1$. 

We are only interested in quantities which are invariant under the action of $\C^\times \times GL_2(\C)$.  The only $\C^\times$ invariant quantities contain the same number of $\dbar_i$'s and $\d \zbar_j$'s, so we see that our complex reduces to
$$
 C_\ast\left(\C^2, C^\ast_{red}(\mscr{G}_k ) \d z_1 \d z_2  \right)^{\C^\times \times GL_2(\C)} [2].
$$ 
Let 
$$C^\ast_{red}(\mscr{G}_k)^i \subset C^\ast_{red}(\mscr{G}_k)$$ be the subcomplex consisting of elements of weight $-i$ under the $\C^\times$ action (with respect to which $z_i$ has weight $1$ and $\eps_j$ has weight $0$).  This complex has a natural $GL_k(\C)$-action; the invariants under $GL_k(\C)$ will be denoted by $C^\ast_{red}(\mscr{G}_k)^{i, GL_k(\C)}$.  

We will compute the $\C^\times\times GL_k(\C)$-invariant cohomology group we want by a spectral sequence.  The first term of our spectral sequence computes the cohomology with respect to the internal differential on $C^\ast_{red}(\mscr{G}_k)$; the next term uses the action of the Abelian Lie algebra $\C^2$. 

The $\C^\times$-invariant part of the first term of our spectral sequence is the the direct sum of the following complexes:
\begin{align*}
C^\ast_{red}(\mscr{G}_k)^{0} \partial_1 \partial_2 \d z_1 \d z_2 [4] \\
C^\ast_{red}(\mscr{G}_k)^{-1} \partial_1 \d z_1 \d z_2[3] \oplus C^\ast_{red}(\mscr{G}_k)^{-1} \partial_2 \d z_1 \d z_2 [3] \\
C^\ast_{red}( \mscr{G}_k)^{-2}  \d z_1 \d z_2 [2].
\end{align*} 
Now, recall that
$$
\mscr{G}_k = \g[[z_1,z_2,\eps_1,\dots,\eps_k]] \oplus \g[[z_1,z_2,\eps_1,\dots,\eps_k]] \d z_1 \d z_2 (\d \eps_1)^{-1} \dots (\d \eps_k)^{-1} [-1]. 
$$
Let us denote by $\g_k$ the Lie algebra
$$
\g_k =  \g[\eps_1,\dots,\eps_k].
$$

Thus, the complex $C^\ast_{red}(\mscr{G}_k)^{0}$ is just $C^\ast_{red}(\g_k)$.   The $GL_k(\C)$-invariants of $C^\ast_{red}(\g_k)$ is $C^\ast_{red}(\g)$.   So, we find $H^\ast_{red}(\g)[4]$ as one summand of the $\C^\times \times GL_k(\C)$-invariant part of the first page of our spectral sequence.

Next, note that $C^\ast_{red}(\mscr{G}_k)^{-1}$ consists of $C^\ast( \g_k, (z_i \g_k)^\vee)$, for $i = 1,2$.   If we further restrict to the $GL_k(\C)$ invariants, we find $C^\ast(\g, (z_i \g)^\vee)$.   Since $\g$ is semi-simple, $H^\ast(\g,\g^\vee) = 0$, so that $GL_k(\C)$-invariant part of the cohomology of $C^\ast_{red}(\mscr{G}_k)^{1}$ vanishes. 

Finally, let us consider $C^\ast_{red}(\mscr{G}_k)^{-2}$.   One possible source of weight $-2$ cochains is $C^\ast(\g_k, (z_i\g_k)^\vee \otimes (z_j \g_k)^\vee$.  If we restrict again to $GL_k(\C)$ invariants, we find $C^\ast(\g_k, (z_i\g_k)^\vee \otimes (z_j \g_k)^\vee$.  Since $H^\ast(\g, \wedge^2 \g^\vee) = 0$, the only non-zero possibility is when $i = 1$ and $j = 2$, giving us
$$
H^\ast(\g,  \Sym^2 \g^\vee (z_1^\vee z_2^\vee )  ) 
$$
as a second direct summand of the $GL_k(\C) \times \C^\times$ invariant part of our spectral sequence.  

A second possibility for weight $-2$ cochains is 
$$
C^\ast(\g_k, (\d z_1 \d z_2 \d \eps_1^{-1} \dots \d \eps_k^{-1} \g_k )^\vee )[2].
$$
Here, we take $\d \eps_1^{-1} \dots \d \eps_k^{-1}$ to have cohomological degree $-k$ if the $\eps_i$ have degree $1$, and $+k$ if the $\eps_i$ have degree $-1$.  Let $\C_{\det}$ refer to $\C$ viewed as the determinant representation of $GL_k(\C)$; then this term in the spectral sequence can be written as $C^\ast(\g_k, \g_k^\vee \otimes \C_{\det} [2 \pm k])$.  

Note that the invariant pairing on the Lie algebra $\g$, as well as the natural Frobenius algebra structure on $\C[\eps_1,\dots,\eps_k]$, gives an isomorphism 
$$\g_k^\vee \otimes \C_{\det} [\pm k] \iso \g_k.$$
Thus, this term in the spectral sequence is simply $C^\ast(\g_k, \g_k )[2]$. 

To summarize: we have shown that the $GL_k(\C) \times \C^\times$-invariant part of the $E_1$ page of our spectral sequence consists of 
$$
H^\ast_{red}(\g) [4] \oplus H^\ast(\g, \Sym^2 \g^\vee ) \oplus H^\ast(\g_k, \g_k  )^{GL_k(\C)}[2].
$$
Note that $H^\ast_{red}(\g) [4]$ consists of $H^3(\g)$ in degree $-1$ and $H^5(\g)$ in degree $+1$; these are the only degrees of interest to us.  Also, $H^\ast(\g, \Sym^2 \g^\vee)$ consists of the $\g$-invariants in $\Sym^2 \g^\vee$ in degree $0$, and is zero in all other degrees which are $\le 1$. 

It is not completely obvious what $H^\ast(\g_k, \g_k)$ is, but we will compute it shortly.  

To complete the proof, we need to prove two lemmas.
\begin{lemma}
On the next page of the spectral sequence, the differential gives an isomorphism between $H^3(\g)$ (situated in degree $-1$) and $H^0(\g \Sym^2 \g^\vee)$ (situated in degree $0$).
\end{lemma}
\begin{proof}
If we reintroduce the symbols $z, \d z$, we see that our possible differential maps 
$$H^3(\g) \partial_1 \partial_2 \to H^0( \g, \Sym^2 \g^\vee z_1^\vee z_2^\vee ) .$$ 
A straightforward computation (carried out in a very similar context in \cite{Cos11}, Chapter 5) shows that this map is an isomorphism. 
\end{proof}

\begin{lemma}
The cohomology $H^\ast(\g_k, \g_k )^{GL_k(\C)}[2]$ is the following:
\begin{enumerate}
\item $0$ if $k = 0$.
\item $\C$ in degree $-1$ if $k = 1$ or $k = 2$.
\end{enumerate} 
\end{lemma}
\begin{proof}
If $k = 0$, then $\g_k = \g$, and $H^\ast(\g,g) = 0$.  

If $k = 1$, then $\g_k = \g[\eps]$.  This corresponds to the theory with $\mscr{N}=2$ supersymmetry; in which case we put $\eps$ in degree $-1$. 

The only $\C^\times$-invariant part of $H^\ast(\g[\eps], \g[\eps] ) [2]$ comes from 
$$H^\ast(\g, (\eps \g)^\vee \otimes \eps \g)[2] = H^\ast(\g, \Sym^2 \g) [1]  .$$  
Thus, we find $H^0(\g, \Sym^2 \g)$ in degree $-1$, and $0$ in degrees $0,1$. 

Next, let us consider the case $k = 2$, which corresponds to the theory with $\mscr{N}=4$ supersymmetry.  In this case, we take the $\eps_i$ to have cohomological degree $1$.  ($GL_2(\C)$-invariance implies that the cohomology groups we end up with are independent of the degree we choose to assign to the $\eps_i$, as long as that degree is odd.  As, $GL_2(\C)$-invariance implies that the number of $\eps_i$'s and $\eps_i^\vee$'s is the same). 

Let us give $\g[\eps_1,\eps_2]$ a grading by giving $\g$ has weight $0$, and the subspace $(\eps_1 \oplus \eps_2)\g[\eps_1,\eps_2]$  weight $-1$. This induces a grading on $C^\ast(\g_2, \g_2)$, not compatible with the differential.  Define $F^k C^\ast (\g_2, \g_2)$ to be the subcomplex of elements of weight $\ge k$ in this grading.   The differential on $C^\ast(\g_2, \g_2)$ preserves these subspaces, so that we have defined a filtration on the complex $C^\ast(\g_2, \g_2)$.  

This filtration induces a spectral  sequence, whose first term is the cohomology of the associated graded.   We will compute the $GL_2(\C)$ invariants of this cohomology.   

Let us denote by $V$ the $2$-dimensional vector space spanned by $\eps_1,\eps_2$, situated in degree $1$.  Thus, our Lie algebra is $\g \otimes \Sym^\ast V$. 

The $GL(V)$-invariants of $\Gr C^\ast(\g_2, \g_2 )[2]$ breaks up as a direct sum of the following four subspaces: 
\begin{enumerate}
\item $C^\ast(\g, \g)$.  The cohomology of this summand is of course zero.
\item
$C^\ast(\g, ( V \otimes \g)^\vee  \otimes (V\otimes \g))[1]$. The $GL(V)$-invariants of $V^\vee \otimes V$ are one dimensional.  The $\g$-invariants of $\g^\vee \otimes \g$ coincide with the $\g$-invariants of $\Sym^2 \g$.  Thus, this summand contributes $H^\ast(\g, \Sym^2 \g)[1]$.  The only part that can contribute to the cohomology groups of interest is $H^0(\g,\Sym^2 \g)$ in degree $-1$.
\item $C^\ast(\g, (\wedge^2 V \otimes \g)^\vee \otimes (\wedge^2 V \otimes g) )[1]$.  Again, the $GL(V)$-invariants gives us $H^0(\g, \Sym^2\g)$ in degree $-1$.   
\item $C^\ast(\g,  \Sym^2 (V \otimes \g)^\vee \otimes  (\wedge^2 V \otimes\g) )$.  The $GL(V)$-invariants of this is one dimensional. Thus, we find $H^0(\g, \wedge^2 \g^\vee \otimes \g)$ in degree $0$.  Since $\g$ is semi-simple, this coincides with $H^0(\g, \wedge^3 \g)$ in degree $0$.  
\end{enumerate} 
The differential on the next page of the spectral  sequence maps the third summand listed above to the fourth.  Thus, we find that the cohomology reduces to $H^0(\g, \Sym^2 \g)$ in degree $-1$. 
\end{proof}

This completes the proof of the theorem. 
\end{proof}


\def\cprime{$'$}

\end{document}